\documentclass[reqno,11pt,T1]{amsproc}
\usepackage{amssymb}
\usepackage{amsmath}
\usepackage{amsthm}
\usepackage{amsfonts}
\usepackage[canadian]{babel}
\usepackage{xcolor}
\usepackage{geometry}
\usepackage{enumitem}
\usepackage{microtype}
\usepackage[T1,T2A]{fontenc}

\usepackage{tikz}
\usetikzlibrary{arrows,cd,decorations.markings,shapes.geometric,shapes}

\definecolor{myurlcolor}{rgb}{0,0,0.4}
\definecolor{mycitecolor}{rgb}{0,0.5,0}
\definecolor{myrefcolor}{rgb}{0.5,0,0}
\usepackage[pagebackref,colorlinks,
linkcolor=myrefcolor,
citecolor=mycitecolor,
urlcolor=myurlcolor]{hyperref}

\PassOptionsToPackage{pagebackref,colorlinks,linkcolor=myrefcolor,citecolor=mycitecolor,urlcolor=myurlcolor}{hyperref}
\usepackage[capitalize]{cleveref}
\crefname{equation}{}{}
\crefname{pluralequation}{}{}

\newcommand{\beq}{\begin{equation}}
\newcommand{\eeq}{\end{equation}}
\newcommand{\Z}{\mathbb{Z}}
\newcommand{\N}{\mathbb{N}}

\newcommand{\R}{\mathbb{R}}
\newcommand{\op}{\mathrm{op}}
\newcommand{\eps}{\varepsilon}

\newcommand{\im}[1]{\mathrm{im}(#1)}
\renewcommand{\C}{\mathsf{C}}
\newcommand{\D}{\mathsf{D}}
\renewcommand{\det}{{\mathrm{det}}}
\newcommand{\Markov}{\mathsf{Markov}}
\newcommand{\dom}{\mathrm{dom}}
\newcommand{\condindproc}[3]{{#1} \perp {#2} \,||\, {#3}}
\newcommand{\condindstate}[3]{{#1} \perp {#2} \,|\, {#3}}
\newcommand{\condindgen}[4]{{#1} \perp {#2} \,|\, {#3} \,||\, {#4}}
\newcommand{\condindmarkov}[3]{{#1} \perp\!\!\!\!\,\perp {#2} \,|\, {#3}}
\newcommand{\as}{\textrm{-}\mathrm{a.s.}}	

\newcommand{\Set}{\mathsf{Set}}
\newcommand{\id}{\mathrm{id}}
\newcommand{\FinStoch}{\mathsf{FinStoch}}
\newcommand{\FinSetMulti}{\mathsf{FinSetMulti}}
\newcommand{\Stoch}{\mathsf{Stoch}}
\newcommand{\BorelStoch}{\mathsf{BorelStoch}}
\newcommand{\FinSet}{\mathsf{FinSet}}
\newcommand{\FinRel}{\mathsf{FinRel}}
\newcommand{\Gauss}{\mathsf{Gauss}}

\newcommand{\Kl}[1]{\mathsf{Kl}(#1)}	
\newcommand{\Meas}{\mathsf{Meas}}
\newcommand{\CHaus}{\mathsf{CHaus}}
\DeclareMathOperator{\Fun}{Fun}

\newcommand{\cop}{\mathsf{copy}}
\newcommand{\del}{\mathsf{del}}
\newcommand{\swap}{\mathsf{swap}}

\swapnumbers
\theoremstyle{plain}
\newtheorem{dummy}{}[section]
\newtheorem{thm}[dummy]{Theorem}
\newtheorem*{thm*}{Theorem}
\newtheorem{lem}[dummy]{Lemma}\Crefname{lem}{Lemma}{Lemmas}
\newtheorem{prop}[dummy]{Proposition}\Crefname{prop}{Proposition}{Propositions}
\newtheorem{cor}[dummy]{Corollary}
\newtheorem{conj}[dummy]{Conjecture}
\newtheorem{qstn}[dummy]{Question}
\newtheorem{defn}[dummy]{Definition}\Crefname{defn}{Definition}{Definitions}

\newtheorem{prob}[dummy]{Problem}
\theoremstyle{remark}
\newtheorem{ex}[dummy]{Example}\Crefname{ex}{Example}{Examples}
\newtheorem{rem}[dummy]{Remark}

\newtheorem{nota}[dummy]{Notation}
\numberwithin{equation}{section}

\setlist{itemsep=3pt}

\allowdisplaybreaks

\pgfdeclarelayer{edgelayer}
\pgfdeclarelayer{nodelayer}
\pgfsetlayers{edgelayer,nodelayer,main}
\tikzstyle{none}=[]

\tikzstyle{morphism}=[fill=white, draw=black, shape=rectangle]
\tikzstyle{medium box}=[fill=white, draw=black, shape=rectangle, minimum width=0.8cm, minimum height=0.9cm]
\tikzstyle{large morphism}=[fill=white, draw=black, shape=rectangle, minimum width=1.7cm, minimum height=1cm]
\tikzstyle{bn}=[fill=black, draw=black, shape=circle, inner sep=1.5pt]
\tikzstyle{state}=[fill=white, draw=black, regular polygon, regular polygon sides=3, minimum width=0.8cm, shape border rotate=180, inner sep=0pt]
\tikzstyle{medium state}=[fill=white, draw=black, regular polygon, regular polygon sides=3, minimum width=1.3cm, inner sep=0pt, shape border rotate=180]
\tikzstyle{large state}=[fill=white, draw=black, regular polygon, regular polygon sides=3, minimum width=2.2cm, shape border rotate=180, inner sep=0pt]
\tikzstyle{wn}=[fill=white, draw=black, shape=circle, inner sep=1.5pt]

\tikzstyle{arrow}=[->]
\tikzstyle{dashed box}=[-, dashed]

\tikzset{baseline=(current  bounding  box.center)}
\tikzset{every picture/.append style={scale=0.5}}

\begin{document}
\sloppy

\setlength{\jot}{6pt}



\title[A synthetic approach to Markov kernels and statistics]{A synthetic approach to Markov kernels, conditional independence and theorems on sufficient statistics}

\author{Tobias Fritz}

\address{Perimeter Institute for Theoretical Physics, Waterloo, Ontario, Canada}
\email{tfritz@pitp.ca}


\subjclass[2010]{Primary: \href{https://mathscinet.ams.org/mathscinet/search/mscbrowse.html?sk=60a05}{60A05}, \href{https://mathscinet.ams.org/mathscinet/search/mscbrowse.html?sk=62a01}{62A01}; Secondary: \href{https://mathscinet.ams.org/mathscinet/search/mscbrowse.html?sk=62b05}{62B05}, \href{https://mathscinet.ams.org/mathscinet/search/mscbrowse.html?sk=18d10}{18D10}, \href{https://mathscinet.ams.org/mathscinet/search/mscbrowse.html?sk=68q55}{68Q55}}

\begin{abstract}
	\vspace{.4cm}
	We develop \emph{Markov categories} as a framework for synthetic probability and statistics, following work of Golubtsov as well as Cho and Jacobs. This means that we treat the following concepts in purely abstract categorical terms: conditioning and disintegration; various versions of conditional independence and its standard properties; conditional products; almost surely; sufficient statistics; versions of theorems on sufficient statistics due to Fisher--Neyman, Basu, and Bahadur.
	
	Besides the conceptual clarity offered by our categorical setup, its main advantage is that it provides a uniform treatment of various types of probability theory, including discrete probability theory, measure-theoretic probability with general measurable spaces, Gaussian probability, stochastic processes of either of these kinds, and many others.
\end{abstract}

\newgeometry{top=2cm,bottom=2cm} 
\maketitle
\thispagestyle{empty}
\vspace{1cm}
\tableofcontents
\restoregeometry 

\newpage

\section{Introduction}

Probability theory and statistics are traditionally built on measure theory as a foundation. This has facilitated the development of a rich and diverse network of theoretical results and practical methods, comprising areas as distinct as stochastic PDEs, large deviation theory, and machine learning. In some of these areas, it is quite common to work with measure theory directly by applying its basic axioms such as $\sigma$-additivity. The position that we take in this paper is that this methodology is similar to programming a computer directly in terms of machine code. Following seminal ideas of Lawvere~\cite{lawvere}, our goal is to demonstrate that the appeal to such low-level machinery can frequently be avoided: there are simple and purely algebraic axioms for ``systems of Markov kernels'' which allow us to give high-level definitions of some of the basic concepts of probability and statistics, such as probability spaces, random variables, independence and conditional independence, almost surely, statistics and sufficient statistics, up to and including completeness and minimality of statistics. The definitions of these concepts also allow us to derive a surprising number of theorems on these notions purely abstractly: among other things, we prove abstract versions of the semigraphoid properties of conditional independence as well as of the Fisher--Neyman factorization theorem, Basu's theorem and Bahadur's theorem on sufficient statistics.

The advantages of this type of approach are various, and we now mention four points in favour of it. First, perhaps the least interesting advantage is that we achieve a novel \emph{understanding} of the above mentioned concepts. Second, a more important point is the greater \emph{generality} achieved through the development of such a framework: there are many different systems (categories\footnote{We use the words ``category'' and ``categorical'' exclusively in the sense of category theory, bearing no relation to categorical data in the sense of statistics.}) of Markov kernels which satisfy our axioms, ranging from probability theory on finite sets via Gaussian probability theory to full-fledged completely general Markov kernels between measurable spaces and even to stochastic processes, where probability theory takes place in time. Therefore our results immediately specialize to different theorems in each one of these cases. There are other interesting categories which satisfy the same axioms but have little to do with probability theory at all, and again our results apply to them; a good example of such a Markov category is the category of sets and multivalued functions, where the morphisms are not Markov kernels, but display largely analogous behaviour.

Third, a related advantage of an abstract framework such as ours is that the axioms are more \emph{high-level} than the standard measure-theoretic ones. If using the standard ones is analogous to programming a computer in machine code, then using the high-level ones is analogous to programming a computer in a language which provides higher abstraction. This enables the programmer to write larger and more complex programs whose functioning is easier to grasp for human minds. As we will see, something similar happens in our context: our proofs of theorems on sufficient statistics such as Basu's are very simple and visual. This should not be surprising, since a certain amount of proof difficulty already enters in the proofs that a given system, such as measure-theoretic probability, satisfies our axioms. Going beyond existing theorems, we expect that our abstract framework will also facilitate the development of new and more complex results in theoretical statistics. Fourth, there is a \emph{modularity} advantage: most of our results only require a fragment of the axioms to be satisfied, so that these results can be instantiated also in Markov categories where merely this fragment holds.

The word ``synthetic'' in our title refers exactly to this idea of higher abstraction, suggesting that we are dealing with an instance of the \emph{synthetic vs.~analytic} dichotomy. Roughly speaking, a mathematical formalism or theory is \emph{analytic} if it focuses on concrete and specific implementation details, often so specific that they allow the user to answer any particular question in a definite way (at least in principle). The paradigmatic example of an analytic theory is Descartes' \emph{analytic geometry} based on coordinate computations. On the other hand, a formalism or theory is \emph{synthetic} if it merely provides higher-level axioms or even just inference rules for reasoning about the structures of interest. Euclid's axioms for planar geometry is the analogous paradigmatic example here. And indeed, doing geometry in terms of Euclid's axioms has advantages over analytic geometry which are parallel to the ones outlined above\footnote{Which is not to deny that analytic geometry, or analytic theories in general, have distinct advantages over synthetic ones as well.}. In this case, the utility of these advantages has played out in a very interesting way historically, leading to the development of other analytic models of a fragment of the axioms, in the form of hyperbolic geometry and more generally Riemannian geometry.

Let us briefly discuss one more well-established analogue for the relation between our approach and the traditional measure-theoretic one. The theory of \emph{abelian categories} is an abstraction of the theory of modules over a ring bearing some conceptual similarity to how our approach abstracts from the conventional foundations of probability and statistics. Again the theory of abelian categories has similar advantages (and disadvantages) over module theory. To see how these manifest themselves, consider the sentiment that abelian categories are in some sense the ``right'' level of abstraction for the development of homological algebra: the irrelevant implementation details of what the objects of an abelian category ``really are'' have been abstracted away, while the focus is precisely on those structures that are relevant to the development of homological algebra~\cite{grothendieck}. Moreover, the greater generality is useful in that not all applications of homological algebra are to categories of modules; this happens e.g.~in algebraic geometry, where many categories of sheaves are abelian categories which are not categories of modules (not even up to equivalence)~\cite{GP}.

However, we certainly cannot claim that our formalism in its present form is as solid or as definite as the theory of abelian categories. Indeed, we do not expect it to have reached its final form, as some of the technical details in our definitions will almost surely be subject to revision. Nevertheless, we believe that the basic spirit will remain unchanged, and has the potential to become a new paradigm for the mathematical foundations of probability and statistics.

\subsection*{Summary}

We now summarize the main definitions and results of the paper, one section at a time.

\begin{itemize}[leftmargin=.7cm]
	\item \Cref{sec_comons} introduces Markov categories as our main protagonists (\Cref{markov_cat}), following work of Golubtsov as well as Cho and Jacobs. These are categories which axiomatize systems of Markov kernels. The axiomatized structures are sequential composition of Markov kernels, their tensor product, as well as distinguished Markov kernels which implement \emph{copying} and \emph{discarding} of values of random variables. These pieces of structure are subject to certain natural compatibility conditions. We discuss the significance of the axioms and introduce $\FinStoch$, the category of Markov kernels (or stochastic matrices) between finite sets (\Cref{finstoch}) as our basic running example, and $\FinSetMulti$, the category of finite sets and multivalued functions (\Cref{setmulti}).
	\item The next few sections are dedicated to examples of Markov categories. This starts with a general construction in \Cref{sec_mon_mond}, where we prove that the Kleisli category of a monoidal affine monad on a Markov category is again a Markov category (\Cref{affine_monoidal_monad}).
	\item \Cref{sec_giry} introduces our main examples besides $\FinStoch$, namely the category of measurable spaces and Markov kernels $\Stoch$ and its full subcategory of standard Borel spaces $\BorelStoch$. We recall how $\Stoch$ and $\BorelStoch$ arise via \Cref{affine_monoidal_monad} as the Kleisli category of the Giry monad.
	\item \Cref{sec_radon} discusses the Kleisli category of the Radon monad, giving a Markov category again per \Cref{affine_monoidal_monad}. It contains continuous Markov kernels between compact Hausdorff spaces.
	\item \Cref{gaussian} shows how Gaussian probability theory is also described by a Markov category denoted $\Gauss$. Here, the morphisms are given by those Markov kernels $\R^n \to \R^m$ which apply a linear transformation and add independent Gaussian noise with specified expectation and variance.
\end{itemize}
After these concrete examples, \Cref{sec_diagrams,sec_hypergraph,comons_gen} provide general categorical recipes for constructing Markov categories. A reader whose primary interest is in theoretical statistics may skip these without significant loss, and perhaps return to \Cref{sec_diagrams} on a second reading in order to see how our results automatically generalize to stochastic processes.
\begin{itemize}[leftmargin=.7cm]
	\item \Cref{sec_diagrams} explains how we can construct new Markov categories by considering diagrams of a given shape in a given Markov category. This implies in particular that all of our definitions and results can be instantiated straightforwardly on stochastic processes, since these are diagrams given by infinite sequences of deterministic morphisms.
	\item \Cref{sec_hypergraph} briefly describes the construction of Markov categories from hypergraph categories, recovering $\FinStoch$ and $\FinSetMulti$ as particular examples.
	\item \Cref{comons_gen} explains how commutative comonoids in \emph{any} symmetric monoidal category form a Markov category with respect to the counit-preserving morphisms.
\end{itemize}
Having then treated examples of Markov categories in detail, at this point we turn to the development of probability and statistics in the abstract framework of Markov categories.
\begin{itemize}[leftmargin=.7cm]
	\item \Cref{sec_det} introduces deterministic morphisms in a Markov category as those morphisms which preserve the comultiplication (\Cref{defn_det}). We show that in $\FinStoch$, $\Stoch$ and $\Gauss$, the deterministic morphisms are indeed exactly those which do not involve any randomness (\Cref{finstoch_det,stoch_det,borelstoch_det,gauss_det}). The second half of the section is devoted to certain categorical technicalities for Markov categories: we prove that the structure morphisms are necessarily deterministic (\Cref{structure_det}), and show that the subcategory of deterministic morphisms is a \emph{cartesian} monoidal subcategory (\Cref{rem_det_cartesian}). We then define the 2-category of Markov categories (\Cref{cacom}) and prove a strictification theorem (\Cref{cacom_strict}) which states that every Markov category is comonoid equivalent to another one in which the monoidal structure is strict.
	\item \Cref{furtheraxioms} introduces four possible additional axioms which a given Markov category may or may not satisfy, and which we believe to have some relevance for probability theory and statistics. The first one is the existence of \emph{conditionals} (\Cref{defn_has_conds}), which we later relate to general disintegrations (\Cref{disint}). The second one is \emph{randomness pushback}, formalizing the intuitive idea that any stochastic operation can be implemented by first generating randomness and then executing a deterministic function which uses the previously generated randomness as an additional input. Third, we discuss another property which we call \emph{positivity} (\Cref{positive_defn}), closely related to the nonnegativity of probability, which will turn out to be useful in later proofs. Fourth, we consider a condition which we call \emph{causality} (\Cref{causal_defn}), expressing the idea that if a choice between two processes does not affect how their outcome is correlated with earlier outcomes, then there should not be any correlation with even earlier outcomes either.

		For each of these axioms, we also discuss which ones of our particular Markov categories satisfy them. The existence of conditionals fails for $\Stoch$ (\Cref{stoch_cond_dist,stoch_disint}) but holds in $\FinStoch$, $\BorelStoch$ and $\Gauss$ (\Cref{finstoch_cond,gauss_cond}). Randomness pushback likewise holds in $\FinStoch$ and $\Gauss$ (\Cref{finstoch_rp,gauss_rp}), while we do not know whether it holds in $\Stoch$. All of our main examples satisfy the positivity axiom (\Cref{cond_positive,stoch_positive}) as well as the causality axiom (\Cref{cond_causal,stoch_causal}).
	\item \Cref{cind} treats conditional independence in our setting. We give various definitions of conditional independence for different kinds of morphisms (\Cref{cistatedef,ciprocdef,cigendef,def_cimarkov}), including the difference on whether one wants to condition on an input or on an output (or both). We prove corresponding versions of the semigraphoid properties (\Cref{cistateprops,ciprocprops} and \Cref{cigenprops,cimarkovprops}), and various results on how conditional independences propagate across composition of morphisms (\Cref{condind_compprop}). Our main innovation on conditional independence over previous work of Cho and Jacobs is that we do not assume the existence of conditionals.
		
		Furthermore, in the presence of conditionals, we also develop the theory of conditional products of joint distributions, showing that these can be defined and satisfy axioms originally proposed by Dawid and Studen{\'y} in any Markov category which has conditionals (\Cref{condprod_defn} and after).
	\item \Cref{sec_as} treats almost sure equality (\Cref{defn_as}), again following an idea of Cho and Jacobs. We prove a number of lemmas on the compositional structure of almost sure equality. We also construct the category of probability spaces and Markov kernels modulo almost sure equality in any Markov category satisfying causality (\Cref{probstoch}). This implies that Bayesian inversion is a symmetric monoidal dagger functor, solving an open problem posed by Clerc, Danos, Dahlqvist and Garnier (\Cref{dagger}). We then introduce the concept of almost sure determinism (\Cref{defn_as_det}) as another example of how our setup allows for a systematic relativization of probabilistic concepts with respect to almost surely. Finally, we introduce a candidate definition of \emph{support} of a morphism (\Cref{defn_support}), conjecturally recovering the usual notion of support of a probability measure (\Cref{supp_usual}).
	\item The final three sections are devoted to the development of some of the basic notions and theorems of statistics within our formalism. This theory frequently makes use of aspects of conditional independence as studied in \Cref{cind}, in line with Dawid's arguments on the central role of conditional independence in statistics.
		
		\Cref{suff} starts by introducing statistical models (\Cref{def_stat_model}) and statistics (\Cref{def_statistic}), leading up to the definition of sufficient statistic (\Cref{suff_defn}). This is followed by \Cref{thm_FN}, which seems to be an abstract close relative of the Fisher--Neyman factorization theorem. This interpretation is corroborated by showing that our result indeed induces the Fisher--Neyman factorization in the case of $\FinStoch$ (\Cref{finstoch_fn}).
	\item \Cref{sec_basu} builds on the previous section by treating the completeness of statistics. Observing that the notion of completeness is secretly a property of Markov kernels in general (\Cref{complete_rem}), we find a simple definition of completeness within our abstract formalism (\Cref{defn_complete}). We show that it specializes to bounded completeness in the case of $\Stoch$ (\Cref{stoch_complete}). After defining ancillary statistics (\Cref{defn_ancillary}), we then prove an abstract version of Basu's theorem (\Cref{thm_basu}).
		\begin{thm*}[Basu]
			Any complete sufficient statistic for a given statistical model is independent of any ancillary statistic.	
		\end{thm*}
	\item \Cref{sec_bahadur} goes further by introducing minimal sufficient statistics within our abstract setup (\Cref{min_suff}), based on a preordering on the set of statistics for a given statistical model which compares statistics by their informativeness (\Cref{stat_preorder}). We then state and prove an abstract version of Bahadur's theorem.
		\begin{thm*}[Bahadur]
			If a minimal sufficient statistic exists for a given statistical model, then a complete sufficient statistic is minimal sufficient.
		\end{thm*}
\end{itemize}
This concludes our summary. Let us reemphasize that all of our definitions should be regarded as subject to further refinement. Of course, our results based on these preliminary definitions are nevertheless definite.

\subsection*{Discussion of prior work} The most important precursors to this paper are works by Golubtsov~\cite{golubtsov1,golubtsov2,golubtsov3,GM}\footnote{Due to the substantial overlap between those papers, we generally only refer to~\cite{golubtsov2} from now on, which is easily digitally accessible. The references in those papers also point to earlier work of Golubtsov on particular categories of interest, where some of the relevant concepts have been studied concretely.} as well as a recent paper by Cho and Jacobs~\cite{cho_jacobs}. The present paper can be regarded as a sort of follow-up to these, and we will make frequent reference to them. In particular, these prior authors already demonstrated that Markov categories\footnote{Or \emph{affine CD-categories} in the terminology of Cho and Jacobs. The \emph{categories of information transformers} or \emph{IT-categories} in the sense of Golubtsov are technically slightly different.} can serve as a canvas for developing aspects of probability theory. Markov categories had also been used somewhat implicitly in Fong's study of the categorical semantics of Bayesian networks~\cite{causaltheories}. 

However, synthetic approaches to probability and statistics in general have a long history going back further than these works, and have been pursued mostly independently in several scientific communities. We here provide a brief overview of the main literature that we are currently aware of.\smallskip

\begin{itemize}[leftmargin=.5cm]
	\item In category theory, categorical probability has been initiated by the work of Lawvere~\cite{lawvere} and Giry~\cite{giry}. It has traditionally focused on the study of \emph{probability monads}~\cite[Chapter~1]{perrone_thesis}, providing formal descriptions of the measure-theoretic and functional-analytic aspects of probability~\cite{avery,lucyshyn_wright}. This is closely related to the present paper insofar as Kleisli categories of probability monads are categories of Markov kernels, and should therefore generally be expected to be Markov categories in the sense of \Cref{markov_cat} (see also \Cref{affine_monoidal_monad}).
	\item Computer scientists have studied structural aspects of probability theory extensively, often with applications to the semantics of probabilistic programs in mind. The seminal paper of Cho and Jacobs~\cite{cho_jacobs} belongs to this area, as well as other work\footnote{The references listed are but a sample of these authors' extensive works on the subject.} by Jacobs~\cite{effectus}, Panangaden~\cite{BMPP}, Simpson~\cite{simpson_sheaves}, Staton~\cite{HKSY}, and others. Due to the affinity of theoretical computer scientists with category theory, there is a strong overlap here with work done by category theorists. This subject seems to be currently maturing, as witnessed by the first book-length treatment on structured probabilistic reasoning becoming available~\cite{jacobs_book}.
	\item In mathematical statistics, there is a large number of papers on its foundations, some of which have investigated approaches which are synthetic in flavour, and some of which use categorical concepts either explicitly or implicitly. For example, this includes works by Dawid on statistical operations~\cite{dawid_operations}, conditional independence~\cite{dawid,dawid2} and conditional products~\cite{DS}, by Lauritzen on combining statistical models~\cite{ML} and composing Gaussian conditionals~\cite{LJ}, as well as by McCullagh on the concept of statistic~\cite{mccullagh}.
	\item In noncommutative probability theory, synthetic notions of independence which generalize various concrete notions of independence have been studied using categorical methods, such as free independence in free probability, for example in the works of Franz~\cite{franz} as well as Gerhold, Lachs and Sch\"urmann~\cite{GLS}.
	\item In the foundations of quantum mechanics community, researchers have attempted to construct quantum analogues of Bayesian inference and Bayesian networks. This endeavour is crucially informed by having a synthetic approach to the classical case to begin with. For example, the work of Coecke and Spekkens~\cite{CS} belongs here.
	\item In information geometry, \v{C}encov had investigated the category $\Stoch$ as the ``category of statistical decisions'' independently of Lawvere and Giry~\cite{cencov,cencov_book}, where it was used subsequently as a canvas for studying entropic distances on probability measures~\cite{cencov_later,MC}.
\end{itemize}

We have tried to integrate some discussion of what has been done in each of these areas, to the best of our knowledge, into the main text. Nevertheless, it is almost inevitable that some readers, in particular those coming from one of these particular fields, will feel that their own area remains underrepresented in our presentation. For these readers, we hope that the above list can provide some convenient entry points into the existing literature of other communities which they may not be aware of.

\subsection*{Conventions and notation.} Throughout, we use the term ``Markov kernel'' in the standard sense, referring to a map between measurable spaces which takes every element of the first space to a \emph{probability measure} on the second space, such that a suitable measurability requirement holds; see e.g.~\cite[Definition~8.25]{klenke} or~\cite[D\'efinition~III-2.1]{neveu}. In other words, ``Markov kernel'' is more or less synonymous with other terms like ``transition kernel'', ``stochastic relation'', ``stochastic map'', ``probabilistic map'', ``conditional distribution'', ``channel'', ``statistical operation'', ``information transformer'' and others (neither one of which we will use). We present a detailed technical treatment of Markov kernels in \Cref{sec_giry}. However, precisely because our categorical formalism is more general and more abstract than the measure-theoretic one, most of our uses of the term ``Markov kernel'' will not require the reader to know the technical details. It is perfectly enough to associate the term ``Markov kernel'' merely with ``noisy map'', i.e.~a function whose output may be inherently random. The morphisms in those Markov categories which are of interest in probability and statistics are always some form of Markov kernels or noisy maps, even if the technical details may differ. We also think of a morphism in a general Markov category as an abstract version of a Markov kernel.

Our categorical setup is as follows. Throughout, $\C$ is a symmetric monoidal category with monoidal unit $I \in \C$ and symmetry morphisms $\swap_{X,Y} : X \otimes Y \to Y \otimes X$. We freely use \emph{string diagrams} as the usual graphical calculus for symmetric monoidal categories~\cite{selinger}, as these are more intuitive and easier to parse for a human reader than long equations. Our convention is that these string diagrams are to be read from bottom to top: the domain of a composite morphism defined by a string diagram is at the bottom while the codomain is on top. Hence our string diagrams can be interpreted as processes with time flowing upwards. For example,~\cref{comonoid_ass} should be interpreted as stating that creating three copies of an object can be achieved either by first copying and then copying the second copy again, or first copying and then copying the first copy again. The symmetry morphisms $\swap_{X,Y}$ are depicted graphically simply as crossings,
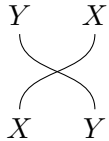
\[
	\input{swap.tikz}
\]
While we usually leave out the object labels on the wires of our string diagrams, we occasionally decorate the overall input and output wires by the corresponding object labels for clarity and emphasis.

While the possibility of using string diagrams is very convenient, we do not regard it as a central selling point of our framework. Rather, our main innovation lies in the greater generality and conceptual clarity facilitated by a categorical treatment, independently of which notation one uses for categorical reasoning.

\subsection*{Acknowledgments} We thank J\"urgen Jost and H{\^o}ng V{\^a}n L{\^e} for the discussions which spurred this paper; Manuel B\"arenz, Kenta Cho, Toma\v{s} Gonda, Bart Jacobs, Dominique Pastor, Paolo Perrone, Eigil Fjeldgren Rischel, Alex Simpson and David Spivak for useful discussions and detailed feedback on a draft; Brendan Fong, TC Fraser, Richard Garner, Paul Marriott, Arthur Parzygnat, Rob Spekkens, Sam Staton and Dario Stein for further helpful discussions; Sam Tenka for pointing out Basu's theorem and Bahadur's theorem to us; as well as Bob Coecke, Philip Dawid, Robert Furber, Peter Golubtsov, Alexander Kurz, Steffen Lauritzen, Nina Otter and MathOverflow user Steve for additional valuable pointers to the literature. Last but not least, we thank the creators of and contributors to the \href{https://ncatlab.org/nlab/show/HomePage}{nLab} and \href{https://tikzit.github.io/}{TikZiT} for having made research and dissemination in category theory so much easier.

Part of this work was conducted while the author was with the Max Planck Institute for Mathematics in the Sciences.

\section{Markov categories}
\label{sec_comons}

Here, we introduce our main objects of study: \emph{Markov categories}, which are symmetric monoidal categories with extra structure which makes them behave (in many ways) like categories of Markov kernels. As far as we know, these kinds of categories have first appeared in slightly different form\footnote{See \Cref{freyd} for a discussion of the difference.} in work of Golubtsov~\cite[Sections~2 and~3]{golubtsov2}. In the present form they first played a major role in Fong's work on the categorical structure of Bayesian networks~\cite{causaltheories}, even if somewhat implicitly. They crop up similarly implicitly in later work by Jacobs and Zanasi~\cite{JZ}, and have subsequently been axiomatized explicitly by Cho and Jacobs in their categorical approach to conditional independence and Bayesian inference~\cite{cho_jacobs}. Cho and Jacobs call them \emph{affine CD-categories}, where ``CD'' stands for ``Copy/Discard'', describing the interpretation of the structure morphisms~\cref{cd}. Due to the central role that these categories seem to play in probability and statistics, we introduce a catchier term which hints at the idea that the morphisms in the categories under consideration behave like Markov kernels.

\begin{defn}
	\label{markov_cat}
	A \emph{Markov category} $\C$ is a symmetric monoidal category in which every object $X \in \C$ is equipped with a commutative comonoid structure given by a comultiplication $\cop_X : X \to X \otimes X$ and a counit $\del_X : X \to I$, depicted in string diagrams as
	\beq
		\label{cd}
		\input{comon_structures.tikz}
	\eeq
	and satisfying the commutative comonoid equations,
	\beq
		\label{comonoid_ass}
		\input{comonoid1.tikz}
	\eeq
	\beq
		\label{comonoid_other}
		\input{comonoid2.tikz}
	\eeq
	as well as compatibility with the monoidal structure,
	\begin{equation}
		\label{delcopyAB}
		\input{multiplicativity.tikz}
	\end{equation}
	and naturality of $\del$, which means that
	\begin{equation}
		\label{counit_nat}
		\input{counit_natural.tikz}
	\end{equation}
	for every morphism $f$.
\end{defn}

We usually think of a morphism in $\C$ as some form of Markov kernel, meaning that it assigns to every input value a probability distribution over output values. The comultiplication $\cop_X$ then represents the map which takes an input value, copies it, and outputs the two copies without introducing any randomness; while $\del_X$ is the map which simply discards its input and produces no output. The above equations then all acquire an intuitive meaning, such as that copying and discarding one copy is equal to simply outputting the input, a property sometimes known as \emph{broadcasting}~\cite{broadcast}.

The Markov kernel interpretation, including the role of $\cop$ and $\del$, will be made precise in \Cref{finstoch} and \Cref{sec_giry,sec_radon} by constructing actual categories of Markov kernels which are Markov categories. The main purpose of our paper is to demonstrate that these pieces of structure are enough to develop a substantial amount of probability theory and statistics.

The first concept from probability theory which acquires meaning in any Markov category is that of a probability measure, or \emph{distribution}; we prefer the latter term due to its more open-ended connotation. A distribution is a Markov kernel which takes no input and produces an output, interpreted as a ``random element'' of the codomain. This motivates the definition that a \emph{distribution} in any Markov category $\C$ is a morphism $\psi : I \to X$ from the unit object $I$ to some other object $X$, depicted in string diagram notation as
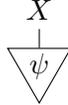
\[
	\begin{tikzpicture}
	\begin{pgfonlayer}{nodelayer}
		\node [style=state] (0) at (0, 0) {$\psi$};
		\node [style=none] (1) at (0, 1) {};
		\node [style=none] (2) at (0, 1.5) {$X$};
	\end{pgfonlayer}
	\begin{pgfonlayer}{edgelayer}
		\draw (1.center) to (0);
	\end{pgfonlayer}
\end{tikzpicture}
	
\]
In the Markov category $\Stoch$, our example from~\Cref{sec_giry}, these are precisely the probability measures on measurable spaces $X$. We can now say that the pair $(X,\psi)$ is a \emph{probability space} in $\C$; in $\Stoch$, this indeed specializes precisely to the traditional notion of probability space. A \emph{random variable} on $(X,\psi)$ taking values in some other object $Y$ can then be defined to be a morphism $f : X \to Y$, where one may want to impose the additional requirement that $f$ must be deterministic (\Cref{defn_det}), and perhaps identify two such morphisms as representing the same random variable as soon as they are $\psi$-almost surely equal (\Cref{defn_as}).\footnote{Since we will have no need for the term ``random variable'' in this paper, we do not commit ourselves to a particular definition of this concept yet, and mention it merely to illustrate how to work with Markov categories. But see \Cref{prob_cat} for further treatment of probability spaces in Markov categories.} The distribution of the random variable $f$ is then given by the composition $f \psi : I \to Y$,
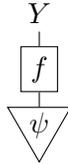
\[
	\begin{tikzpicture}
	\begin{pgfonlayer}{nodelayer}
		\node [style=state] (0) at (0, -1.5) {$\psi$};
		\node [style=morphism] (1) at (0, 0) {$f$};
		\node [style=none] (2) at (0, 1) {};
		\node [style=none] (3) at (0, 1.5) {$Y$};
	\end{pgfonlayer}
	\begin{pgfonlayer}{edgelayer}
		\draw (2.center) to (1);
		\draw (1) to (0);
	\end{pgfonlayer}
\end{tikzpicture}

\]
so that the pair $(Y,f\psi)$ is also a probability space in $\C$. This string diagram can be interpreted conveniently in terms of processes happening in time: first, $\psi$ generates a random element of $X$; second, this element is fed into the function $f$, which subsequently outputs an element of $Y$. The overall effect is that we have produced a random element of $Y$, corresponding to the distribution of our random variable. One takeaway of this is that the formalism of Markov categories unifies the concepts of probability distributions and random variables by considering both of them as instances of the more general concept of Markov kernel.

For objects $X,Y \in \C$, morphisms $\psi : I \to X \otimes Y$ correspond to \emph{joint distributions}. In the case of $\Stoch$, these are exactly the probability measures on the product $\sigma$-algebra on the product of measurable spaces $X$ and $Y$. In general, such a joint distribution can be \emph{marginalized} over $Y$ by composing with $\id_X \otimes \del_Y$. In terms of string diagrams, the resulting composite morphism $(\id_X \otimes \del_Y) \circ \psi$ can be conveniently depicted as
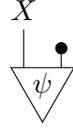
\[
	\input{marginal.tikz}
\]
and similarly for marginalization over $Y$. Analogous statements apply to joint distributions on more than two factors. Of course, these interpretations are also central to the work of Cho and Jacobs~\cite{cho_jacobs}, and have also been used prior to that e.g.~in the work of Coecke and Spekkens~\cite{CS}.

\begin{rem}
	Written out in a way which explicitly keeps track of the structure isomorphisms in $\C$, the first condition in \cref{delcopyAB} amounts to the commutativity of the diagram
	\beq
		\label{deldiag}
		\begin{tikzcd}
			X \otimes Y \ar["\del_{X\otimes Y}"]{rr} \ar[equal]{d} & & I \ar["\cong"]{d} \\
			X \otimes Y \ar["\del_X \otimes \id"]{r} & I \otimes Y \ar["\id \otimes \del_Y"]{r} & I \otimes I
		\end{tikzcd}
	\eeq
	and similarly the second one is
	\beq
		\label{copydiag}
		\begin{tikzcd}[column sep=1.5cm, row sep=1cm]
			X \otimes Y \ar["\cop_{X \otimes Y}"]{rr} \ar[equal]{ddd} & & (X \otimes Y) \otimes (X \otimes Y) \ar["\cong"]{d} \\
			& & (X \otimes (Y \otimes X)) \otimes Y \ar["(\id \otimes \swap) \otimes \id"]{d} \\
			& & (X \otimes (X \otimes Y)) \otimes Y \ar["\cong"]{d} \\
			X \otimes Y \ar["\cop_X \otimes \id"]{r} & (X \otimes X) \otimes Y \ar["\id \otimes \cop_Y"]{r} & (X \otimes X) \otimes (Y \otimes Y)
		\end{tikzcd}
	\eeq
	where the unnamed isomorphisms are the monoidal unitors and composite associators.
\end{rem}

\begin{rem}
One may also think that we should impose $\del_I = \cop_I = \id_I$, where $I$ is the monoidal unit. These constraints are automatic if $\C$ is strict, because then we have
\[
	\del_I \otimes \del_I = \del_I, \qquad \cop_I \otimes \cop_I = \cop_I,
\]
by assumption, as well as
\[
	\del_I \otimes \cop_I = \del_I \circ \cop_I = \id_I
\]
in the commutative endomorphism monoid $\C(I,I)$. These properties imply
\[
	\del_I = \del_I \otimes \del_I \otimes \cop_I = \del_I \otimes \cop_I = \id_I,
\]
and hence also $\cop_I = \id_I$.

If $\C$ is not strict, then the equation $\cop_I = \id_I$ is not actually well-typed. Rather, the correct statement is that $\del_I = \id_I$ and that $\cop_I : I \to I \otimes I$ is equal to the unique coherence isomorphism. The proof is similar.
\end{rem}

By $\del_I = \id_I$ and~\cref{counit_nat}, we can conclude that $I$ is a terminal object. Put differently, our postulated properties of $\del$ can equivalently be phrased as saying that $\C$ should be a \emph{semicartesian} symmetric monoidal category, for which there are several equivalent characterizations (see~\cite{semicartesian} and~\cite[Theorem~3.5]{GLS}), with the simplest one being that $I$ should be terminal\footnote{This condition has also been considered in~\cite{cho_jacobs,coecke,KHC}. However, their claimed equivalence with~\cref{counit_nat} is missing the assumption $\del_I = \id_I$.}~\cite[Definition~2.1(ii)]{jacobs_weakening}, another one being the notion of \emph{monoidal category with projections} as studied by Franz with a minor variation~\cite[Definition~3.3]{franz}\footnote{The minor variation is that in order for Franz's definition to be equivalent to semicartesianness, one needs to amend it by the extra condition that at least one of the two projections $\pi_1, \pi_2 : I \otimes I \to I$ must coincide with the structure isomorphism $I \otimes I \stackrel{\cong}{\to} I$, as per~\cite{leinster_comment} and~\cite[Theorem~3.5]{GLS}. If $\pi_1$ does, then any $f : X \to I$ is equal to the composite $X \stackrel{\cong}{\to} I \otimes X \stackrel{\pi_1}{\to} I$ due to the diagram
\[
	\begin{tikzcd}[ampersand replacement=\&]
		X \ar["f"]{d} \ar["\cong"]{r} \& I \otimes X \ar["\id \otimes f"]{d} \ar["\pi_1"]{r} \& I \ar[equal]{d} \\
		I \ar["\cong"]{r} \& I \otimes I \ar["\cong"]{r} \& I 
	\end{tikzcd}
\]
where the right-hand square is a naturality square of $\pi_1$. In other words, $I$ is terminal. Conversely, it is easy to see that if $I$ is terminal, then $\C$ is a monoidal category with projections satisfying the extra condition.}. In particular, the first equation in~\cref{delcopyAB} can also be seen as a consequence of the terminality of $I$.

\begin{rem}
	It is well-known that if $\C$ is cartesian monoidal, then there is a unique comonoid structure on every object. This makes $\C$ into a Markov category. However, the interesting Markov categories are those not of this form.
\end{rem}

While we will consider a substantial range of particular Markov categories in \Cref{sec_mon_mond,sec_giry,sec_radon,gaussian,sec_diagrams,sec_hypergraph,comons_gen}, we now introduce two basic examples, which we start with due to the technical simplicity of working with finite sets. We will use the first one throughout as our primary running example.

\begin{ex}
	\label{finstoch}
	$\C$ may be $\FinStoch$, the category of finite sets with Markov kernels. This means that a morphism $f:X\to Y$ between finite sets\footnote{The restriction to finite sets is not essential and can easily be lifted. The most general Markov category relevant for discrete probability would be the one where the objects are sets (of arbitrary cardinality) and a morphism $f : X \to Y$ is a matrix $(f_{xy})_{x\in X, y \in Y}$ with nonnegative real entries such that $\sum_{y \in Y} f_{xy} = 1$ for every $x \in X$. In particular, for every $x$ there are at most countably many $y$ with $f_{xy} > 0$. We prefer to use $\FinStoch$ as our version of discrete probability because of its simplicity.} $X$ and $Y$ is a stochastic matrix $(f_{xy})_{x\in X,\, y\in Y}$\footnote{Recall that this means that $f_{xy} \geq 0$ and $\sum_y f_{xy} = 1$. While the primary case of interest is in matrices with nonnegative real entries, one can in principle work with nonnegative entries in any ordered field, and all results proven in this paper still hold just the same with the same proofs.}, where the entry $f_{xy}$ is interpreted as the probability that the Markov kernel $f$ returns output $y$ on input $x$. Composition of morphisms is given by matrix multiplication. In the following, we also use the suggestive notation $f(y|x)$ for such a morphism $f : X \to Y$, so that composition of $f$ with some $g : Y \to Z$ turns into the Chapman--Kolmogorov equation,
	\beq
		\label{ck}
		(gf)(z|x) = \sum_y g(z|y) f(y|x).	
	\eeq
	There is a simple inclusion functor $\FinSet \to \FinStoch$, which takes an honest function between finite sets $f : X \to Y$ and turns it into a stochastic matrix: writing $f$ for both by abuse of notation, we have
	\[
		f(y|x) := \begin{cases} 1 & \textrm{ if } y = f(x), \\ 0 & \textrm{ otherwise.}	\end{cases}
	\]
	The symmetric monoidal structure on $\FinStoch$ is given by cartesian product at the level of objects, and the tensor product of stochastic matrices at the level of morphisms: for $f : A \to X$ and $g : B \to Y$,
	\beq
		\label{fg}
		(f \otimes g)(xy|ab) := f(x|a) \, g(y|b),
	\eeq
	with the obvious structure morphisms, inherited from the symmetric monoidal inclusion functor $\FinSet \to \FinStoch$ based on the cartesian monoidal structure on $\FinSet$. The bifunctoriality of this tensor product is easy to verify.

	Finally, the comonoid operations are also inherited from $\FinSet$, which is a Markov category by virtue of being cartesian monoidal. More concretely, $\del_X : X \to I$ is the all-ones row vector indexed by the elements of $X$, which we formally write in the above probability notation as
	\[
		\del_X(|x) = 1,
	\]
	where the first argument is empty since $\del_X$ is a morphism with trivial output. $\cop_X$ is the stochastic matrix with entries
	\[
		\cop_X(x_1,x_2|x_0) := \begin{cases} 1 & \textrm{if } x_1 = x_2 = x_0,\\
							0 & \textrm{otherwise}. \end{cases}
	\]
	While it is routine to verify the relevant equations from \Cref{markov_cat} directly, a more elegant and abstract argument is to use the fact that these equations are automatic in $\FinSet$, where the monoidal structure is cartesian, and then transfer to $\FinStoch$ by the functoriality and symmetric monoidality of the inclusion $\FinSet \to \FinStoch$. This works for all conditions except~\cref{counit_nat}, which is obvious.

	It may be instructive to see explicitly why the monoidal structure is not cartesian. A morphism $\psi : I \to X \otimes Y$ is a joint distribution, whose composition with the two projections $X \otimes Y \to X$ and $X \otimes Y \to Y$ gives the two marginal distributions $I \to X$ and $I \to Y$. Now if $\FinStoch$ was cartesian, the joints $I \to X \otimes Y$ would have to be in bijection with the pairs of marginals $I \to X$ and $I \to Y$. Thus the fact that a joint distribution is usually not uniquely specified by its marginals implies that $\FinStoch$ is not cartesian. More generally, this is the reason for why cartesian monoidal categories are not interesting as Markov categories.
\end{ex}

\begin{ex}[{See also~\cite{golubtsov_multi} and~\cite[Section~8.4]{golubtsov2}}]
	\label{setmulti}
	$\C$ may be $\FinSetMulti$, the category of finite sets with multivalued functions, defined in more detail as follows. A morphism $f : X\to Y$ between sets $X$ and $Y$ is a function from $X$ to the set of nonempty subsets of $Y$, or equivalently a matrix $(f(y|x))_{x\in X, y \in Y}$ with entries in $\{0,1\}$ such that for every $x \in X$ there is $y \in Y$ with $f(y|x) = 1$, where again we use notation suggestive of conditional probabilities as in \Cref{finstoch}. We think of the statement $f(y|x) = 1$ as saying that the output $y$ \emph{is possible} given the input $x$, without specifying a further quantification of this possibility, while $f(y|x) = 0$ states that the output $y$ is \emph{impossible} on input $x$. The extra condition is then that for every input, at least one output is possible. It is in this sense that we are dealing with multivalued functions.
	
	The definition of composition is then just the usual composition of relations: for $f : X \to Y$ and $g : Y \to Z$, we put $(gf)(z|x) := 1$ if and only if there is $y \in Y$ such that $g(z|y) = 1$ and $f(y|x) = 1$, meaning that an output $z$ is possible for $gf$ if and only if an output $y$ is possible for $f$ such that the output $z$ is possible for $g$ on input $y$. Using the convention $1 + 1 := 1$, this composition operation again takes the form of the Chapman--Kolmogorov equation~\cref{ck}. The tensor product of morphisms is also as in~\cref{fg}, making an output $xy$ possible on input $ab$ if $x$ is possible for $f$ on $a$ and $y$ is possible for $g$ on $b$. The comonoid morphisms $\del$ and $\cop$ take the exact same form as in $\FinStoch$.

	Although $\FinSetMulti$ looks like an interesting example for the concepts and results developed in the rest of this paper, we will leave its more detailed investigation in our context to future work.
\end{ex}

\begin{rem}
	Given morphisms $f : X \to Y$ and $g : X \to Z$ in a Markov category, we can form a morphism $X \to Y \otimes Z$ given by
	\beq
		\label{cond_prod_proc}
		\input{conditional_product.tikz}	
	\eeq
	It is not hard to see that this equips the coslice category $X/\C$ with a monoidal structure, where the monoidal unit is given by $\del_X$. This hints at a possible relation between Markov categories and monoidal fibrations~\cite{shulman} by virtue of equipping the domain fibration $\dom : \C^\to \to \C$ with a monoidal structure, and possibly hints at an equivalent but more abstract definition of Markov categories. For a more ``pedestrian'' definition of a categorical structure very similar to Markov categories based on this operation, see~\cite[Definition~1]{golubtsov2}.

	In categories of Markov kernels like $\FinStoch$, the monoidal structure on the coslices $X/\C$ is closely related to \emph{conditional products} of probability distributions, which we will investigate in~\Cref{condprod_defn}.
\end{rem}

\begin{nota}
	\label{formal_notation}
	Although we will do this only with explicit mention, it can sometimes be convenient to use notation like ``$f(x,y|a,b)$'' to indicate codomain and codomain of a morphism $f : A \otimes B \to X \otimes Y$ in a \emph{general} Markov category. This generalizes the usual notation for $\FinStoch$ from \Cref{finstoch}. We can then use the Chapman--Kolmogorov equation in the form~\cref{ck} as a general \emph{formal} notation for composition of morphisms, irrespectively of whether such composition involves any sort of summation in the category under consideration. We can similarly use~\cref{fg} as a general formal notation for the tensor product of morphisms.

	Given a morphism $f : A \to X \otimes Y$, composing with $\del_X \otimes \id : X \otimes Y \to Y$ implements \emph{marginalization} over the first variable. Just as we can denote the original morphism by $f(x,y|a)$, we can simply write $f(y|a)$ for this marginal, and likewise for the similarly defined other marginal $f(x|a)$. In line with our formal probability notation, the defining equation of marginalization
	\[
		f(y|a) = \sum_x f(x,y|a)
	\]
	thereby acquires a completely formal meaning valid in any Markov category, where the sum over $x$ corresponds to the composition in $(\del_X \otimes \id) \circ f$. The comonoid maps $\del_X$ and $\cop_X$ do not appear explicitly in this notation: $\del_X$ is denoted simply by the empty symbol, indicating no dependence on $x$; while using $\cop_X$ amounts to using the same variable $x$ in an expression twice. For example, the composition of the tensor product of $f : X \to Y$ and $g : X \to Z$ with $\cop_X$ as in~\cref{cond_prod_proc} will be the morphism denoted simply by $f(y|x) \, g(z|x)$.

	Despite the apparent utility and generality of this notation, we will leave a complete formalization of it to future work, and will use it in this paper only occasionally and with explicit mention. It may be interesting to note that it is similar to the \emph{abstract index notation} used in differential geometry and general relativity~\cite{wald}.
\end{nota}

\begin{nota}
	As is commonly done with monoid or comonoid structures in the context of string diagrams, we also draw a bullet with three wires coming out
	\[
		\input{triple_copy.tikz}
	\]
	as a shorthand for either side of the associativity equation~\cref{comonoid_ass}, to be interpreted as copying with three outputs. And similarly for more than three outgoing wires, which is easily seen to be unambiguous by repeated application of associativity.
\end{nota}

\section{Example: Kleisli categories of monoidal monads}
\label{sec_mon_mond}

In this section, we give a general construction of Markov categories as Kleisli categories of symmetric monoidal affine monads. Subsequently, we will sketch the application of this construction to the Giry monad and to the Radon monad. This implies that all results from our later sections can be instantiated beyond the discrete case implemented by $\FinStoch$, namely in particular in the category of measurable spaces and measurable Markov kernels (\Cref{sec_giry}) and in the category of compact Hausdorff spaces and continuous Markov kernels (\Cref{sec_radon}). We expect that an analogous straightforward development will be possible also in other cases, such as for the monad of Scott-continuous probability valuations on the category of (all) topological spaces~\cite[Section~10]{heckmann} (see also~\cite{GLJ,FPR}).

To begin the general construction, let $\D$ be a symmetric monoidal category. In the examples of interest, $\D$ will typically be cartesian monoidal, but we do not need this assumption. Further, let $T : \D \to \D$ be a monad with multiplication $\mu : TT \to T$ and unit $\eta : 1 \to T$. It is a \emph{symmetric monoidal monad} (e.g.~\cite{guitart,seal}) if it comes equipped with morphisms
\[
	\nabla_{X,Y} : TX \otimes TY \longrightarrow T(X\otimes Y)
\]
natural in $X$ and $Y$, which make $T$ into a lax symmetric monoidal functor with unit $\eta_I : I \to TI$, and such that $\mu$ and $\eta$ are monoidal transformations. Concretely, the latter means that the diagrams
\beq
	\label{monmon1}
	\begin{tikzcd}
		TTX \otimes TTY \ar["\nabla"]{r} \ar["\mu\otimes\mu"]{d} & T(TX \otimes TY) \ar["T\nabla"]{r} & TT(X \otimes Y) \ar["\mu"]{d} \\
		TX \otimes TY \ar["\nabla"]{rr} & & T(X \otimes Y)
	\end{tikzcd}
\eeq
\beq
	\label{monmon2}
	\begin{tikzcd}
		& X \otimes Y \ar[swap,"\eta \otimes \eta"]{dl} \ar["\eta"]{dr} \\
		TX \otimes TY \ar["\nabla"]{rr} & & T(X \otimes Y)
	\end{tikzcd}
\eeq
must commute for all $X,Y \in \D$. We then have the following standard result~\cite[Corollaire~7]{guitart},~\cite[Proposition~1.2.2]{seal}.

\begin{prop}
	Under these assumptions, the Kleisli category $\Kl{T}$ becomes symmetric monoidal, with:
	\begin{itemize}
		\item the tensor product of objects being the one of $\D$,
		\item the tensor product of morphisms represented by $f : A \to TX$ and $g : B \to TY$ being represented by the composite
		\[
			\begin{tikzcd}
				A \otimes B \ar["f\otimes g"]{r} & TX \otimes TY \ar["\nabla"]{r} & T(X \otimes Y).
			\end{tikzcd}
		\]
	\end{itemize}
	Moreover, the standard inclusion functor $\D \to \Kl{T}$ is strict symmetric monoidal.
\end{prop}

Thanks to the final statement, we can also transport comonoid structures along the inclusion functor: if an object $X \in \D$ carries a distinguished comonoid structure, then so does the corresponding object $X \in \Kl{T}$. Moreover, if the monoidal unit $I$ is terminal in $\D$ and $TI \cong I$ (such a monad is also called \emph{affine}), then $I$ is terminal in $\Kl{T}$ as well. Thus we have the following:

\begin{cor}
	\label{affine_monoidal_monad}
	Let $\D$ be a Markov category and $T$ a symmetric monoidal affine monad on $\D$. Then the Kleisli category $\Kl{T}$ is again a Markov category in a canonical way.
\end{cor}

Concretely, the comonoid multiplication of an object $X \in \Kl{T}$ is represented by the composite morphism
\[
	\begin{tikzcd}
		X \ar["\cop_X"]{r} & X \otimes X \ar["\eta"]{r} & T(X \otimes X).
	\end{tikzcd}
\]

In the case where $\D$ is cartesian monoidal, \Cref{affine_monoidal_monad} has essentially been given by Golubtsov~\cite[Theorem~2]{golubtsov2} and has also been hinted at by Cho and Jacobs~\cite[p.~6/7]{cho_jacobs}. For example, the submonad of normalized measures of a \emph{measure category} in the sense of~\cite[Definition~4.2]{Setal} is one context in which the result applies.

\begin{ex}
	For $\D := \FinSet$ considered as a cartesian monoidal category, let $T$ be the covariant nonempty powerset monad, whose underlying functor is $TX := 2^X \setminus \{\emptyset\}$. Then $T$ comes equipped with a canonical symmetric monoidal structure $\nabla_{X,Y}$ which takes a subset of $X$ and subset of $Y$ and maps it to their cartesian product, which is a subset of $X \times Y$. Then \Cref{affine_monoidal_monad} applies, and the resulting Kleisli category $\Kl{T}$ is exactly the category of finite sets and multivalued functions from \Cref{setmulti}.
\end{ex}

\section{Example: measurable spaces and measurable Markov kernels}
\label{sec_giry}

In this section, we introduce $\Stoch$, the most general category of measurable Markov kernels between measurable spaces, as considered as a running example in the later sections, as well as its more well-behaved subcategory $\BorelStoch$ consisting of standard Borel spaces. $\Stoch$ contains as particular subcategories also many other Markov categories that may be of interest to probabilists and statisticians; see e.g.~\Cref{gauss_stoch}. It has previously been considered as a Markov category in~\cite[Example~2.5]{cho_jacobs} and with minor variations prior to that in~\cite[Section~8.1]{golubtsov2}.

Roughly speaking, $\Stoch$ is the generalization of $\FinStoch$ from \Cref{finstoch} beyond the finite case, facilitating the treatment of Markov kernels between arbitrary measurable spaces. This category can be considered to be \emph{the} paradigmatic example of a Markov category to which our formalism applies. $\Stoch$ can be obtained by applying the construction of \Cref{sec_mon_mond} to the \emph{Giry monad}. But let us start with a brief outline of the illustrious history of this category, or rather that part of its history that we know of.

It seems that the category $\Stoch$ was first considered by Lawvere~\cite{lawvere} in 1962, before Kleisli categories had been introduced in general~\cite{kleisli}. Lawvere used these ideas to develop aspects of a formal theory of stochastic processes, and seems to have envisioned synthetic probability theory, as partly developed in this paper, more generally. In 1965, and apparently independently of Lawvere's work, the category $\Stoch$ was also considered by \v{C}encov as the \emph{category of statistical decision rules}~\cite{cencov} in the context of information geometry, and \v{C}encov derived another description of this category~\cite[Theorem~5.2]{cencov_book} reminiscent of what we now know as the Kleisli category construction~\cite{jhht}. In 1966, a closely related category, in fact containing $\Stoch$ as a full subcategory, was introduced by Morse and Sacksteder~\cite{MS}, again apparently independently of Lawvere and \v{C}encov. The construction of $\Stoch$ as a Kleisli category was made explicit in 1982 by Giry~\cite[Section~4]{giry}, using what we now call the \emph{Giry monad} based only on Lawvere's ideas. A more contemporary expository account can be found e.g.~in Panangaden's book on labelled Markov processes~\cite[Chapter~5]{panangaden}\footnote{Panangaden works with \emph{subprobability} measures, meaning measures normalized to $\leq 1$. This does not change anything substantially.}. We will follow his account by spelling out directly what the Kleisli category is.

To begin, let $\Meas$ be the category of measurable spaces $(X,\Sigma_X)$ with measurable maps as morphisms. $\Meas$ becomes symmetric monoidal upon equipping it with the usual product of measurable spaces with product $\sigma$-algebras,
\[
	(X,\Sigma_X) \otimes (Y,\Sigma_Y) := (X\times Y, \Sigma_X \otimes \Sigma_Y),
\]
The Giry monad functor $P : \Meas \to \Meas$ assigns to every measurable space $(X,\Sigma_X)$ the set of probability measures $P(X,\Sigma_X)$, equipped with the smallest $\sigma$-algebra which makes the evaluation map
\[
	P(X,\Sigma_X) \longrightarrow [0,1], \qquad \mu \longmapsto \mu(S)
\]
measurable for every $S \in \Sigma_X$. It acts on measurable maps by taking every probability measure to its pushforward, which is easily seen to be a measurable operation itself.

We omit the construction of the unit and the multiplication of the Giry monad, referring to~\cite{giry,avery} for the details. Instead, we spell out directly what the Kleisli category $\Stoch := \Kl{P}$ looks like. The objects are again measurable spaces. A morphism $(X,\Sigma_X) \to (Y,\Sigma_Y)$ is given by a map
\[
	f : \Sigma_Y \times X \longrightarrow [0,1], \qquad (S,x) \longmapsto f(S|x),
\]
having the property that every $f(-|x) : \Sigma_Y \to [0,1]$ is a probability measure, and such that $f(S|-) : X \to [0,1]$ is measurable for every $S \in \Sigma_Y$. The vertical bar reminds us that we think of $f$ as a specification of probability measures conditional on any given $x \in X$. Composition with an analogous $g : \Sigma_Z \times Y \to [0,1]$ is given by a version of the Chapman--Kolmogorov equation~\cref{ck},
\beq
	\label{ck2}
	g\circ f \: : \: \Sigma_Z \times X \longrightarrow [0,1], \qquad (g\circ f)(T|x) := \int_{y \in Y} g(T|y) \, f(dy|x).
\eeq
Here, the integral exists since $g(T|-)$ is measurable by assumption, and trivially bounded since it is $[0,1]$-valued. The $\sigma$-additivity in $T$ follows from the $\sigma$-additivity of $g(-|y)$ by the dominated convergence theorem. The measurability in $x$ for fixed $T$ holds by the measurability assumption on $f$ whenever $g$ is a simple function; the definition of the Lebesgue integral then implies the general case, using the fact that the limit of a sequence of measurable functions is again measurable. For the proof that this indeed gives a category, see e.g.~\cite[Proposition~5.2]{panangaden}. Identity morphisms are given by
\[
	\id_X(S|x) = \begin{cases} 1 & \textrm{ if } x \in S, \\ 0 & \textrm{ otherwise}. \end{cases}
\]
In order to make $\Stoch = \Kl{P}$ into a Markov category~\cite[Example~2.5]{cho_jacobs}, we use \Cref{affine_monoidal_monad} and make $P$ into a symmetric monoidal monad via the maps
\beq
	\label{giry_products}
	P(X,\Sigma_X) \otimes P(Y,\Sigma_Y) \longrightarrow P( (X,\Sigma_X) \otimes (Y,\Sigma_Y) ), \qquad (\mu,\nu) \longmapsto \mu \otimes \nu.
\eeq
As usual in categorical probability, these maps implement the formation of product measures $\mu\otimes\nu$~\cite{FP}.

\begin{lem}
	\label{giry_affine_monoidal}
	These maps make the Giry monad into an affine symmetric monoidal monad.
\end{lem}

\begin{proof}
	It is clear that the Giry monad is affine. It is also straightforward to check that the maps~\cref{giry_products} are natural in $X$ and $Y$, as well as that all the equations hold that make $P$ into a symmetric monoidal monad. The only problem is that it is not so clear whether~\cref{giry_products} is a morphism of $\Meas$ in the first place, i.e.~whether it is measurable. We now focus on proving this.

	Since the $\sigma$-algebra on $P( (X,\Sigma_X) \otimes (Y,\Sigma_Y) )$ is the smallest one which makes the evaluation map on every $S \in \Sigma_X \otimes \Sigma_Y$ measurable, it is enough to show that the composite
	\[
		\begin{tikzcd}
			P(X,\Sigma_X) \otimes P(Y,\Sigma_Y) \ar{r} & P( (X,\Sigma_X) \otimes (Y,\Sigma_Y) ) \ar["{\mathrm{ev}_S}"]{r} & \left[0,1\right]
		\end{tikzcd}
	\]
	is jointly measurable for each fixed $S$. Concretely, this map is given by
	\[
		P(X,\Sigma_X) \otimes P(Y,\Sigma_Y) \longrightarrow [0,1],\qquad (\mu,\nu) \longmapsto (\mu\otimes\nu)(S).
	\]
	We prove that it is measurable using a standard application of the $\pi$-$\lambda$-theorem. Namely, consider the collection of all measurable $S$ which make this map measurable. Since limits of pointwise convergent sequences of measurable maps are again measurable, this set system is closed under countable disjoint unions; and it is clearly closed under complements, so that we have a $\lambda$-system. On the other hand, $\Sigma_X \otimes \Sigma_Y$ is itself the $\sigma$-algebra generated by the rectangle sets $T \times U$ with $T \in \Sigma_X$ and $U \in \Sigma_Y$, which itself form a $\pi$-system. Thus by the $\pi$-$\lambda$-theorem, it is now enough to prove the claim in the case $S = T \times U$, i.e.~to show that the map
	\[
		P(X,\Sigma_X) \otimes P(Y,\Sigma_Y) \longrightarrow [0,1], \qquad (\mu,\nu) \longmapsto \mu(T) \nu(U)
	\]
	is measurable for every $T \in \Sigma_X$ and $U \in \Sigma_Y$. But now this is because this map is equal to the composite
	\beq
		\label{product_square}
		\begin{tikzcd}
			P(X,\Sigma_X) \otimes P(Y,\Sigma_Y) \ar["\mathrm{ev}_T \otimes \mathrm{ev}_U"]{r} & \left[0,1\right] \otimes \left[0,1\right] \ar["\cdot"]{r} & \left[0,1\right],
		\end{tikzcd}
	\eeq
	where the second map is just multiplication, which is clearly measurable.
\end{proof}

Concretely, the comonoid maps in $\Stoch$ are given by
\[
	\cop_{(X,\Sigma_X)} : (\Sigma_X \otimes \Sigma_X) \times X \longrightarrow [0,1], \qquad (S \times T,x) \longmapsto \begin{cases} 1 & \textrm{ if } x \in A \cap B, \\ 0 & \textrm{ otherwise}. \end{cases}
\]
for any generating rectangle $A \times B$ in the product $\sigma$-algebra. Indeed this corresponds to assigning to every $x\in X$ the measure $\delta_{x,x}$ on $X\times X$.

Here is an observation which we will frequently use when instantiating our definitions and results in $\Stoch$:

\begin{lem}
	\label{stoch_equal}
	Two morphisms $f, g : A \to X_1 \otimes \ldots \otimes X_n$ in $\Stoch$ are equal if and only if
	\[
		f(S_1 \times \ldots \times S_n|a) = g(S_1 \times \ldots \times S_n|a)	
	\]
	for all $a \in A$ and $S_i \in \Sigma_{X_i}$.
\end{lem}

\begin{proof}
	Use the fact that the cylinder sets $S_1 \times \ldots \times S_n$ form a $\pi$-system which generates the product $\sigma$-algebra $\Sigma_{X_1} \otimes \ldots \otimes \Sigma_{X_n}$, and apply the $\pi$-$\lambda$-theorem.
\end{proof}

\newcommand{\negligible}[1]{\mathcal{N}_{#1}}	

Although we will consider measure-theoretic probability to take place in $\Stoch$ in the rest of this paper, there is evidence that $\Stoch$ is not the only or even the best Markov category for this purpose; see~\Cref{stoch_det,stoch_disint} for some of its counterintuitive and perhaps undesirable features. However, its full subcategory $\BorelStoch \subseteq \Stoch$, consisting of all standard Borel spaces or equivalently all Polish spaces together with measurable maps as morphisms, turns out to be much more well-behaved. Since the Giry monad restricts to Polish spaces, $\BorelStoch$ also arises as a Kleisli category via \Cref{affine_monoidal_monad}. As far as we are aware, the only thing left to be desired of $\BorelStoch$ is that the Kolmogorov extension theorem holds in this Markov category only with respect to countable products~\cite[Example~3.6]{FR}.

Finally, we present two other candidate categories which are of potential interest in the context of measure-theoretic probability, based on two related definitions of morphisms given by Sacksteder and Dawid, and which may relate to conditional expectations (see also~\Cref{cond_exp}). Readers who are less interested in these technical details may proceed to \Cref{sec_radon} without any loss.

If $(X,\Sigma_X)$ is a measurable space, then a \emph{$\sigma$-ideal} is a subcollection $\negligible{X} \subseteq \Sigma_X$ closed under countable unions and intersections with elements of $\Sigma_X$. The idea is that $\negligible{X}$ specifies a collection of sets which are considered negligibly small. If $(X,\Sigma_X,\negligible{X})$ and $(Y,\Sigma_Y,\negligible{Y})$ are two measurable spaces with $\sigma$-ideals, then their \emph{product} is defined by
\[
	(X,\Sigma_X,\negligible{X}) \otimes (Y,\Sigma_Y,\negligible{Y}) := (X \times Y, \Sigma_X \otimes \Sigma_Y, (\negligible{X} \times \Sigma_Y) \lor (\Sigma_X \times \negligible{Y})),
\]
where the indicated new $\sigma$-ideal is the one generated by all sets of the form $S \times T$ for $S \in \negligible{X}$ and $T \in \Sigma_Y$ or $S \in \Sigma_X$ and $T \in \negligible{Y}$. We would like this product to be the object part of a monoidal structure in both of the categories which we define next, although we have not yet been able to define a monoidal structure on the morphisms.

We will say that some property holds \emph{almost everywhere} if the set of points where it does not hold is in the specified $\sigma$-ideal.

\begin{ex}
	\label{sacksteder_stoch}
	Following Sacksteder~\cite[Section~2]{sacksteder}, a \emph{statistical operation} $(X,\Sigma_X,\negligible{X}) \to (Y,\Sigma_Y,\negligible{Y})$ is an equivalence class of maps
	\[
		f : \Sigma_Y \times X \longrightarrow [0,1], \qquad (S,y) \longmapsto f(S|y),
	\]
	which satisfy the following conditions:
	\begin{enumerate}
		\item\label{ssa} For given $S \in \Sigma_Y$, the value $f(S|x)$ may not be defined for every $x$, but is defined almost everywhere;
		\item $f(Y|x) = 1$ for almost every $x$;
		\item Given a sequence of disjoint sets $(S_n)_{n\in\N}$ in $\Sigma_X$, we have $f\left( \bigcup_n S_n | x\right) = \sum_n f(S_n|x)$ for almost every $x$;
		\item\label{ssd} For every $S \in \negligible{Y}$, we have $f(S|x) = 0$ for almost every $x$.
	\end{enumerate}
	Although this does not seem to have been considered by Sacksteder, it is natural for two such maps $f_1$ and $f_2$ to be considered equivalent, $f_1 \sim f_2$, if for every $S \in \Sigma_Y$ the equation $f_1(S|x) = f_2(S|x)$ holds for almost every $x$.

	We define composition of such statistical operations using the already familiar Chapman--Kolmogorov equation~\eqref{ck2}. Some straightforward measure-theoretic arguments show that this integral indeed makes sense, that the resulting composite again satisfies properties~\ref{ssa} to~\ref{ssd}, and that this composition operation respects the equivalence. Since the associativity of composition and the existence of identities are obvious, we can define the category of statistical operations in Sacksteder's sense, using equivalence classes of statistical operations as morphisms.

	In order to make this into a monoidal category, for every $f : A \to X$ and $g : B \to Y$ we would like to define their tensor product $f \otimes g : A \otimes B \to X \otimes Y$. On rectangles $S\times T$ for $S \in \Sigma_X$ and $T \in \Sigma_Y$, this tensor product should be a new statistical operation satisfying
	\[
		(f \otimes g)(S\times T|a,b) := f(S|a) \, g(T|b)
	\]
	whenever the right-hand side is defined. However, since neither $f(-|a)$ nor $g(-|b)$ is $\sigma$-additive in general, it is not clear to us whether this definition can be extended in any meaningful way from measurable rectangles to the product $\sigma$-algebra. In other words, it is not clear to us whether Sacksteder's category of statistical operations can be made into a symmetric monoidal category, let alone into a Markov category.
\end{ex}

For $(X,\Sigma_X,\negligible{X})$ a measurable space equipped with a $\sigma$-ideal, we write $L^\infty(X,\Sigma_X,\negligible{X})$, or more concisely $L^\infty(X)$, for the set of equivalence classes of bounded measurable functions $f : X \to \R$, where $f_1 \sim f_2$ if and only if $f_1(x) = f_2(x)$ for almost every $x$. Declaring $f \geq 0$ to hold if $f(x) \geq 0$ for almost all $x$ makes $L^\infty(X)$ into an ordered vector space.

\begin{ex}
	\label{dawid_stoch}	
	Following Dawid~\cite[Section~3]{dawid_operations}, a \emph{statistical operation} $(X,\Sigma_X,\negligible{X}) \to (Y,\Sigma_Y,\negligible{Y})$ is a map 
	\[
		F : L^\infty(Y) \to L^\infty(X)
	\]
	having the following properties:
	\begin{enumerate}
		\item $F$ is $\R$-linear;
		\item If $g \geq 0$ for $g \in L^\infty(Y)$, then also $F(g) \geq 0$ in $L^\infty(X)$;
		\item $F(1) = 1$;
		\item\label{nScott} If $(g_n)_{n\in\N}$ is a monotonically increasing sequence in $L^\infty(Y)$ whose supremum exists in $L^\infty(Y)$, then $F(\sup_n g_n) = \sup_n F(g_n)$.
	\end{enumerate}
	The idea is that $F$ encodes the operation of pulling back a random variable along the statistical operation. It is clear that these statistical operations can be composed, resulting in the category of statistical operations in Dawid's sense.
	
	As in the previous example, we run into a problem upon trying to define the tensor product of morphisms. If $F : L^\infty(X) \to L^\infty(A)$ and $G : L^\infty(Y) \to L^\infty(B)$ are two operators representing statistical operations $A \to X$ and $B \to Y$, then their tensor product $F \otimes G : L^\infty(A \times B) \to L^\infty(X \times Y)$ should be an operator which satisfies
	\[
		(F \otimes G)(f\otimes g) := F(f) \, G(g)
	\]
	for all $f \in L^\infty(X)$ and $g \in L^\infty(Y)$, where $f\otimes g \in L^\infty(X \times Y)$ is the product function given by $(f\otimes g)(x,y) := f(x) g(y)$, and similarly for the right-hand side\footnote{It is easy to see that the product $f\otimes g \in L^\infty(X \times Y)$ does not depend on the choice of representatives.}. Upon trying to extend this definition to all $h \in L^\infty(X \times Y)$, we run into the problem that $h$ may not have a representative as the supremum of an increasing sequence of functions of the form $\sum_{i=1}^n f_i \otimes g_i$. Hence we again do not know whether this category of statistical operations can be made into a symmetric monoidal category, let alone into a Markov category. 

	A possible way around this problem may be to restrict attention to those measurable spaces equipped with $\sigma$-ideals $(X,\Sigma_X,\negligible{X})$ which are such that $\Sigma_X / \negligible{X}$ is a complete Boolean algebra, and then strengthen condition~\ref{nScott} by replacing sequences by arbitrary directed sets (Scott continuity). Then it would be enough to write every $h \in L^\infty(X \times Y)$ as the directed supremum of a set of functions of the form $\sum_{i=1}^n f_i \otimes g_i$.
\end{ex}

Due to the problems with the definition of the tensor product of morphisms in the previous two examples, we have not yet been able to instantiate the definitions and results of \Cref{sec_det,furtheraxioms,cind,sec_as,suff,sec_basu,sec_bahadur} in these cases. We also do not know whether either of these categories arises as a Kleisli category via \Cref{affine_monoidal_monad}, nor do we know if and how these two categories are related.

\section{Example: compact Hausdorff spaces and continuous Markov kernels}
\label{sec_radon}

In measure theory, it is often useful to work with probability measures on Borel $\sigma$-algebras of suitable topological spaces. We here investigate one particularly clean way to do so, based on the \emph{Radon monad}. It involves $\CHaus$, the category of compact Hausdorff spaces and continuous maps, as perhaps the simplest nontrivial category of not necessarily finite topological spaces. We equip it with the cartesian monoidal structure.

In more detail, the Radon monad $R : \CHaus \to \CHaus$, as introduced by \'Swirszcz~\cite{swirszcz}, and studied further by Furber and Jacobs~\cite{furber_jacobs,furber_thesis}, is the canonical probability monad on compact Hausdorff spaces. Its underlying functor assigns to every $X \in \CHaus$ the set of Radon probability measures on $X$, equipped with the coarsest topology which makes the integration maps
\[
		RX \longrightarrow \R, \qquad \mu \longmapsto \int f \, d\mu
\]
continuous for every continuous $f : X \to \R$. As shown in~\cite{furber_jacobs}, the Kleisli category $\Kl{R}$ is opposite equivalent to the category of unital commutative C*-algebras with positive unital linear maps as the morphisms.

Concretely, a morphism from $X$ to $Y$ in $\Kl{R}$ is represented by a continuous map $f : X \to RY$, and composition is again given by the Chapman--Kolmogorov equation in the form~\cref{ck2}. In order to show that $\Kl{R}$ is also a Markov category, by \Cref{affine_monoidal_monad} it is enough to equip $R$ with the structure of a symmetric monoidal monad. We again want to do this by taking product measures,
\beq
	\label{radon_product}
	RX \times RY \longrightarrow R(X \times Y), \qquad (\mu,\nu) \longmapsto \mu \otimes \nu,
\eeq
but we now run into the problem that the Borel $\sigma$-algebra on $X \otimes Y$ is generally much larger than the product of the Borel $\sigma$-algebras. We therefore define the product measure $\mu \otimes \nu$ via the Riesz representation theorem by putting
\beq
	\label{radon_product2}
	\int_{X \times Y} f \, d(\mu \otimes \nu) := \int_X \int_Y f \, d\mu \, d\nu,
\eeq
leaving the necessary verifications of linearity and positivity to the reader. We now note that:

\begin{lem}
	These maps equip $R$ with the structure of a symmetric monoidal monad.
\end{lem}

\begin{proof}
	As in the proof of \Cref{giry_affine_monoidal}, the nontrivial part is to show that \cref{radon_product} is actually a morphism, i.e.~that it is continuous, and our argument will be structurally similar.
	
	To see this, it is enough to show that for any continuous $f : X \times Y \to \R$, the composite
	\[
		\begin{tikzcd}
			RX \times RY \ar{r} & R(X \times Y) \ar["\int f \, d(-)"]{r} & \R
		\end{tikzcd}
	\]
	is continuous, which is by definition given by
	\[
		(\mu,\nu) \longmapsto \int_X \int_Y f \, d\mu \, d\nu.
	\]
	If $f$ is of the form $f(x,y) = g(x) h(y)$ for continuous $g : X \to \R$ and $h : Y \to \R$, then it is easy to see that this holds by an argument analogous to the one around~\cref{product_square}. Moreover, the continuity clearly also holds for any linear combination of such functions; and since these linear combinations are also closed under multiplication and clearly separate points, the Stone--Weierstra\ss{} theorem implies that they are dense with respect to the supremum norm.

	Now let an arbitrary $f$ and $\eps > 0$ be given, and suppose that we want to find a neighbourhood $U$ of $(\mu,\nu)$ in $RX \times RY$ such that $\left| \int_X \int_Y f \, d\mu' \, d\nu' - \int_X \int_Y f \, d\mu \, d\nu \right| < \eps$ for all $(\mu',\nu') \in U$. We first choose an approximation $\left\| f - \sum_{i=1}^n g_i \otimes h_i \right\|_\infty < \eps/3$ for suitable $n \in \N$ and continuous $g_i : X \to \R$ and $h_i : Y \to \R$, which is possible by the previous paragraph. We next choose a neighbourhood $U$ such that the desired inequality holds for every $g_i \otimes h_i$ in place of $f$, and with quality of approximation $\eps/3n$ in place of $\eps$. This gives, for every $(\mu',\nu') \in U$,
	\begin{align*}
		\left| \int_X \int_Y f \, d\mu' \, d\nu' \right. & \left. - \int_X \int_Y f \, d\mu \, d\nu \right| \\[4pt]
		\leq & \quad \left| \int_X \int_Y \left( f - \sum_{i=1}^n g_i \otimes h_i \right) \, d\mu' \, d\nu' \right| \\[4pt]
		& + \left| \sum_{i=1}^n \int_X \int_Y (g_i \otimes h_i) \, d\mu' \, d\nu' - \sum_{i=1}^n \int_X \int_Y (g_i \otimes h_i) \, d\mu \, d\nu \right| \\[4pt]
		& + \left| \int_X \int_Y \left( f - \sum_{i=1}^n g_i \otimes h_i \right) \, d\mu \, d\nu \right| \\[4pt]
		\leq {} & \eps/3 + n\cdot \eps/3n + \eps/3 = \eps,
	\end{align*}
	as was to be shown.
\end{proof}

We therefore have that $\Kl{R}$ is a Markov category per \Cref{affine_monoidal_monad}. The copy maps are again represented by $x \mapsto \delta_{(x,x)}$.

\section{Example: Gaussian probability theory}
\label{gaussian}

We now consider a Markov category describing Gaussian probability theory. This involves formulas for the composition of Gaussian conditionals reproducing those of Lauritzen and Jensen~\cite[Section~4.4]{LJ}. As our previous two examples, it also has the flavour of a Kleisli category, although we conjecture that it does not actually arise in that way. A similar category was also considered by Golubtsov in~\cite[\S~6]{golubtsov_sls} and~\cite[Section~8.2]{golubtsov2}, who in addition claimed to have proven that it is not a Kleisli category.

\newcommand{\E}[1]{\mathbf{E}[#1]}
\newcommand{\Var}[1]{\mathbf{Var}[#1]}

So let $\Gauss$ be the monoidal category whose monoid of objects is $(\N,+)$, and whose morphisms $n\to m$ are tuples $(M,C,s)$ with $M\in \R^{m\times n}$, $C\in \R^{m\times m}$ positive semidefinite, and $s\in\R^m$; we think of this data as defining the conditional distribution of a random variable $Y\in\R^m$ in terms of $X\in\R^n$ as
\begin{equation}
\label{gausstoch}
	Y := MX + \xi,
\end{equation}
where $\xi$ is Gaussian noise independent of $X$ with mean $\E{\xi} := s$ and covariance matrix
\[
	\Var{\xi} = \E{\xi \xi^T} - \E{\xi} \E{\xi}^T = \E{(\xi - s)(\xi - s)^T} := C.
\]
If $C$ is invertible, then the distribution of $Y$ has a density with respect to the Lebesgue measure on $\R^m$ which is given by
\[
	\frac{1}{\sqrt{(2\pi)^m \mathrm{det}(C)}} \exp\left( -\frac{1}{2}(y-Mx-s)^T C^{-1} (y-Mx-s) \right).
\]
Before formally defining the composition of two morphisms
\[
	(M,C,s) : \ell \to m, \qquad (N,D,t) : m \to n,
\]
let us consider what it should represent. If $Y = M X + \xi$ and $Z = N Y + \eta$, then we get
\[
	Z = N M X + N \xi + \eta.
\]
When composing two morphisms in a category of Markov kernels, we assume that the noise added in each of them is independent of the other. Hence we expect
\[
	\E{N \xi + \eta} = N \, \E{\xi} + \E{\eta} = N s + t,
\]
and
\[
	\Var{N \xi + \eta} = N\, \Var{\xi} N^T + \Var{\eta} = N C N^T + D.
\]
Therefore what we want to define is the composition
\[
	(N,D,t) \circ (M,C,s) := (NM,NCN^T + D,Ns + t),
\]
where it is straightforward to check that this is indeed associative and has the obvious identity morphisms $(1_n,0,0)$. This composition operation is our version of the ``direct combination'' equations given at~\cite[Section~4.4]{LJ}, specialized to case where no discrete variables are present\footnote{It should not be difficult to combine $\FinStoch$ and $\Gauss$ into one Markov category which models both discrete variables and Gaussian ones as in~\cite{LJ}, but we have not yet done so.}. The tensor product of morphisms should correspond to acting on two random variables at the same time. This means that we want to take direct sums of all the data describing two arbitrary morphisms,
\[
	(M, C, s) \otimes (N, D, t) := (M \oplus N, C \oplus D, s \oplus t),
\]
since for independent random variables $\xi \in \R^m$ and $\eta \in \R^n$, the expectation and variance of $(\xi,\eta) \in \R^{m + n}$ are given by
\[
	\E{(\xi,\eta)} = (\E{\xi},\E{\eta}), \qquad \Var{(\xi,\eta)} = \left( \begin{matrix} \Var{\xi} & 0 \\ 0 & \Var{\eta} \end{matrix} \right).
\]
It is also straightforward to check that this indeed defines a (strict) symmetric monoidal category with the obvious symmetry isomorphisms. It may be interesting to note that forgetting the covariance matrix of every morphism, $(M,C,s) \mapsto (M,s)$, implements a strong monoidal functor from $\Gauss$ to the category of affine maps between Euclidean spaces.

Since the only morphism $n\to 0$ is the trivial one $(0,0,0)$, the monoidal unit $0 \in \N$ is terminal, so that $\Gauss$ is indeed semicartesian. Moreover, every object $n$ has a canonical comultiplication given by
\[
	\cop_n := \left( \left( \begin{matrix} 1_n \\ 1_n \end{matrix} \right), 0, 0 \right),
\]
which simply amounts to copying the value of a variable $X\in\R^n$, resulting in $(X,X)\in\R^{n + n}$. The necessary requirements to be a Markov category are again straightforward to check.

We briefly investigate the relation between $\Gauss$ and $\Stoch$. Unfortunately this is somewhat interdependent with the upcoming \Cref{cacom}, since the following statement involves the notion of Markov functor introduced there.

\begin{prop}
	\label{gauss_stoch}
	There is a faithful Markov functor $\Gauss \to \Stoch$ which takes $n \mapsto \R^n$, where $\R^n$ carries the Borel $\sigma$-algebra.
\end{prop}

\begin{proof}
	We have already described how to interpret a morphism $(M,C,s) : n \to m$ as a Markov kernel taking an $\R^n$-valued variable to an $\R^m$-valued variable. It is straightforward to verify that this defines a Markov functor $\Gauss \to \Stoch$, and that this functor is faithful.
\end{proof}

\begin{rem}
	As pointed out to us by Sam Staton, another Markov category with a similar flavour and significance for probability theory has been considered in~\cite{SSYAFR}. In more explicit categorical terms, the objects are lists of natural numbers $(n_1,\ldots,n_k)$, representing a disjoint union of measurable spaces $[0,1]^{n_1} + \ldots + [0,1]^{n_k}$, and the morphisms are those Markov kernels that are definable just using beta and Bernoulli distributions. The paper \cite{SSYAFR} then gives a description of this category in purely syntactic terms; for example, the comultiplications $\cop_{(n)} : [0,1]^n \to [0,1]^{2n}$ are given syntactically by
	\[
		(p_1,\ldots,p_n) \mid x : 2n \vdash x(p_1,\ldots,p_n,p_1,\ldots,p_n).
	\]
\end{rem}

\section{Example: diagram categories and stochastic processes}
\label{sec_diagrams}

In this section and the subsequent two, we present classes of examples of Markov categories arising from more general categorical constructions. Unfortunately, the construction of this section already requires the notion of deterministic morphism from~\Cref{defn_det} and the subcategory of deterministic morphisms $\C_\det \subseteq \C$ from \Cref{rem_det_cartesian}.

Let $\C$ be a Markov category and $\D$ an arbitrary small category. We then introduce another Markov category $\Fun(\D,\C)$. Its objects are defined to be functors $\D \to \C_\det$, which we think of as diagrams with shape $\D$ of deterministic morphisms in $\C$, and its morphisms are natural transformations between these functors (with not necessarily deterministic components). We use notation $X = (X_d)_{d\in \D}$ for such functors, emphasizing their interpretation as diagrams.

\begin{prop}
	\label{functor_cat}
	Under these assumptions, $\Fun(\D,\C)$ is again a Markov category with respect to the pointwise monoidal structure: for $X, Y \in \Fun(\D,\C)$, we put
	\[
		(X \otimes Y)_d := X_d \otimes Y_d,
	\]
	with the obvious action on morphisms of $\D$, and similarly for the tensor product of morphisms. The comultiplications are also defined pointwise,
	\beq
		\label{diag_comult}
		\left(\cop_X\right)_d := \cop_{X_d} \: : \: X_d \longrightarrow X_d \otimes X_d.
	\eeq
\end{prop}

The proof is straightforward; the assumption that all morphisms in the diagrams $(X_d)$ are deterministic is required in order to make the comultiplications~\eqref{diag_comult} into natural transformations.

As an interesting and highly relevant example, we now illustrate how \Cref{functor_cat} lets us instantiate the theory of Markov categories in the context of stochastic processes. Here, we take $\D := (\Z,\geq)$, the category whose objects are integers, with exactly one morphism $n \to m$ if and only if $n \geq m$ and no morphism otherwise. For example with $\C = \BorelStoch$, the functors $\D \to \C_\det$ can be identified with infinite diagrams of measurable spaces
\[
	\begin{tikzcd}
		\ldots \ar{r} & X_{n+1} \ar{r} & X_n \ar{r} & X_{n-1} \ar{r} & \ldots
	\end{tikzcd}
\]
where all the maps involved are measurable functions (\Cref{borelstoch_det}). We think of each component $X_n$ as the space of joint values of a stochastic process over all times $t \le n$, where the map $X_{n+1} \to X_n$ amounts to forgetting the value of the process at time $n + 1$. A morphism $I \to X$ in $\Fun(\D,\C)$ then assigns a probability measure to each of the measurable spaces $X_n$ in such a way that the projections $X_{n+1} \to X_n$ are measure-preserving; thus such a morphism exactly turns $X$ into a stochastic process. More generally, a morphism $f : A \to X$ may be interpreted as a stochastic process with control: $A_n$ is the space of values of control parameters over all times $t \le n$, and the component $f_n : A_n \to X_n$ assigns a joint distribution over all values of the stochastic process as a function of these control parameters. The required commutativity of the diagram
\[
	\begin{tikzcd}
		A_{n+1} \ar{r} \ar["f_{n+1}"]{d} & A_n \ar["f_n"]{d} \\
		X_{n+1} \ar{r} & X_n
	\end{tikzcd}
\]
then ensures that these joint distributions are compatible in exactly the way that one would expect: given the control parameters up to time $n+1$, the marginal of the associated joint distribution up to time $n$ coincides with the joint distribution associated to the control parameters up to time $n$.

It therefore follows that all of our definitions and results can be applied at the level of stochastic processes\footnote{Golubtsov's mention of ``dynamical nondeterministic information transformers [\ldots] and information interactions evolving in time''~\cite[p.~64]{golubtsov2} suggests that he may already have been aware of this possibility.}. For example, the definition of statistic from \Cref{def_statistic} amounts to the following: a statistic consists of another sequence of spaces $(T_n)_{n \in \Z}$ together with a sequence of deterministic maps $s_n : X_n \to T_n$ such that the diagrams
\[
	\begin{tikzcd}
		X_{n+1} \ar{r} \ar["s_{n+1}"]{d} & X_n \ar["s_n"]{d} \\
		T_{n+1} \ar{r} & T_n
	\end{tikzcd}
\]
commute. This states that for every given value of $X_n$, the values of the statistic $T_n$ can be computed from the value of the process up to time $n+1$ by first computing the values of the extended statistic $T_{n+1}$ and then projecting, or by marginalizing the process first to time up to $n$ and then applying the statistic there. Lauritzen's notion of \emph{transitivity} of a sequence of statistics~\cite[Definition~2.1]{lauritzen_book2} is closely related.

Thus we can apply the upcoming results of \Cref{suff,sec_basu,sec_bahadur} to the relevant functor category and thereby obtain results about this type of statistic living in time.

\begin{rem}
	There are many conceivable variations on this basic idea. The extension from discrete to continuous time is fairly obvious. We can also consider \emph{stationary} stochastic processes as follows. Let now $\D$ be the category consisting of one single object, such that its monoid of endomorphisms is $(\Z,+)$. Then the objects of $\Fun(\D,\BorelStoch)$ are measurable spaces equipped with an action of $\Z$, which makes them into dynamical systems. The morphisms $I \to X$ are then easily seen to be in canonical bijection with stationary measures, turning $X$ into a measure-preserving dynamical system.
\end{rem}

\begin{rem}
	The developments of this section are not sufficiently expressive to be turned into a formal theory of stochastic processes, because the latter would clearly require more structure: the theory of Markov categories itself does not provide any way to talk about the temporal aspects. Nevertheless, we believe that it may be possible to develop a formal theory of stochastic processes by equipping Markov categories with extra structure, such as an action of a functor $\C \to \C$ which implements time translation, or more generally a group action on $\C$.
\end{rem}

\section{Example: hypergraph categories}
\label{sec_hypergraph}

Let $\D$ be a hypergraph category~\cite{carboni,kissinger,fong,FS}. This means in particular that each object $X \in \D$ comes equipped with a distinguished comultiplication $\cop_X : X \to X \otimes X$ and counit morphism $\del_X : X \to I$ satisfying~\cref{delcopyAB}. Now let $\C$ be the subcategory consisting of all morphisms which preserve the counit. In other words, $f : X \to Y$ is in $\C$ if and only if $\del_Y \circ f = \del_X$. Then it is not hard to see that $\C$ carries the structure of a Markov category.

\begin{ex}
A good example is $\FinRel$, the category of finite sets and relations, with its usual monoidal structure given by the cartesian product and the canonical hypergraph structure where the distinguished relation $X^n \to X^m$ relates a tuple $(x_1,\ldots,x_n)$ to a tuple $(x'_1,\ldots,x'_m)$ if and only if $x_1 = \ldots = x_n = x'_1 = \ldots = x'_m$, implementing both ``copying'' as a morphism $X \to X \otimes X$ and ``testing for equality'' as a morphism $X \otimes X \to I$.

Restricting $\FinRel$ to the counit-preserving morphisms recovers exactly $\FinSetMulti$ (\Cref{setmulti}).
\end{ex}

\begin{ex}
	A closely related example is the hypergraph category $\mathsf{Mat}(R)$ for a semiring $(R,+,\cdot)$. Its objects are finite sets, and morphisms $X \to Y$ are matrices $(f_{xy})_{x\in X, y\in Y}$ with $f_{xy} \in R$, which compose by matrix multiplication. The monoidal structure is given by cartesian product of sets and tensor product of matrices. If $R$ is the Boolean semiring $R = \{0,1\}$ with $1 + 1 := 1$, then $\mathsf{Mat}(R) \cong \FinRel$, so that this generalizes the previous example.

	Restricting to the counit-preserving morphisms results in a Markov category, namely the category of ``stochastic matrices'' with entries in $R$. In particular for $R = \R_+$, we obtain the category of stochastic matrices, which is equivalent to $\FinStoch$ (\Cref{finstoch}). Taking $R := ([0,1],\max,\min)$ or $R := ([0,1],\max,\cdot)$ recovers Golubtsov's two categories of \emph{fuzzy information transformers}~\cite[Section~8.5]{golubtsov2}, modulo the minor difference that our objects are merely finite instead of arbitrary sets.
\end{ex}

While the idea of imposing normalization in this way by restricting to the subcategory of counit-preserving morphisms is not new and can be done in somewhat greater generality (see e.g.~\cite[Section~7]{cho_jacobs}), we believe that applying this procedure in the particular case of hypergraph categories, such as the ones of~\cite{fong}, has the potential to produce further examples of Markov categories in which our results will have interesting instantiations.

\section{Example: categories of commutative comonoids}
\label{comons_gen}

Let $\D$ be a symmetric monoidal category, and let $\C$ be the category of commutative comonoids in $\D$ with morphisms those maps which preserve the counits. Hence an object in $\C$ is by definition a triple $(X,\mu_X,e_X)$ with $\mu_X : X \to X \otimes X$ and $e_X : X \to I$ in $\D$ satisfying the usual equations, while a $\C$-morphism $(X,\mu_X,e_X) \to (Y,\mu_Y,e_Y)$ is a $\D$-morphism $f : X \to Y$ such that $e_Y = e_X f$. The usual tensor product of commutative comonoids makes $\C$ into a symmetric monoidal category, where the symmetry is also the one inherited from $\D$. Every object $(X,\mu_X,e_X) \in \C$ now has a canonical commutative comonoid structure not only in $\D$, but even in $\C$, given by the structure maps $\cop_X := \mu_X$ and $\del_X := e_X$. We have therefore sketched the proof of:

\begin{prop}
	\label{comons}
	If $\D$ is a symmetric monoidal category, then the category of commutative comonoids in $\D$ with counit-preserving morphisms is a Markov category.
\end{prop}

One use of Markov categories of this type is that they can provide counterexamples to plausible-sounding conjectures, as for example in \Cref{isononiso,noncausal}.

\section{Deterministic morphisms and a strictification theorem}
\label{sec_det}

We now return to the development of the general theory, frequently illustrating it in some of the previous examples.

The reader may have noticed that \Cref{markov_cat} postulates the naturality of $\del$ but not of $\cop$, for which it would take the form~\cref{comult_natural}. Indeed, the reason is that this naturality does not hold in the examples of interest to us, such as $\FinStoch$ from \Cref{finstoch}, or more generally $\Stoch$, the example of measurable Markov kernels between measurable spaces from \Cref{sec_giry}. The simple intuitive reason is that applying a Markov kernel independently to two copies of a variable is different than applying the map first and then copying. In fact, those morphisms for which the naturality holds are of a very special kind:

\begin{defn}
	\label{defn_det}
	A morphism $f : X \to Y$ in a Markov category is \emph{deterministic} if it respects the comultiplication,
	\beq
		\label{comult_natural}
		\input{multiplication_natural.tikz}
	\eeq
\end{defn}

This definition goes back to the characterization of functions as comonoid-preserving relations given by Carboni and Walters~\cite[Lemma~2.5]{CW} and has also been used in the present form by Fong~\cite[Lemma~5.2]{causaltheories}. Deterministic morphisms have also been investigated as \emph{copyable} morphisms in the program semantics literature~\cite{thielecke,fuhrmann}. In the specific context of Markov kernels, our definition may first have been used in~\cite[Section~3.3]{CPP}, although there also are strong similarities with Golubtsov's earlier work~\cite[Eq.~(2)]{golubtsov2}.

In terms of our notion of conditional independence to be introduced in \Cref{ciprocdef}, this can also be understood as saying that the morphisms on the right-hand side displays the conditional independence $\condindproc{Y}{Y}{X}$. Intuitively, we can phrase this as saying that the random variable $f$ is independent of itself, which is a standard equivalent way to talk about deterministic events.

We denote the collection of deterministic morphisms by $\C_\det \subseteq \C$. As we will show in~\Cref{rem_det_cartesian}, $\C_\det$ is a symmetric monoidal subcategory of $\C$.

\begin{rem}
	\label{comon_det}
	The commutative comonoid axioms of \Cref{markov_cat} imply that the structure maps $\del_X : X \to I$ and $\cop_X : X \to X \otimes X$ are themselves deterministic.
\end{rem}

\begin{ex}
	\label{finstoch_det}
	It is easy to see that in $\FinStoch$, the deterministic maps are exactly the usual ones, i.e.~those stochastic matrices $f$ where each entry $f_{xy}$ is either $0$ or $1$. It follows that the deterministic morphisms form a symmetric monoidal subcategory which is precisely the image of the inclusion $\FinSet \to \FinStoch$.
\end{ex}

\begin{ex}
	\label{stoch_det}
	Let us consider deterministic morphisms in $\Stoch$, the category of measurable spaces and Markov kernels (\Cref{sec_giry}). For $f : (X,\Sigma_X) \to (Y,\Sigma_Y)$, the left-hand side of the determinism equation is represented by the map
	\[
		(\Sigma_Y \otimes \Sigma_Y) \times X \longrightarrow [0,1], \qquad (S \times T,x) \longmapsto f(S|x) f(T|x),
	\]
	where we have specified the map only on generating rectangles $S \times T \in (\Sigma_Y \otimes \Sigma_Y)$, which is enough. Similarly, the right-hand side of the determinism equation is
	\[
		(\Sigma_Y \otimes \Sigma_Y) \times X \longrightarrow [0,1], \qquad (S \times T,x) \longmapsto f(S \cap T|x).
	\]
	So if the left-hand side and right-hand side are equal, then we must have $f(S|x)^2 = f(S|x)$ for all $S \in \Sigma_Y$, which implies that every measure $f(-|x) : \Sigma_Y \to [0,1]$ is actually $\{0,1\}$-valued. Conversely, \Cref{stoch_equal} and additivity of $f$ in the first argument show that this is enough. Hence the deterministic morphisms in $\Stoch$ are precisely those Markov kernels which land in $\{0,1\}$-valued measures.
	
	In general, a $\{0,1\}$-valued measure does \emph{not} need to be a delta measure; a standard example is given by the countable-cocountable $\sigma$-algebra on an uncountable set $Y$, with a $\{0,1\}$-valued measure which assigns measure $1$ to precisely the cocountable sets. So although the deterministic morphisms are again precisely those Markov kernels which do not involve any randomness, the subcategory of deterministic morphisms does not coincide with $\Meas$, the category of measurable spaces: the induced functor from $\Meas$ to the deterministic morphisms in $\Stoch$ is not full. And it is not even faithful either, although for a more stupid reason: if $\Sigma_Y$ does not separate the points of $Y$, then there will be different measurable maps $1 \to Y$ whose induced delta measures $\Sigma_Y \to \{0,1\}$ are the same. As is well-known, both pathologies can be avoided by imposing suitable restrictions on the measurable spaces involved. Namely if $\Sigma_Y$ is countably generated and separates the points of $Y$, then the deterministic morphisms $(X,\Sigma_X) \to (Y,\Sigma_Y)$ are exactly the measurable maps, as can be seen as follows. The map $y \mapsto \delta_y$ identifies the elements of $Y$ with the $\{0,1\}$-valued probability measures $\Sigma_Y \to [0,1]$ by~\cite[p.~14]{raorao}, and the measurability condition on $f(S|-) : X \to [0,1]$ for $S \in \Sigma_Y$ then becomes precisely the requirement that
	\[
		\{ x\in X \mid f(S|x) = 1 \}\, \in\, \Sigma_X,
	\]
	which amounts to the condition that the representing map $X \to Y$ must be measurable. Hence the functor from $\Meas$ to deterministic morphisms in $\Stoch$ is fully faithful on morphisms into $(Y,\Sigma_Y)$ as soon as $\Sigma_Y$ is countably generated and separating, but in general neither full nor faithful. We take this as evidence for the hypothesis that $\Stoch$ is not actually the ``best'' Markov category for measure-theoretic probability. (See also \Cref{sacksteder_stoch,dawid_stoch}.)
\end{ex}

\begin{ex}
	\label{borelstoch_det}
	The situation is much better in the full subcategory $\BorelStoch \subseteq \Stoch$. Thinking of a standard Borel space as a Polish space, noting that every probability measure on a Polish space has a support implies that a $\{0,1\}$-valued measure is indeed a delta measure. Furthermore using the fact that the relevant Borel $\sigma$-algebras separate points, it follows that the subcategory of deterministic morphisms is isomorphic to the category of standard Borel spaces and measurable maps.
\end{ex}

\begin{rem}
	Consider more generally the setting of \Cref{affine_monoidal_monad}. There, it is easy to see that the image of every deterministic morphism in $\D$ is also deterministic in $\Kl{T}$. But as we have seen in \Cref{stoch_det}, the functor $\D \to \Kl{T}$ need not be faithful, and it may also not be full on deterministic morphisms. In other words, deterministic morphisms and morphisms of $\D$ must be carefully distinguished.
\end{rem}

\begin{ex}
	\label{gauss_det}
	In $\Gauss$, a morphism $(M,C,s) : m \to n$ is deterministic if and only if the morphism
	\[
		\left( \left( \begin{matrix} 1_n \\ 1_n \end{matrix} \right), 0, 0 \right) \circ (M,C,s) = \left( \left( \begin{matrix} M \\ M \end{matrix} \right), \left( \begin{matrix} C & C \\ C & C \end{matrix} \right), \left( \begin{matrix} s \\ s \end{matrix} \right) \right)
	\]
	is equal to the morphism
	\[
		\left( (M,C,s) \otimes (M,C,s) \right) \circ \left( \left( \begin{matrix} 1_m \\ 1_m \end{matrix} \right), 0, 0 \right) =  \left( \left( \begin{matrix} M \\ M \end{matrix} \right), \left( \begin{matrix} C & 0 \\ 0 & C \end{matrix} \right), \left( \begin{matrix} s \\ s \end{matrix} \right) \right).
	\]
	Therefore $(M,C,s)$ is deterministic if and only if $C = 0$, which indeed means if and only if there is no randomness involved.
\end{ex}

\begin{ex}
	In a category of commutative comonoids and counit-preserving morphisms (\Cref{comons}), the deterministic morphisms are precisely the comonoid homomorphisms.
\end{ex}

Let us return to the general theory. Deterministic morphisms in a Markov category have the following in common with epimorphisms in any category:

\begin{lem}
	\label{deterministic_composition}
	Let $f : X \to Y$ and $g : Y \to Z$. Then:
	\begin{enumerate}
		\item\label{det_comp} If $f$ and $g$ are deterministic, then so is $gf$.
		\item\label{2of3} If $gf$ is deterministic and $f$ is a deterministic epimorphism, then $g$ is deterministic as well.
		\item\label{inverseisos} If $f$ and $g$ are inverse isomorphisms and if one of them is deterministic, then so is the other.
	\end{enumerate}
\end{lem}

\begin{proof}
	We show~\ref{2of3}, leaving the even simpler proofs of~\ref{det_comp} and~\ref{inverseisos} to the reader. We compute
	\[
		\input{2of3.tikz}
	\]
	and now use the epimorphism property to cancel $f$.
\end{proof}

We dedicate the remainder of this section to formal categorical considerations. These address some important technical issues which turn out to be surprisingly subtle. Readers who are primarily interested in our treatment of conditional independence and statistics instead can safely move on to \Cref{furtheraxioms}.

The following observation is the main reason for why categorical subtleties arise.

\begin{rem}
	\label{isononiso}
	There are Markov categories $\C$ which contain isomorphisms that are not deterministic.

	We can build concrete examples of this phenomenon based on the construction of Markov categories $\C$ as categories of comonoids in symmetric monoidal categories $\D$ via \Cref{comons}. The idea is the same object $X \in \D$ may support two comonoid structures with the same counit $e_X : X \to I$ but different comultiplications 
	\[
		\mu_X,\,\mu'_X \: : \: X \to X \otimes X,
	\]
	resulting in a non-deterministic isomorphism
	\[
		\id_X : (X,\mu_X,e_X) \to (X,\mu'_X,e_X).
	\]
	Here, it may or may not be the case that the two comonoids $(X,\mu_X,e_X)$ and $(X,\mu'_X,e_X)$ are isomorphic in some other way.

	More concretely, $\D$ may be the category of finite-dimensional vector spaces over a field $k$ equipped with the usual tensor product of vector spaces. Then since this $\D$ is monoidally equivalent to its own opposite, we can identify $\C$ with the opposite of the category of finite-dimensional commutative $k$-algebras together with $k$-linear unital maps. If we restrict to the $k$-algebras of the form $k^n$ (with componentwise operations), then our $\D$ becomes equivalent to the category of ``stochastic matrices'' with entries in $k$, where a stochastic matrix is $(f_{xy})_{x\in X,y\in Y} \in k^{X\times Y}$ with $\sum_y f_{xy} = 1$ as in \Cref{finstoch}. Equivalently, this $\D$ is the full subcategory of the Kleisli category of the $k$-affine space monad on $\Set$ with finite sets as objects. But without the restriction to $k$-algebras of this form, the category is significantly larger, and contains for example algebras like $k[X]/\langle X^2 \rangle$. This algebra is isomorphic to $k^2$ in $\C$, even in a number of different ways, but none of these isomorphisms is deterministic because $k[X]/\langle X^2\rangle$ and $k^2$ are not isomorphic as $k$-algebras.
\end{rem}

\begin{rem}
	We expect that the remainder of this section will also apply, after some straightforward modifications, in the context of \emph{hypergraph categories}~\cite{kissinger,fong,FS}, and perhaps more generally to symmetric monoidal categories \emph{supplied} with extra structure~\cite{FS_new}.
\end{rem}

\begin{lem}
	\label{structure_det}
	The subcategory of deterministic morphisms $\C_\det \subseteq \C$ is closed under tensor product and contain all structure maps, meaning the associators, unitors, $\del$, $\cop$, and $\swap$.
\end{lem}

See the proof of~\cite[Theorem~3.12]{FS_new} for an independently obtained similar result in a more general setup for symmetric monoidal categories with extra ``bells and whistles''.

\begin{proof}
	The first claim is a direct consequence of~\cref{delcopyAB}. That the structure maps are deterministic is easiest to see for every $\del_X$ and $\cop_X$ (\Cref{comon_det}). For $\swap_{X,Y} : X \otimes Y \to Y \otimes X$, we only need~\cref{delcopyAB} and the naturality of the symmetry isomorphisms,
	\[
		\input{swap_deterministic.tikz}
	\]
	The unitors are deterministic thanks to~\cref{copydiag} applied with one of the two tensor factors being $I$ together with $\cop_I$ being equal to the unitor $I \cong I \otimes I$. Finally, the problem is trickiest for an associator $(X\otimes Y) \otimes Z \stackrel{\cong}{\longrightarrow} X \otimes (Y \otimes Z)$. For this case, we denote the composite vertical isomorphism $(X \otimes Y) \otimes (X \otimes Y) \longrightarrow (X \otimes X) \otimes (Y \otimes Y)$ of~\cref{copydiag} by $\mathsf{gr}_{X,Y}$. Writing ``$;$'' instead of ``$\otimes$'' for brevity, we consider the diagram
	\[
		\begin{tikzcd}[column sep=1.2cm,row sep=1.2cm]
			(X ; Y) ; Z \ar[equal]{d} \ar["\cop_{(X; Y) ; Z}"]{rrr} & & & ( (X ; Y) ; Z) ; ( (X; Y); Z ) \ar["\mathsf{gr}_{(X;Y),Z}"]{d} \\
			(X ; Y) ; Z \ar[equal]{d} \ar["\cop_{X;Y} ; \id"]{rr} & & ( (X; Y) ; (X ; Y) ) ; Z  \ar["\id ; \cop_Z"]{r} \ar["\mathsf{gr}_{X,Y} ; \id"]{d} & ( (X; Y) ; (X; Y) ) ; (Z ; Z) \ar["\mathsf{gr}_{X,Y} ; \id"]{d} \\
			(X ; Y) ; Z \ar["(\id ; \cop_Y) ; \id"]{r} \ar["\cong"]{d} & ( X ; (Y ; Y)) ; Z \ar["(\cop_X ; \id) ; \id"]{r} \ar["\cong"]{d} & ( (X ; X) ; (Y ; Y) ) ; Z \ar["\cong"]{d} \ar["\id ; \cop_Z"]{r} & ((X ; X) ; (Y ; Y)) ; (Z ; Z) \ar["\cong"]{d} \\
			X ; (Y ; Z) \ar[equal]{d} \ar["\id ; (\cop_Y ; \id)"]{r} & X ; ( (Y ; Y) ; Z) \ar[equal]{d} \ar["\cop_X ; \id"]{r} & (X ; X) ; ( (Y ; Y) ; Z) \ar["\id;(\id;\cop_Z)"]{r} & (X ; X) ; ( (Y; Y) ; (Z ; Z)) \ar[equal]{d} \\
			X ; (Y ; Z) \ar["(\id ; \cop_Y) ; \id"]{r} \ar[equal]{d} & X ; ( (Y; Y) ; Z) \ar["\id ; (\id; \cop_Z)"]{r} & X ; ( (Y ; Y) ; (Z ; Z)) \ar["\id;\mathsf{gr}^{-1}_{Y,Z}"]{d} \ar["\cop_X;\id"]{r} & (X ; X) ; ( (Y; Y) ; (Z ; Z)) \ar["\id;\mathsf{gr}^{-1}_{Y,Z}"]{d} \\
			X ; (Y ; Z) \ar[equal]{d} \ar["\id ; \cop_{Y;Z}"]{rr} & & X ; ( (Y ; Z) ; (Y ; Z) ) \ar["\cop_X ; \id"]{r} & (X ; X) ; ( (Y; Z) ; (Y ; Z)) \ar["\mathsf{gr}_{X;(Y;Z)}"]{d} \\
			X ; (Y ; Z) \ar["\cop_{X;(Y;Z)}"]{rrr} & & & (X ; (Y ; Z)) ; (X ; (Y ; Z))
		\end{tikzcd}
	\]
	Each square commutes either by standard properties of symmetric monoidal categories, including naturality of the associator, or as an instance of \cref{copydiag}. Note that the bottom three rows of the diagram are somewhat similar to the top three. Commutativity of this diagram establishes the claim: the left vertical composite is obviously just the associator, while the right vertical composite is the tensor product of the associator with itself by the coherence theorem for symmetric monoidal categories, where applying this requires one to distinguish the two factors of $X$ on the right, and likewise for $Y$ and $Z$, noting that their order is the same at the bottom as it is on top.
\end{proof}

\begin{rem}
	\label{rem_det_cartesian}
	Since we already know that the deterministic morphisms are closed under composition and include all structure maps, \Cref{structure_det} therefore concludes that they form a symmetric monoidal subcategory $\C_\det \subseteq \C$. Using the well-known fact that having a natural comonoid structure on every object of a symmetric monoidal category implies that the monoidal structure is cartesian (e.g.~\cite{HV}), we see that $\C_\det$ is cartesian monoidal. We will make judicious use of this observation below. It provides the basis for relating our \Cref{markov_cat} to Golubtsov's formulation in terms of cartesian monoidal structure~\cite{golubtsov2}.
	
	Note also that every monoidal product $X\otimes Y$ in $\C$, when equipped with the two projections
	\[
		\id_X \otimes \del_Y : X\otimes Y \longrightarrow X, \qquad \del_X \otimes \id_Y : X\otimes Y \longrightarrow Y,
	\]
	always satisfies the \emph{existence} part of the universal property of a product by~\cref{cond_prod_proc}. The uniqueness part hinges on the naturality of $\cop$.
\end{rem}

\begin{defn}
	\label{cacom}
	The collection of Markov categories makes up a 2-category $\Markov$, where:
	\begin{itemize}
		\item objects are Markov categories;
		\item $1$-morphisms $F : \C \to \D$ are \emph{Markov functors}, by which we mean strong symmetric monoidal functors which preserve the comonoids in the sense that the diagram
			\beq
				\label{comonoidfun}
				\begin{tikzcd}
					& FX \ar[swap,"\cop_{FX}"]{dl} \ar["F(\cop_X)"]{dr} \\
					FX \otimes FX \ar["\cong"]{rr} & & F(X \otimes X)
				\end{tikzcd}
			\eeq
			commutes for every $X \in \C$, where the horizontal arrow is the relevant structure morphism of $F$;
		\item $2$-morphisms $t : F \Rightarrow G$ for $F,G : \C \to \D$ are monoidal natural transformations with deterministic components.
	\end{itemize}
	A Markov functor which is faithful is also called a \emph{Markov embedding}.
\end{defn}

For both $1$-morphisms and $2$-morphisms, one might also expect the obvious compatibility conditions with the counits to be required. These hold automatically due to the terminality of the monoidal unit, and we therefore do not spell them out. Another relevant observation is that the monoidal structure isomorphisms $FX \otimes FY \stackrel{\cong}{\longrightarrow} F(X \otimes Y)$ are automatically deterministic~\cite[Theorem~4.6]{FS_new}.

\begin{lem}
	Let $F : \C \to \D$ be a comonoid functor. Then $F$ takes deterministic morphisms to deterministic morphisms, giving a functor $F_\det : \C_\det \to \D_\det$. Morever, $F_\det$ preserves finite products.
\end{lem}

\begin{proof}
	Straightforward.		
\end{proof}

We say that $F : \C \to \D$ is a \emph{comonoid equivalence} if it is an equivalence in the 2-category $\Markov$. The following result provides a convenient way of recognizing comonoid equivalences, analogous to the standard fact that a functor is an equivalence of categories if and only if it is fully faithful and essentially surjective.

\begin{prop}
	\label{equiv_crit}
	A symmetric monoidal functor $F : \C \to \D$ is a comonoid equivalence if and only if it is an equivalence of categories, and for every object $Y \in \D$ there is $X \in \C$ and an isomorphism $i : Y \stackrel{\cong}{\to} FX$ which is deterministic, meaning that the diagram
	\[
		\begin{tikzcd}
			Y \ar{r}{i}[swap]{\cong} \ar["\cop_Y",swap]{d} & FX \ar["\cop_{FX}"]{d} \\
			Y \otimes Y \ar{r}{i\otimes i}[swap]{\cong} & FX \otimes FX
		\end{tikzcd}
	\]
	commutes.
\end{prop}

\begin{proof}
	If $F$ is an equivalence in $\Markov$ with inverse equivalence $G : \D \to \C$, then we can take $X := GY$, and the square commutes because the given natural isomorphism $FG \cong \id_\D$ is required to have deterministic components.

	Conversely, choose for every object $Y \in \D$ an object $X \in \C$ together with $i : Y \stackrel{\cong}{\to} FX$ which makes the square commute. Then it is a standard fact that these choices determine an equivalence $g : \D \to \C$ sending every $Y \in \D$ to the corresponding $X \in \C$, and being right adjoint to $F$ via $\eps : FG \cong \id_\D$ and $\eta : \id_\C \cong GF$~\cite[Theorem~IV.4.1]{maclane}. Moreover, $G$ has a unique monoidal structure which makes it into a monoidal equivalence weakly inverse to $F$~\cite[4.4.2.1]{SR}. Now the natural isomorphism $\eps : FG \cong \id_\D$ has deterministic components by construction. We show that same holds true for $\eta : \id_\C \cong GF$. Since $F$ is faithful, it is enough to prove commutativity of the relevant diagrams after application of $F$. In other words, we only need to show that
	\[
		\begin{tikzcd}[column sep=1.5cm]
			FX \ar[swap,"\cop_{FX}"]{d} \ar["F\eta_X"]{r} & FGFX \ar["\cop_{FGFX}"]{d} \\
			FX \otimes FX \ar["\cong",swap]{d} \ar["F\eta_X \otimes F\eta_X"]{r} & FGFX \otimes FGFX \ar["\cong"]{d} \\
			F(X \otimes X) \ar["F(\eta_X \otimes \eta_X)"]{r} & F(GFX \otimes GFX) 
		\end{tikzcd}
	\]
	commutes for every $X \in \C$. Since the lower square commutes by naturality, it remains to be shown that $F\eta_X$ is deterministic. But this is because $F\eta_X$ is inverse to $\eps_{FX}$ by the triangle identities for adjoints, and we already know that $\eps_{FX}$ is deterministic.
	
	We still need to prove that $G$ preserves the comonoids. We first show this for objects of the form $FX \in \D$, where we have
	\[
		\begin{tikzcd}[column sep=1.5cm]
			GFX \ar[swap,"G(\cop_{FX})"]{d} \ar[equal]{r} & GFX \ar[swap,"GF(\cop_X)"]{d} \ar["\eta^{-1}_X"]{r} & X \ar["\cop_X"]{d} \ar["\eta_X"]{r} & GFX \ar["\cop_{GFX}"]{d} \\
			G(FX \otimes FX) \ar["G(\cong)"]{r} & GF(X \otimes X) \ar["\eta^{-1}_{X\otimes X}"]{r} & X \otimes X \ar["\eta_X \otimes \eta_X"]{r} & GFX \otimes GFX 
		\end{tikzcd}
	\]
	where the top composite is $\id_{GFX}$ as desired, and the bottom composite is the desired structure isomorphism $G(FX \otimes FX) \cong GFX \otimes GFX$ thanks to the monoidality of $\eta$. The left square commutes since $F$ is a comonoid functor, the middle square by naturality of $\eta^{-1}$, and the right one since $\eta_X$ is deterministic. The general case reduces to this one via $i : Y \stackrel{\cong}{\to} FX$ and the diagram
	\[
		\begin{tikzcd}[column sep=1.7cm]
			GY \ar[swap,"G(\cop_Y)"]{d} \ar["Gi"]{r} & GFX \ar["G(\cop_{FX})"]{d} \ar[equal]{r} & GFX \ar["\cop_{GFX}"]{d} \ar["G(i^{-1})"]{r} & GY \ar["\cop_{GY}"]{d} \\
			G(Y \otimes Y) \ar["G(i\otimes i)"]{r} & G(FX \otimes FX) \ar["\cong"]{r} & GFX \otimes GFX \ar{r} \ar["G(i^{-1}) \otimes G(i^{-1})"]{r} & GY \otimes GY
		\end{tikzcd}
	\]
	where the top composite is simply the desired $\id_{GY}$, and the bottom composite is the desired structure isomorphism $G(Y \otimes Y) \cong GY \otimes GY$. The left square is the $G$-image of a commuting squares by the assumption that $i$ is deterministic, the middle square commutes as per the above, and the right square because $G(i^{-1})$ is deterministic by virtue of being equal to $G\eps_Y$, which is the inverse of the deterministic morphism $\eta_{GY}$.
\end{proof}

The following theorem is analogous to an existing strictification theorem for hypergraph categories due to Fong and Spivak~\cite[Theorem~1.1]{FS}, whereas our proof has the additional feature that it uses Mac Lane's strictification theorem as a mere black box. It has inspired Fong and Spivak to prove a more general strictification theorem for symmetric monoidal categories with extra structure~\cite[Proposition~3.24 and Corollary~4.9]{FS_new}.

Let us say that a Markov category is \emph{strict} if its underlying monoidal category is strict.

\begin{thm}
	\label{cacom_strict}
	Every Markov category is comonoid equivalent to a strict one.
\end{thm}

\begin{proof}
	Consider the subcategory of deterministic morphisms $\C_\det \subseteq \C$, which is cartesian monoidal, and choose a strict symmetric monoidal category $\C'_\det$ with a symmetric monoidal equivalence $F_\det : \C'_\det \stackrel{\cong}{\to} \C_\det$. Now let $\C'$ be the monoidal category having the same objects as $\C'_\det$, but hom-sets $\C'(X,Y) := \C(F_\det X,F_\det Y)$. If we consider $\C'_\det$ to be a subcategory of $\C'$ via $F_\det$, and write $F$ for the obvious functor $\C' \to \C$, then we have a diagram of categories
	\[
		\begin{tikzcd}[row sep=0.5cm]
			\C'_\det \ar[draw=none]{d}[sloped,auto=false]{\subseteq} \ar["F_\det"]{r} & \C_\det \ar[draw=none]{d}[sloped,auto=false]{\subseteq} \\
			\C' \ar["F"]{r} & \C
		\end{tikzcd}
	\]
	which commutes on the nose. Since $\C'_\det$ is cartesian monoidal, $\C'$ inherits canonical comonoids satisfying the compatibility with the monoidal structure of~\cref{copydiag}, and the monoidal unit $I \in \C'$ is terminal by construction. Hence $\C'$ is also a Markov category.
	
	Since $F_\det$ is symmetric monoidal and $F$ merely extends it to the non-deterministic morphisms, also $F$ comes equipped with a symmetric monoidal structure, and it is a symmetric monoidal equivalence by virtue of being an equivalence and symmetric monoidal. The essential surjectivity criterion of \Cref{equiv_crit} now holds because $F_\det$ is an equivalence.
\end{proof}

This result guarantees us that applying the graphical language for symmetric monoidal categories~\cite{JS}, as we do throughout the paper, indeed gives results that apply to all Markov categories (as the graphical language depicts only the strict case).

\begin{rem}[{cf.~\cite[Proposition~4.3]{HP}}]
	\label{eckmannhilton1}
	If $\C$ is a Markov category and $X \in \C$, then a version of the Eckmann--Hilton argument shows that $X$ carries a unique deterministic comonoid structure. For if we have any other deterministic comultiplication on $X$, drawn as a white dot, then we get
	\[
		\input{eckmann_hilton.tikz}
	\]
	where the second equation is by the determinism assumption as well as~\cref{delcopyAB}, and we have used the assumption that $I$ is terminal, which means that the two counits coincide.	
	
	This property is not particularly surprising, since the subcategory $\C_\det \subseteq \C$ is cartesian monoidal (\Cref{rem_det_cartesian}). So we mention this mainly as preparation for the stronger observation made in the upcoming~\Cref{eh_positive}.
\end{rem}

\begin{rem}
	\label{freyd}
	The previous arguments indicate that it may be useful to think of the subcategory $\C_\det \subseteq \C$ as the primary object which specifies a Markov category, together with some way of equipping this category with additional morphisms. This leads us to suspect that there could be a simple alternative definition of Markov categories as categories with finite products, modelling $\C_\det$, equipped with a suitable kind of \emph{arrow}~\cite{JHH}, modelling $\C$. Similarly and essentially equivalently~\cite[Theorem~6.1]{JHH}, the inclusion functor $\C_\det \to \C$ can also be seen as a \emph{commutative Freyd category}~\cite{LPT}\footnote{This has been pointed out to us by Sam Staton.}. In fact, it may be possible to generalize many of our methods and results to all semicartesian commutative Freyd categories. In our language, this amounts to starting with a pair of categories $(\C_\det,\C)$ where $\C_\det$ is cartesian monoidal, $\C$ is symmetric semicartesian monoidal, and we are given a strong monoidal functor $\C_\det \to \C$, which in particular equips $\C$ with affine comonoids, where the comultiplications are given as the $\C$-images of the diagonals $X \to X \otimes X$ defined in $\C_\det$. In this setting, deterministic morphisms and stochastic morphisms live in two separate categories, a setup proposed earlier by Golubtsov~\cite[Sections~2 and~3]{golubtsov2}\footnote{See also Golubtsov's~\cite[Definition~7]{golubtsov_fuzzy} for where the same definition has already been put forward earlier in the context of a specific category.} with minor variations. For cartesian monoidal $\D$, the situation $\D \to \Kl{T}$ described in \Cref{affine_monoidal_monad} gives a paradigmatic example.

	There are advantages and disadvantages to such an approach. The main advantage is that one can work with the inclusion $\Meas \to \Stoch$, and consider the deterministic morphisms to be the morphisms of $\Meas$, thereby getting around the problem described in \Cref{stoch_det}. A seemingly greater disadvantage is that one can no longer state and prove results such as \Cref{retract_asdet}, stating that retracts are almost surely deterministic.
\end{rem}

\section{Further candidate axioms for Markov categories}
\label{furtheraxioms}

Traditionally, probability and statistics take place in some souped-up version of $\FinStoch$, namely some suitable category of measurable spaces and Markov kernels (\Cref{sec_giry,sec_radon}). Such a category has many additional properties which sometimes play an important role. In this section, we describe some of these properties, thinking of them as potential additional axioms that can make a Markov category into a category that ``looks more like'' a category of Markov kernels. However, our present investigations seem too preliminary to settle on any particular set of axioms, so that we will not make a definite choice. None of these axioms will play a central role in the rest of this paper, although we will have occasional use for them in the upcoming sections (where we make explicit mention when one of them is assumed).

\subsection*{Conditionals and conditioning}

The first special property of probability theory is the possibility to do conditioning. The standard properties of conditional distributions are captured in the following definition, which is originally due to Golubtsov~\cite[Section~7.1]{golubtsov2} and has been rediscovered more recently by Cho and Jacobs~\cite[Section~3]{cho_jacobs}, who called it \emph{admitting disintegration}.

\begin{defn}
	\label{defn_cond_dists}
	Let $\C$ be a Markov category. We say that $\C$ \emph{has conditional distributions} if for every morphism $\psi : I \to X \otimes Y$, there is $\psi_{|X} : X \to Y$ such that
	\beq
		\label{cond_dist}
		\input{cond_dist.tikz}
	\eeq
\end{defn}

In terms of \Cref{formal_notation}, this equation looks exactly like the familiar defining equation of conditional distributions in the discrete case, which is the \emph{chain rule}
\beq
	\label{formal_cond_dist}
	\psi(x,y) = \psi_{|X}(y|x) \, \psi(x).
\eeq
The difference is that \Cref{formal_notation} lets us interpret this equation not only in the discrete case, but in \emph{any} Markov category, including in categories like $\Stoch$ or $\BorelStoch$ concerned with measure-theoretic probability. Note that in this case, this equation is \emph{not} an equation about densities, and $x$ and $y$ are \emph{not} elements of the measurable spaces involved, but rather mere formal variables which encode the wiring structure of \eqref{cond_dist} in symbolic form. The correct interpretation of this equation in measure-theoretic probability is the upcoming~\eqref{prcp}.

\begin{ex}
	\label{finstoch_cond_dist}
	In $\FinStoch$, we can interpret the formal variables $x$ and $y$ as actual elements. It follows that $\FinStoch$ has conditional distributions given by
	\[
		\psi_{|X}(y|x) = \frac{\psi(x,y)}{\psi(x)},	
	\]
	in case that $x \in X$ is such that $\psi(x) > 0$, which is the usual defining equation for conditional probabilities. In case that $\psi(x) = 0$, we choose any set of values $\psi_{|X}(y|x) \in [0,1]$ satisfying the required normalization condition $\sum_y \psi_{|X}(y|x) = 1$.
	Now indeed the relevant equation~\eqref{formal_cond_dist} holds by definition, and it can be interpreted either formally in the sense of \Cref{formal_notation} as a different way of writing~\cref{cond_eq}, or now equivalently as an actual equation involving real numbers and multiplication.
\end{ex}

\begin{ex}
	\label{stoch_cond_dist}
	What are conditional distributions in $\Stoch$? For a given probability measure $\psi$ on a product $\sigma$-algebra $\Sigma_X \otimes \Sigma_Y$, a conditional distribution is a Markov kernel $\psi_{|X} : (X,\Sigma_X) \to (Y,\Sigma_Y)$ satisfying the equation
	\beq
		\label{prcp}
		\psi(S\times T) = \int_{x\in S} \psi_{|X}(T|x) \, \psi(dx)
	\eeq
	for all $S \in \Sigma_X$ and $T \in \Sigma_Y$, since this is enough to prove the equation~\eqref{cond_dist} by \Cref{stoch_equal}. Such a $\psi_{|X}$ is sometimes called a \emph{product regular conditional probability}, and is known to exist whenever $(Y,\Sigma_Y)$ is standard Borel, but not in general~\cite{faden}. Thus $\BorelStoch$ has conditional distributions, but $\Stoch$ does not.
	
	A more commonly studied concept in probability and statistics are \emph{regular conditional distributions}, defined as follows~\cite[Definition~7.28]{klenke}. Let $(W,\Sigma_W,\psi)$ be a probability space and $f : (W,\Sigma_W) \to (Z,\Sigma_Z)$ a measurable function (or ``random variable'') taking values in some other measurable space $(Z,\Sigma_Z)$. Suppose that $\mathcal{F} \subseteq \Sigma_W$ is a sub-$\sigma$-algebra. Then the identity map $i : (W,\Sigma_W) \to (W,\mathcal{F})$ is measurable, and therefore also a (deterministic) morphism in $\Stoch$. Then we can consider the joint probability measure on $(W,\mathcal{F})$ and $(Z,\Sigma_Z)$ given by
	\[
		\input{rcp.tikz}
	\]
	A \emph{regular conditional distribution} in the traditional sense is then the same as a conditional distribution in our sense, instantiated in $\Stoch$ on this particular joint distribution. Again, it is a standard fact that such regular conditional distributions exist whenever $(Z,\Sigma_Z)$ is standard Borel~\cite[Theorem~8.37]{klenke},~\cite[452N]{fremlin4}, but not in general~\cite[Exercise~33.13]{billingsley}.
	
	Conversely, a generic conditional distribution in our sense can be constructed as a regular conditional probability by putting
	\[
		W := X \times Y, \qquad Z := Y,
	\]
	with the obvious $\sigma$-algebras, as well as taking $f$ to be the projection $X \times Y \to Y$ and $\mathcal{F}$ to be given by the sub-$\sigma$-algebra $\Sigma_X \subseteq \Sigma_X \otimes \Sigma_Y$. Thus conditional distributions in our sense and regular conditional distributions in the usual sense are mutual special cases of each other.

	In conclusion, $\Stoch$ does not have conditional distributions in general, but $\BorelStoch$ does. Furthermore, the existence of conditional distributions in our sense is equivalent to the existence of regular conditional distributions in the usual sense.
\end{ex}

\begin{ex}
	The Kleisli category of the Radon monad $\Kl{R}$ (\Cref{sec_radon}) does not have conditional distributions, for the simple reason that requiring the conditional $X \to Y$ to be represented by a continuous function $X \to RY$ is too much to ask for: conditionals are rarely continuous.
\end{ex}

However, \Cref{defn_cond_dists} is not yet the definition that seems most natural to us, and we will next present a more general definition which turns out to be more useful for our purposes. For motivation, consider a morphism $\psi : I \to X \otimes Y \otimes Z$, which we may think of as a joint distribution of three random variables. Then it ought to be possible to obtain a conditional $\psi_{|X \otimes Y} : X \otimes Y \to Z$ by first conditioning on $X$ and then conditioning on $Y$. However, the second step is \emph{not} a form of conditioning in the sense of \Cref{defn_cond_dists}, since only conditioning on morphisms with domain $I$ is allowed, while we would like to condition $\psi_{|X} : X \to Y \otimes Z$ on $Y$. Using the following improved definition, we will prove this iterated conditioning property in \Cref{conditional_iterate}.

\begin{defn}
	\label{defn_has_conds}
	Let $\C$ be a Markov category. We say that $\C$ \emph{has conditionals} if for every morphism $f : A \to X \otimes Y$, there is $f_{|X} : X \otimes A \to Y$ such that
	\beq
		\label{cond_eq}
		\input{conditioning_recover.tikz}
	\eeq
\end{defn}

This property specializes to the property of \Cref{defn_cond_dists} by setting $A = I$.

In terms of \Cref{formal_notation}, equation \eqref{cond_eq} becomes a conditional version of the chain rule,
\[
	f(x,y|a) = f_{|X}(y|x,a) f(x|a),
\]
which again is familiar from discrete probability, but now becomes valid in any Markov category. If $\C$ has conditionals, then we similarly get a conditional $f_{|Y}(x|y,a)$ by suitable permutations of the wires. This results in \emph{Bayes' theorem} in the generalized form
\[
	f_{|X}(y|x,a) f(x|a) = f_{|Y}(x|y,a) f(y|a),
\]
where the special case $A = I$ has been treated by Cho and Jacobs~\cite{cho_jacobs}, and appears also in similar form and somewhat implicitly in earlier work of Golubtsov~\cite[Theorem~6]{golubtsov2}\footnote{See also~\cite[Theorem~4]{golubtsov_fuzzy} for an even earlier reference where Golubtsov's considerations already appear in the context of one particular category.}.

In typical examples, permitting $A$ to be general in the definition of conditionals amounts to a certain regularity condition on the formation of conditionals: if the joint distribution of $X$ and $Y$ depends suitably nicely on $A$ so as to form a morphism in $\C$, then there also is a conditional $f_{|X}$ which depends on $A$ suitably nicely so as to form a morphism in $\C$.

\begin{ex}
	\label{finstoch_cond}
	$\FinStoch$ has conditionals. Similar to \Cref{finstoch_cond_dist}, for a given morphism $f(x,y|a)$ we form
	\beq
		\label{fcondeq}
		f_{|X}(y|x,a) := \frac{f(x,y|a)}{f(x|a)},
	\eeq
	in case that $x \in X$ and $a\in A$ are such that $f(x|a) > 0$, which again is essentially the usual defining equation for conditional probabilities, and again arbitrary values satisfying the normalization $\sum_y f_{|X}(y|x,a) = 1$ otherwise. 
\end{ex}

\begin{ex}
	\label{borelstoch_cond}
	Continuing from \Cref{stoch_cond_dist}, it was recently proven by Bogachev and Malofeev that $\BorelStoch$ has conditionals~\cite[Theorem~3.5]{BM}\footnote{We thank Kenta Cho for pointing this out to us.}. In other words, a joint distribution which depends measurably on another parameter $a \in A$ has a conditional whose dependence on $a \in A$ is still measurable. This is obvious when the $\sigma$-algebra $\Sigma_A$ is discrete, since then this extra measurability requirement is trivial.
\end{ex}

\begin{ex}
	\label{gauss_cond}
	We now investigate whether $\Gauss$ (\Cref{gaussian}) has conditionals. A morphism $A \to X \otimes Y$ can conveniently be represented in the block form
\[
	\left( \left( \begin{matrix} M \\ N \end{matrix} \right) , \left( \begin{matrix} C_{\xi\xi} & C_{\xi\eta} \\ C_{\eta\xi} & C_{\eta\eta} \end{matrix} \right) , \left( \begin{matrix} s \\ t \end{matrix} \right) \right),
\]
meaning that we have
\[
	X = M A + \xi, \qquad Y = N A + \eta,
\]
where now $\xi$ and $\eta$ are potentially \emph{correlated} random variables, with the specified expectations and covariance matrices. Using well-known results on conditionals of multivariate Gaussians~\cite[Proposition~3.13]{eaton}, we find that the conditional distribution of $Y$ given $X$ can itself be written in the form
\[
	Y = C_{\eta\xi} C^-_{\xi\xi} X + (N - C_{\eta\xi} C^-_{\xi\xi} M) A + \tau, 
\]
where $\tau$ is normal with expectation $\E{\tau} := t - C_{\eta\xi} C^-_{\xi\xi} s$ and variance
\[
	\Var{\tau} := C_{\eta\eta} - C_{\eta\xi} C^-_{\xi\xi} C_{\xi\eta},
\]
and $C^-_{\xi\xi}$ is the Moore--Penrose pseudoinverse of $C_{\xi\xi}$. This matrix satisfies the relevant equations
\[
	C^-_{\xi\xi} C_{\xi\xi} C^-_{\xi\xi} = C^-_{\xi\xi}, \qquad C_{\eta\xi} C^-_{\xi\xi} C_{\xi\xi} = C_{\eta\xi}, \qquad C_{\xi\xi} C^-_{\xi\xi} C_{\xi\eta} = C_{\xi\eta},
\]
where the proof of the latter two employs standard properties of positive semidefinite matrices\footnote{They are simple to prove in a basis where $C_{\xi\xi}$ is diagonal: if there are zeros on the diagonal of $C_{\xi\xi}$, then positive semidefiniteness of $C$ implies that the relevant rows and columns of $C_{\eta\xi}$ and $C_{\xi\eta}$ must vanish as well and hence can be ignored, reducing the problem to the case where $C_{\xi\xi}$ is invertible.}. The positive semidefiniteness of the above expression for $\Var{\tau}$ follows from the fact that it is a matrix of the form $v^T C v$, where $C$ is the original covariance matrix in the above block form and
\[
	v = \left( \begin{matrix} C^-_{\xi\xi} C_{\xi\eta} \\ -1 \end{matrix} \right).
\]
The dependence of $Y$ on $X$ and $A$ therefore corresponds to the morphism $X \otimes A \to Y$ represented by the triple
\[
	\left( \left( \begin{matrix} C_{\eta\xi} C^-_{\xi\xi} & N - C_{\eta\xi} C^-_{\xi\xi} M\end{matrix} \right) ,\, C_{\eta\eta} - C_{\eta\xi} C^-_{\xi\xi} C_{\xi\eta},\, t - C_{\eta\xi} C^-_{\xi\xi} s \right),
\]
reproducing formulas due to Lauritzen and Jensen~\cite[Section~4.5]{LJ}, which generalize Golubtsov's earlier~\cite[Theorem~4]{golubtsov_gauss}. A straightforward but tedious computation using the Markov category structure introduced above shows that this morphism is indeed the desired conditional in the sense of \Cref{defn_has_conds}. Therefore $\Gauss$ has conditionals.
\end{ex}

We now present some open problems before developing the theory of conditionals further.

\begin{prob}
	Suppose that $\C$ is a Markov category which has conditionals, and $\D$ a small category. Then does the functor category $\Fun(\D,\C)$ from \Cref{functor_cat} have conditionals as well?
\end{prob}

This is not obvious, because if we try to form the conditional of a morphism of diagrams by doing so separately at each object of $\D$, then there is no guarantee that we again obtain a morphism of diagrams. A possible way around might be to relax the notion of morphism of diagrams in some way with respect to almost sure equality (\Cref{sec_as}).

In measure-theoretic probability, conditioning is also often formalized in terms of \emph{conditional expectations}. We suspect that conditional expectations indeed arise from \Cref{defn_has_conds} as well as follows.

\begin{conj}
	\label{cond_exp}
	There is a Markov category such that:
	\begin{itemize}
		\item Objects are measurable spaces $(X,\Sigma_X)$, or some variant thereof such as \emph{measure algebras}~\cite{fremlin3}.
		\item Morphisms $(X,\Sigma_X) \to (Y,\Sigma_Y)$ are some version of positive unital linear maps between spaces of bounded measurable functions in the other direction, $\mathcal{L}^\infty(Y,\Sigma_Y) \to \mathcal{L}^\infty(X,\Sigma_X)$, taking bounded directed suprema to bounded directed suprema.
		\item The monoidal structure is given by some variant of the usual product of measurable spaces.
		\item The comultiplications $X \to X \otimes X$ are implemented by restriction to the diagonal, $\mathcal{L}^\infty(X\times X,\Sigma_X \otimes \Sigma_X) \longrightarrow \mathcal{L}^\infty(X,\Sigma_X)$.
		\item Conditionals exist in the sense of \Cref{defn_has_conds}, and are given by conditional expectations.
	\end{itemize}
\end{conj}

One candidate for such a Markov category may be the category described in \Cref{dawid_stoch}.

Returning to the general theory, sometimes we may want to condition on more than one output. The following guarantees that this can be done either all at once or step by step:

\begin{lem}
	\label{conditional_iterate}
	Suppose that $\C$ has conditionals. Then for any morphism $f : A \to X \otimes Y \otimes Z$, we have
	\[
		\input{double_conditional.tikz}
	\]
	in the sense that every double conditional $(f_{|X})_{|Y}$ as on the right is also a conditional $f_{|X\otimes Y}$ as on the left (with first two inputs swapped).
\end{lem}

Note that this result can also be applied iteratively in order to compute a conditional $f_{|X_1\otimes \ldots \otimes X_n}$ of a morphism $f : A \to X_1 \otimes \ldots \otimes X_n \otimes Y$. So far we have not found any formulation and proof of a converse, meaning a way to say that every conditional $f_{|X \otimes Y}$ is also a double conditional $(f_{|X})_{|Y}$ (with first two inputs swapped).

\begin{proof}
	We start with the chosen $(f_{|X})_{|Y}$ and compute	
	\[
		\input{double_conditional_proof1.tikz}
	\]
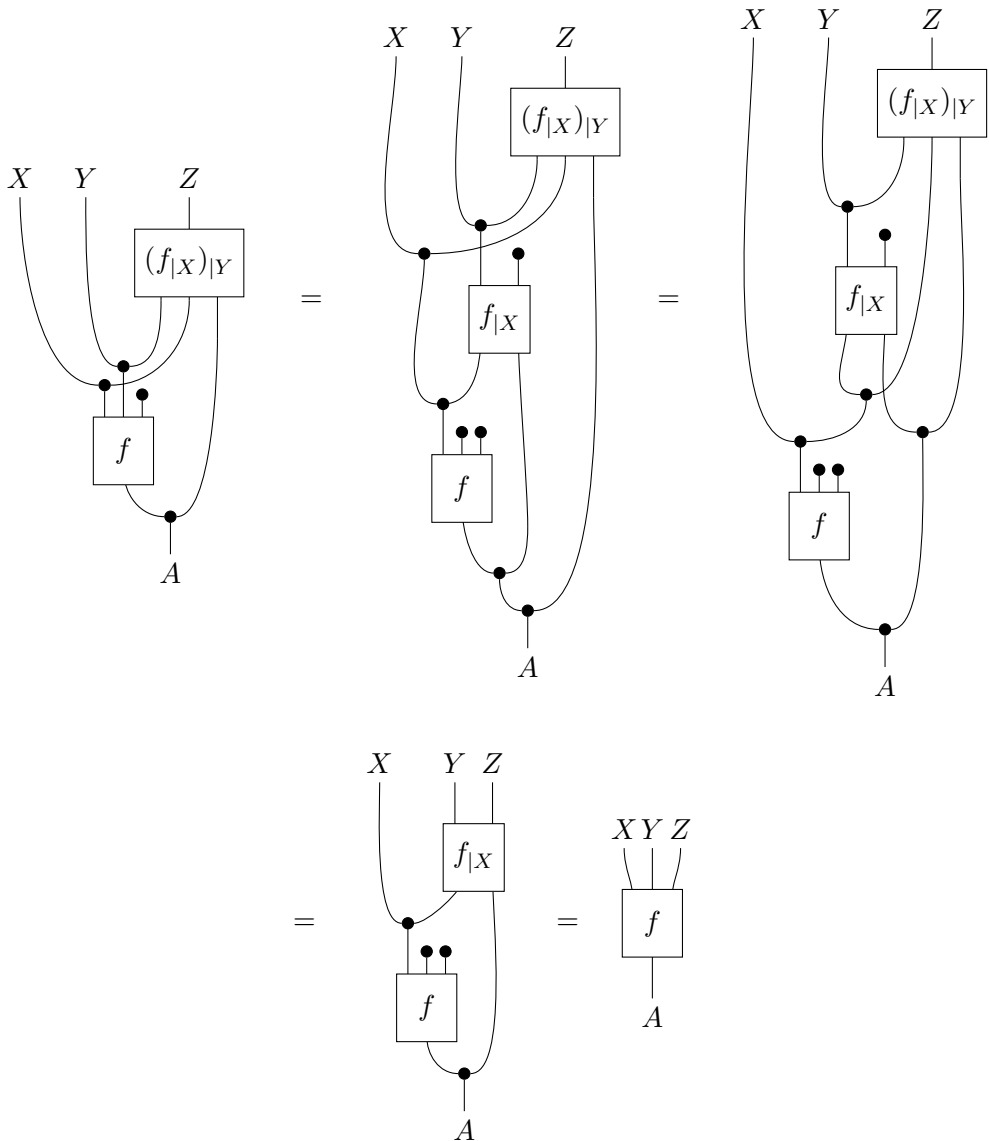
	\[
		\input{double_conditional_proof2.tikz}
	\]
	which shows indeed that the given $(f_{|X})_{|Y)}$ is also a conditional $f_{|X \otimes Y}$, modulo swapping the first two inputs.
\end{proof}

Conditioning also enjoys a certain compositionality property:

\begin{lem}
	\label{cond_postcompose}
	If $f : A \to X \otimes W$ and $g : W \otimes B \to Y$ are arbitrary morphisms and $f_{|X} : X \otimes A \to W$ is a corresponding conditional, then the composite
	\[
		\input{general_compose.tikz}
	\]
	has a conditional given by
	\[
		\input{conditional_compose.tikz}
	\]
\end{lem}

\begin{proof}
	Straightforward string diagram computation.
\end{proof}

\begin{rem}
	\label{cond_marginal}
	As a special case, we obtain that conditioning commutes with marginalization as follows. Using \Cref{formal_notation}, any $f : A \to X \otimes Y \otimes Z$ is such that the marginal $f_{|X}(y|xa)$ of the conditional $f_{|X}(yz|xa)$ is also a conditional $f(y|xa)_{|X}$ for the marginal $f(xy|a)$. This follows from \Cref{cond_postcompose} upon taking $B := I$ and $g : Y \otimes Z \to Y$ to be the projection.
\end{rem}

\begin{rem}
	Besides the previous two lemmas, we have not been able to identify any further potentially useful categorical properties of conditioning. Our original goal in this work had been to axiomatize conditioning as an operation on a category satisfying certain naturality conditions. We have abandoned this due to the lack of further candidate axioms for conditioning which would hold e.g.~in $\FinStoch$, and the well-known problem that conditionals are generally non-unique in cases where zero probabilities occur (\Cref{cond_notunique}), leading us to suspect that a reasonably coherent choice of conditionals is not possible. We have been all the more surprised to realize that conditioning is also not a very relevant concept for most of the developments of this paper.
\end{rem}

Within our framework, it is impossible to guarantee the uniqueness of conditionals except in uninteresting degenerate cases, because of the following observation.

\begin{prop}
	\label{cond_notunique}
	Suppose that $\C$ has conditionals. Then conditionals in $\C$ are unique if and only if any two parallel morphisms in $\C$ are equal.
\end{prop}

\begin{proof}
	The ``if'' direction is trivial, so that we focus on the ``only if'' part. Assuming the uniqueness of conditionals, we first prove that for every $Y \in \C$, we have an equality of morphisms $Y \otimes Y \to Y$ given by
	\[
		\input{xdegenerate.tikz}
	\]
	Indeed, a straightforward computation using the comonoid axioms \cref{comonoid_ass,comonoid_other} shows that both sides are conditionals for the comultiplication
	\[
		\input{comultiplication.tikz}
	\]
	Now for any two parallel morphisms $f, g : X \to Y$, we get
	\[
		\input{degenerate_comp.tikz}
	\]
	as was to be shown.
\end{proof}

\begin{rem}
The equality of parallel morphisms, $f = g$ for $f,g : X \to Y$, means that $\C$ is equivalent to a preordered set: the category structure is uniquely determined by writing $X \leq Y$ if and only if there is a (necessarily unique) morphism $X \to Y$. Then the comonoid structures of \Cref{markov_cat} show that the tensor product of objects must be a meet operation in the sense of lattice theory, making $\C$ into a meet-semilatice with neutral element $I \in \C$. Conversely, upon considering every meet-semilattice as a monoidal category, it also is a Markov category which has unique conditionals.
\end{rem}

The existence of conditionals implies a stronger disintegration property, which we will relate to the standard notion of disintegration in measure-theoretic probability in \Cref{stoch_disint}.

\begin{prop}[Disintegration]
	\label{disint}
	Suppose that $\C$ has conditionals. Then for every $p : A \to X$ and $f : X \to Y$, there is $s : A \otimes Y \to X$ such that
	\beq
		\label{disint_eq1}
		\input{disintegration.tikz}
	\eeq
\end{prop}

For $A = I$ and deterministic $f$, this equation implies the almost sure equality $fs =_{fp\as} \id_Y$ (\Cref{sec_as}) by composing with $f$ on the right output.

\begin{proof}
	Define $s$ to be a conditional of the morphism
	\[
		\input{disintegration_proof1.tikz}
	\]
	on $Y$. This means that $s$ satisfies the equation
	\[
		\input{disintegration_proof2.tikz}
	\]
	which immediately implies the claim.
\end{proof}

\begin{ex}
	\label{stoch_disint}
	In $\Stoch$, setting $A = I$ in \Cref{disint} and taking $f$ to be given by a measure-preserving function of measurable spaces $f : X \to Y$ specializes the statement to the usual notion of disintegration of a probability measure $p$ on $X$ with respect to a measure-preserving function $f : X \to Y$.\footnote{For example, we may consider a sub-$\sigma$-algebra $\Sigma'_X$ and consider the measurable map $f := \id_X : (X,\Sigma_X) \to (X,\Sigma'_X)$, thereby recovering the notion of disintegration with respect to a sub-$\sigma$-algebra.} Indeed under these assumptions,~\cref{disint_eq1} states that 
	\[
		p(f^{-1}(T) \cap S) \stackrel{!}{=} \int_{y\in T} s(S|y) \, (fp)(dy) = \int_{y\in T} s(S|y) \, p(f^{-1}(dy))
	\]
	for all $S \in \Sigma_X$ and $T \in \Sigma_Y$. This is exactly the usual notion of \emph{regular conditional probability}. It is well-known that regular conditional probabilities do not exist in general~\cite[p.~624]{doob}. We can therefore conclude, by the contrapositive of \Cref{disint}, that $\Stoch$ does not have conditionals, as we had already noted in \Cref{stoch_cond_dist}.
\end{ex}

\subsection*{Randomness pushback} We move on to our second candidate axiom. It is inspired by the treatment of random mappings in~\cite[Example~2.6(3)]{jhht}, and a similar definition has been given as~\cite[Definition~8]{CH}. For $\FinStoch$, it is well-known that randomness can be ``pushed back'', in the sense that every Markov kernel can be replaced by a deterministic map with one extra input into which one feeds a random value with a fixed distribution.

\begin{defn}[Randomness pushback]
	\label{random_pushback}
	We say that $\C$ has \emph{randomness pushback} if for every morphism $f : X \to Y$ in $\C$, there are $A \in \C$, deterministic $g : A \otimes X \to Y$ and $\psi : I \to A$ such that
	\beq
		\label{rp}
		\input{randomness_pushback.tikz}
	\eeq
\end{defn}

While this is an interesting property, we will have no use for it in this paper. In certain other contexts, it may be useful to apply it repeatedly. This happens for example in the theory of Bayesian networks, where one can use it to prove the equivalence between Bayesian networks with stochastic dependencies, and networks with purely functional (i.e.~deterministic) dependencies plus extra ``exogenous'' root variables which supply the randomness~\cite{DS_others}. Our new formal perspective should have the potential to shed new light on technical problems with measure-theoretic aspects in these contexts, such as~\cite[Conjecture~4.5]{fritz}.

As the following examples show more concretely, for $Y = X$ randomness pushback essentially specializes to the problem of representing a stochastic dynamical system in terms of a random choice of deterministic dynamics~\cite[Theorem~1.1]{kifer}.

\begin{ex}
	\label{finstoch_rp}
	$\FinStoch$ has randomness pushback. For given $f$, take $A := Y^X$ to be the finite set of functions from $X$ to $Y$, and define $\psi : I \to A$ by declaring the probability weight of any function $\mathfrak{y} \in Y^X$ to be given by
	\[
		\psi(\mathfrak{y}) := \prod_{x\in X} f(\mathfrak{y}(x)|x),
	\]
	so that all the individual values $\mathfrak{y}(x)$ are stochastically independent and have the same distribution as $f(y|x)$.\footnote{Any other joint distribution with the same marginals on the $Y$-factors of $Y^X$ will work just as well.} Using the deterministic function $g : Y^X \times X \to Y$ given by evaluation,
	\[
		g(\mathfrak{y}, x) := \mathfrak{y}(x),
	\]
	then clearly results in~\cref{rp}.
\end{ex}

We do not know whether $\Stoch$ has randomness pushback in general. One obstacle in generalizing the above argument from $\FinStoch$ to $\Stoch$ is the fact that $\Meas$ is not cartesian closed~\cite{aumann}.

\begin{ex}
	\label{gauss_rp}
	We now show that $\Gauss$ also has randomness pushback. Any morphism $(M,C,s) : n \to m$, representing $Y = MX + \xi$, can be obtained by plugging in the morphism $\psi : 0 \to m$ with parameters $(0,C,s)$ into the addition map $+ : m + m \to m$, which is the deterministic morphism represented by $\left( \left( \begin{matrix} 1 & 1 \end{matrix} \right), 0, 0 \right)$.
\end{ex}

\subsection*{Positivity}

Our third candidate axiom is a bit more mysterious. We will show that it holds in our running examples, and then make crucial use of it to prove our version of the Fisher--Neyman factorization theorem (\Cref{thm_FN}).

\begin{defn}
	\label{positive_defn}
	We say that $\C$ is \emph{positive} if whenever $f : X \to Y$ and $g : Y \to Z$ are such that $gf$ is deterministic, then also
	\beq
		\label{positive}
		\input{positive.tikz}
	\eeq
\end{defn}

In terms of the upcoming \Cref{ciprocdef}, this equation can be understood as saying that the morphism shown on the left-hand side displays the conditional independence $\condindproc{Y}{Z}{X}$. Intuitively, this equation means the following: if a deterministic process (first output) has a random intermediate result (second output), then the same distribution over both results can also be obtained by computing the intermediate result independently of the deterministic process.

\begin{rem}
	\label{id_positive}
	A relevant special case which comes up frequently is with $gf = \id$, in which case we get
	\[
		\input{id_positive.tikz}
	\]
\end{rem}

The following result implies that many of our example Markov categories are indeed positive, in particular $\FinStoch$ and $\Gauss$.

\begin{lem}
	If $\C$ has conditionals, then $\C$ is positive.
	\label{cond_positive}
\end{lem}

\begin{proof}
	By \Cref{disint}, we have $s$ such that
	\beq
		\label{positive_cond_eq}
		\input{positive_cond.tikz}
	\eeq
	The determinism assumption evaluates this further to
	\[
		\input{positive_cond2.tikz}
	\]
	where the second step is another use of~\eqref{positive_cond_eq}. This implies the claim.
\end{proof}

The term ``positive'' in \Cref{positive_defn} is motivated by the crucial role of the nonnegativity of probabilities in the following arguments.

\begin{ex}
	\label{stoch_positive}
	$\Stoch$ is positive, as we now show. By the characterization of deterministic morphisms in \Cref{stoch_det}, the assumption is that
	\beq
		\label{stoch_positive_ass}
		\int_{y\in Y} g(S|y) \, f(dy|x)  \: \in \{0,1\}
	\eeq
	for all $S \in \Sigma_X$ and $x \in X$. By \Cref{stoch_equal}, what we want to prove is the equation
	\beq
		\label{stoch_positive_claim}
		\int_{y \in T} g(S|y) \, f(dy|x) = f(T|y) \, \int_{y \in Y} g(S|y) \, f(dy|x)
	\eeq
	for all $S \in \Sigma_Z$ and $T \in \Sigma_Y$ and $x \in X$. 
	
	It is easy to see that if the claimed equation holds for $S$, then it also holds for the complementary event $Z \setminus S$. Thus it is enough to consider the case where \eqref{stoch_positive_ass} is zero. Since an integral of a nonnegative function vanishes if and only if the integrand vanishes almost surely, we can conclude that the integrand $y \mapsto g(S|y)$ vanishes $f(-|x)$-almost surely on all of $Y$. Therefore both sides of the claimed equation~\eqref{stoch_positive_claim} vanish as well.
\end{ex}

\begin{rem}
	It is easy to see that if $\C \to \C'$ is a Markov embedding and $\C'$ is positive, then so is $\C$. For example since $\Stoch$ is positive, the positivity of $\Gauss$ follows alternatively from the Markov embedding $\Gauss \to \Stoch$ of \Cref{gauss_stoch}.
\end{rem}

We now present a negative example.

\begin{ex}
	\label{finstoch_pm}
	Let $\FinStoch_\pm$ be defined just as $\FinStoch$, but dropping the nonnegativity of the matrix entries, so that a morphism $f : X \to Y$ is a matrix $(f(y|x))_{x\in X, y \in Y}$ of arbitrary real numbers with $\sum_y f(y|x) = 1$ for all $x \in X$. In other words, negative probabilities are allowed. $\FinStoch_\pm$ is easily seen to be a Markov category in the same way as $\FinStoch$ is. We can also obtain $\FinStoch_\pm$ as a Markov category from the hypergraph category $\mathsf{Mat}(\R)$ via~\Cref{sec_hypergraph}.
	
	$\FinStoch_\pm$ is not positive: a straightforward but slightly tedious computation shows that
	\[
		f = \left( \begin{matrix} 2 & 2 \\ -1 & 0 \\ 0 & -1 \end{matrix} \right), \qquad g = \left( \begin{matrix} \frac{1}{2} & 0 & 1 \\[2pt] \frac{1}{2} & 1 & 0 \end{matrix} \right),
	\]
	satisfy $gf = \id$, but the equation of \Cref{id_positive} does not hold.
	
	In particular, the contrapositive of \Cref{cond_positive} implies that $\FinStoch_\pm$ does not have conditionals. This can also be seen directly by noting that the desired equation for conditional distributions, $\psi(x,y) = \psi_{|X}(y|x) \, \psi(x)$, has vanishing right-hand side whenever $\sum_y \psi(x,y) = 0$, although the left-hand side does not need to vanish in this case.
\end{ex}

Let us return to the general theory.

\begin{rem}
	If $\C$ is positive, then it is easy to see that every isomorphism is deterministic: just take $g$ to be the inverse of the isomorphism $f$, and compose the equation of \Cref{id_positive} with $f \otimes \id$. This is a property which fails in general Markov categories (\Cref{comons_gen}).
\end{rem}

\begin{rem}
	\label{eh_positive}
	If $\C$ is positive, then we also get a stronger version of \Cref{eckmannhilton1}: \emph{every} comonoid structure on every object $X \in \C$ is equal to the distinguished one. Namely, we have $gf = \id$ if we take $f$ to be the other comultiplication and $g$ to be the marginalization over the second tensor factor,
	\[
		\input{counitality2.tikz}
	\]
	Now applying the positivity condition of \Cref{id_positive} results in the equation
	\[
		\input{eckmann_hilton2.tikz}
	\]
	Marginalizing over the third output now shows that the two comultiplications must be equal.

	While we have not worked out the details, we therefore hope that it may be possible to define positive Markov categories equivalently as semicartesian symmetric monoidal categories together with a symmetric monoidal equivalence to their category of internal comonoids. This comes close to describing them entirely in 2-monadic terms.

	These considerations make it seem plausible to us that some version of the positivity condition could also be of interest in the study of hypergraph categories~\cite{kissinger,fong,FS}.
\end{rem}

Using positivity, we can strengthen \Cref{deterministic_composition}\ref{2of3}.

\begin{lem}
	Suppose that $\C$ is positive. Let $f : X \to Y$ and $g : Y \to Z$ in $\C$. If $gf$ is deterministic and $f$ is an epimorphism, then $g$ is deterministic as well.
\end{lem}

\begin{proof}
	This is by the same computation as in the proof of \Cref{deterministic_composition}, now using the positivity axiom for the second step.
\end{proof}

We finally consider another property which will be of use in some of our later proofs. As we show in \Cref{eq_strengthen}, this axiom is closely related to the notion of \emph{equality strengthening} introduced by Cho and Jacobs~\cite[Definition~5.7]{cho_jacobs} (\Cref{eq_strengthen}).

\begin{defn}
	\label{causal_defn}
	We say that $\C$ is \emph{causal} if the following condition holds: whenever
	\[
		\input{causal1.tikz}
	\]
	holds for any morphisms as indicated, then also
	\beq
		\label{causal_conc}
		\input{causal2.tikz}
	\eeq
\end{defn}

The name of this condition indicates that we think of it as a causality property: with time going up, we assume that the choice of using $h_1$ or $h_2$ at a late time has no influence on how its outcome is correlated with what happened at an earlier time; and then if causality holds, we can conclude that there is no correlation with what happened at an even earlier time either.

\begin{rem}
	One may wonder how our axiom relates to the naturality condition on the counit~\eqref{counit_nat}, which is sometimes also called \emph{causality}~\cite{CL,CDP}. So which one of the two actually expresses causality, and correspondingly deserves to carry the name \emph{causality axiom}? Our perspective is that they \emph{both} capture disjoint aspects of the causality of information flow. It is an open question, suggested to us by Thomas C.~Fraser, whether there is an interesting causality axiom or axiom schema which comprises both as special cases, and which could be argued to capture all aspects of causality. 

	It may also be the case that our \Cref{positive_defn} and \Cref{causal_defn} are instances of a more comprehensive axiom comprising both as special cases. We have not yet found any such condition either.
\end{rem}

The implication postulated by the causality axiom holds automatically in many cases: for example if $f$ is deterministic; if $g$ is deterministic; if $f$ is an epimorphism; if $gf = g$; or if $\C$ has conditionals (\Cref{cond_causal}). It is correspondingly nontrivial to find examples where our causality condition fails. Here is one.

\begin{ex}
	\label{noncausal}
	Applying \Cref{comons} with $\D := \Set^\op$ shows that the opposite of the category of commutative monoids with unit-preserving functions is a Markov category. Using additive notation for the monoid operations, \Cref{causal_defn} now states the following: whenever $f$, $g$ and $h_1$, $h_2$ are unit-preserving maps between commutative monoids, and if the equation
	\beq
		\label{pfass}
		f(g(h_1(x) + y)) = f(g(h_2(x) + y))
	\eeq
	holds for all $x$ and $y$, then also the stronger equation
	\beq
		\label{pfconc}
		f(g(h_1(x) + y) + z) = f(g(h_2(x) + y) + z)
	\eeq
	holds for all $x$, $y$ and $z$. To see that this is not true in general, let us specialize to the monoid $(\N,+)$ and all maps of type $(\N,+) \to (\N,+)$, and consider more concretely the maps given by
	\[
		f(x) := \max(x-1,0), \qquad g(x) := \min(x,1), \qquad h_1(x) := x, \qquad h_2(x) := 0,
	\]
	which clearly all preserve the unit $0$. Then $fg = 0$, so that the assumption~\cref{pfass} holds trivially. However, the conclusion~\cref{pfconc} fails with $x = z = 1$ and $y = 0$. We therefore have found a Markov category which is not causal.
\end{ex}

Many uses of the causality condition use it only in a weaker form, where the middle output in~\cref{causal_conc} is marginalized over; see e.g.~the proof of \Cref{probstoch}.

We now show that causality is intimately related to the existence of conditionals.

\begin{prop}
	\label{cond_causal}
	If $\C$ has conditionals, then $\C$ is causal.
\end{prop}

\begin{proof}
	In terms of $s$ as in the proof of \Cref{cond_positive},
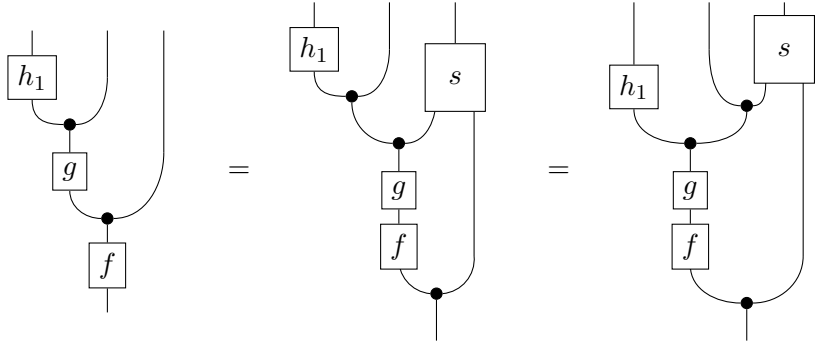
	\[
		\input{causal_proof.tikz}
	\]
	Since $h_2$ can be replaced by $h_2$ on the right-hand side by assumption, the same holds true on the left-hand side.
\end{proof}

It follows in particular that $\FinStoch$ and $\Gauss$ are causal. The next example shows in particular that causality does not imply the existence of conditionals, and is therefore strictly weaker.

\begin{ex}
	\label{stoch_causal}
	$\Stoch$ is causal despite not having conditionals (\Cref{stoch_disint}). Indeed, with $f : A \to X$ and $g : X \to Y$ and $h_1, h_2 : Y \to Z$, the assumption is that
	\beq
		\label{intass}
		\int_{x \in X} \int_{y \in T} h_1(U|y) \, g(dy|x) \, f(dx|a) = \int_{x \in X} \int_{y \in T} h_2(U|y) \, g(dy|x) \, f(dx|a)
	\eeq
	for all $T \in \Sigma_Y$ and $U \in \Sigma_Z$ and $a \in A$. What we need to prove is that this equation implies a generalization where the domain of integration of $x$ is a general $S \in \Sigma_X$,
	\beq
		\label{intconc}
		\int_{x \in S} \int_{y \in T} h_1(U|y) \, g(dy|x) \, f(dx|a) = \int_{x \in S} \int_{y \in T} h_2(U|y) \, g(dy|x) \, f(dx|a).
	\eeq
	Writing the assumption~\cref{intass} in the form
	\[
		\int_{y \in T} h_1(U|y) \int_{x \in X} g(dy|x) \, f(dx|a) = \int_{y \in T} h_2(U|y) \int_{x \in X} g(dy|x) \, f(dx|a)
	\]
	results in two functions on $Y$ whose integrals over every $T \in \Sigma_Y$ are equal when integrated against the measure $\int_{x \in X} g(dy|x) \, f(dx|a)$. Hence the functions $y \mapsto h_1(U|y)$ and $y \mapsto h_2(U|y)$ are equal almost surely with respect to this measure. Therefore they are also equal almost everywhere with respect to any submeasure $\int_{x \in S} g(dy|x) \, f(dx|a)$ for arbitrary $S \in \Sigma_X$, which is~\cref{intconc}. Hence $\Stoch$ is causal.
\end{ex}

\begin{rem}
	\label{eq_strengthen}
	As pointed out to us by Kenta Cho, our causality condition is closely related to the notion of \emph{equality strengthening} considered by Cho and Jacobs~\cite[Definition~5.7]{cho_jacobs}. Indeed, equality strengthening means that every equation of the form
	\[
		\input{eq_strengthen1.tikz}
	\]
	implies the stronger equation
	\[
		\input{eq_strengthen2.tikz}
	\]
	Now if $\C$ is causal in the sense of \Cref{causal_defn}, then it also admits equality strengthening, as one can see readily by taking $f$ to be $\psi$ and $g$ to be the marginalization map which discards the right output. Furthermore, if one generalizes equality strengthening by replacing the $\psi$ above by an arbitrary morphism with nontrivial input, then it is straightforward to prove that causality and generalized equality strengthening are equivalent.

	There is also an existing proof showing that the existence of conditionals implies equality strengthening in the original sense~\cite[Proposition~5.8]{cho_jacobs}, which now also follows from \Cref{cond_causal}. Similarly, the proof that $\Stoch$ has equality strengthening given by Cho and Jacobs~\cite[Proposition~5.9]{cho_jacobs} now follows from \Cref{stoch_causal}. On the other hand, Kenta Cho has informed us that equality strengthening is enough for the proof of our \Cref{probstoch}\ref{probstoch_welldef}.
\end{rem}

\section{Conditional independence and the semigraphoid properties}
\label{cind}

Conditional independence is a central concept in probability, for example due to its importance for statistics~\cite{dawid} and its basic role in the definition of graphical model~\cite{lauritzen_book}. Generalizing conditional independence to arbitrary Markov categories is therefore a natural part of developing synthetic probability theory. Among other things, a formal treatment of conditional independence at our level of generality enables the interpretation of graphical models in \emph{any} Markov category (although we will not elaborate on this further).

Traditionally, conditional independence is closely tied to conditioning. For example in $\FinStoch$, a distribution $\psi : I \to X \otimes W \otimes Y$ is considered to display the conditional independence of $X$ and $Y$ given $W$ if the conditional $\psi_{|W}(xy|w)$ satisfies the factorization
\[
	\psi_{|W}(xy|w) = \psi_{|W}(x|w) \, \psi_{|W}(y|w),
\]
or a number of equivalent reformulations of this equation, such as $\psi(xyw) \psi(w) = \psi(xw) \psi(yw)$. In the context of Markov categories with conditional distributions, conditional independence has been defined by Cho and Jacobs~\cite{cho_jacobs}, who proved that some of the usual equivalences hold at this level of generality, where they can be derived in terms of string diagrams. In~\cite[Proposition~6.10]{cho_jacobs}, Cho and Jacobs also showed that their generalized notion of conditional independence satisfies the well-known semigraphoid properties~\cite{PP},~\cite[Section~2.2.2]{studeny_book}. In a somewhat different technical setup, similar definitions had also been made earlier by Coecke and Spekkens~\cite[Section~5.1]{CS}, who also proved the equivalence of various definitions~\cite[Proposition~5.3]{CS} under strong assumptions on the existence of their version of conditionals.

Here, we will show that it is still meaningful to introduce and work with conditional independence in the absence of conditionals. However, now the single notion of conditional independence splits into several distinct ones. For example, one may want to say that for a morphism $X \to Y \otimes Z$, the output $Y$ is independent of the input $X$ given the output $Z$ (as in the definition of sufficient statistic, \Cref{suff_defn}). Here, the standard notion of conditional independence does not apply directly, since $X$ does not come with any distribution; it therefore seems pertinent to distinguish this notion of conditional independence from the corresponding one for distributions $I \to X \otimes Y \otimes Z$. We believe that such distinctions may be able to shed further light on Dawid's analysis of the role of conditional independence in statistics~\cite{dawid}, where distinct notions of conditional independence seem to be coming up.

In the following, we introduce the relevant definitions of conditional independence and study some of their properties, proving versions of the semigraphoid properties as well as certain compositionality properties, including conditional products. Our various definitions of conditional independence can be unified by formulating them in terms of disconnecting the relevant string diagrams (\Cref{disconnect}). An interesting application, which we mostly leave to future work, should be to instantiate these definitions in $\Stoch$, the category of measurable spaces and Markov kernels of \Cref{sec_giry}, and to compare it in more detail with existing notions of conditional independence in measure-theoretic probability (\Cref{stoch_cistate}). One difficulty here is that $\Stoch$ does not have conditionals (\Cref{stoch_cond_dist}). Instantiating our definitions and results in this case should give a theory of conditional independence for Markov kernels between products of arbitrary measurable spaces, which may be new (since conditional independence is usually primarily considered for joint distributions of three random variables, or equivalently for a triple of $\sigma$-algebras on a probability space).

Concretely, we start by introducing two simple notions of conditional independence, one denoted $\condindstate{X}{Y}{W}$ for ``distributions'' $I \to X\otimes W \otimes Y$, and one denoted $\condindproc{X}{Y}{A}$ for ``processes'' $A \to X \otimes Y$, and study some of their properties. Subsequently, we will show how to combine both notions into one more general notion of conditional independence for morphisms $A \to X \otimes W \otimes Y$.

\begin{defn}
	\label{cistatedef}
	Let $\C$ be a Markov category. A morphism $\psi : I \to X \otimes W \otimes Y$ in $\C$ \emph{displays the conditional independence} $\condindstate{X}{Y}{W}$ if there are $f : W \to X$, $g : W \to Y$ and $\phi : I \to W$ such that
	\beq
		\label{cistateeq}
		\input{cond_ind_state.tikz}
	\eeq
\end{defn}

\begin{rem}
	\label{cistate_cond}
	In terms of \Cref{formal_notation}, marginalizing over the first and third output shows that $\phi$ must be equal to the marginal $\psi(w)$. Marginalizing over the third output only then shows that $f$ must be a conditional $\psi(xw)_{|W}$ in the sense of \Cref{defn_has_conds}. Similarly, $g$ must be a conditional $\psi(yw)_{|W}$.
\end{rem}

Our definition essentially coincides with the characterization given by Cho and Jacobs in~\cite[Proposition~6.9(2)]{cho_jacobs}, who have shown that it is equivalent to the factorization $\psi_{|W}(xy|w) = \psi_{|W}(x|w) \psi_{|W}(y|w)$, in terms of \Cref{formal_notation}, if conditionals exist. In the case of $\FinStoch$, this implies in particular that \Cref{cistatedef} is equivalent to the usual definition of conditional independence for probability distributions on products of finite sets.

Note that our definition applies regardless of whether $\C$ has conditionals or not. And of course, it also applies when $X$, $W$ or $Y$ are themselves tensor products of other objects. The latter recovers the fact that conditional independence as usually conceived applies to \emph{sets} of random variables~\cite{studeny_book}.

\begin{rem}
	\label{independence}
	As a simple special case going back to Manes~\cite[Section~5]{manes} and Golubtsov~\cite[Section~7.1]{golubtsov2}, for $W = I$ we get the obvious notion of plain \emph{independence}: a morphism $\psi : I \to X \otimes Y$ displays the independence $X \perp Y$ if and only if
	\[
		\input{independence.tikz}	
	\]
	which, in terms of \Cref{formal_notation}, becomes exactly the usual factorization condition $\psi(x,y) = \psi(x) \psi(y)$.
\end{rem}

\begin{rem}
	\label{stoch_cistate}
	It is an interesting problem to work out what \Cref{cistatedef} instantiates to in $\Stoch$. The following is based on discussion with Alex Simpson.

	In $\Stoch$, our definition directly coincides with the one given by Dawid and Studen\'y via~\cite[2.2.2 \& Eq.~(3)]{DS}; while their definition uses regular conditional probabilities in place of $f$ and $g$, our $f$ and $g$ must also automatically be conditionals as noted in \Cref{cistate_cond}.

	A more commonly used definition of conditional independence in measure-theoretic probability is arguably the definition of conditional independence for a triple of $\sigma$-algebras on a probability space in terms of conditional expectation~\cite[Section~7.3]{CT}. Applying this definition in the case of a measure $\psi$ on a product $\sigma$-algebra $\Sigma_X \otimes \Sigma_W \otimes \Sigma_Z$, using the three constituting factors, also gives a notion of conditional independence for $\psi : I \to X \otimes W \otimes Y$ in $\Stoch$, in the form
	\beq
		\label{ceci}
		\mathbb{E}[1_{S\cap T}|\Sigma_W] = \mathbb{E}[1_S|\Sigma_W]\, \mathbb{E}[1_T|\Sigma_W]
	\eeq
	for all $S \in \Sigma_X$ and $T \in \Sigma_Y$. (The general definition involving a triple of $\sigma$-algebras is not really more general, since one can always work with the third power of the original space, equipped with the product of the three original $\sigma$-algebras, and equipping this product measurable space with the pushforward probability measure along the diagonal inclusion.) While we currently do not see how to arrive at the more commonly used definition~\cref{ceci} within our framework, we suspect that a Markov category as in \Cref{cond_exp} would do exactly that.
\end{rem}

Returning to the general theory, we can now derive versions of the four semigraphoid properties, which we state in terms of \Cref{formal_notation}.

\begin{lem}
	\label{cistateprops}
	The following hold:
	\begin{enumerate}
		\item Symmetry: if $\psi(x,w,y)$ displays $\condindstate{X}{Y}{W}$, then $\psi(y,w,x)$ displays $\condindstate{Y}{X}{W}$.
		\item Decomposition: if $\psi(x_1,x_2,w,y)$ displays $\condindstate{(X_1\otimes X_2)}{Y}{W}$, then the marginal $\psi(x_1,w,y)$ displays $\condindstate{X_1}{Y}{W}$.
		\item\label{cstate} Contraction: if $\psi(x,w,z,y)$ displays $\condindstate{X}{Y}{(W\otimes Z)}$ and the marginal $\psi(x,w,z)$ displays $\condindstate{X}{Z}{W}$, then $\psi$ also displays $\condindstate{X}{(Z\otimes Y)}{W}$.
	\end{enumerate}
	Moreover if $\C$ has conditionals, then also:
	\begin{enumerate}[resume]
		\item\label{wustate} Weak union: if $\psi(x_1,x_2,w,y)$ displays $\condindstate{(X_1\otimes X_2)}{Y}{W}$, then it also displays $\condindstate{X_1}{Y}{(W\otimes X_2)}$.
	\end{enumerate}
\end{lem}

\begin{rem}
On first look, these statements may sound more complicated than the usual semigraphoid properties, since they state explicitly \emph{which morphism} displays the conditional independence, where this morphism is taken to be the given joint distribution of all variables under consideration. But this is implicitly contained also in the standard way of phrasing the semigraphoid properties, and our formulation merely makes this more explicit.
\end{rem}

\begin{proof}
	The first two statements are trivial. Concerning \ref{cstate}, we have the assumptions
	\[
		\input{contraction_proof1.tikz}
	\]
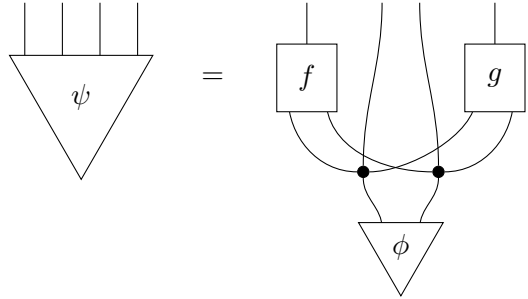
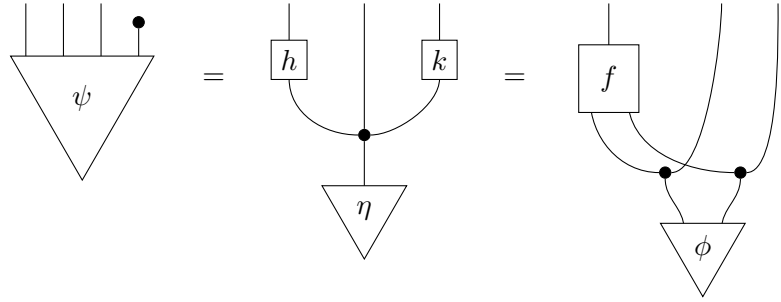
	\[
		\input{contraction_proof2.tikz}
	\]
	for suitable additional morphisms as indicated, where the final equation follows by marginalizing the first one. Substituting this third equation into the first one gives
	\[
		\input{contraction_proof3.tikz}
	\]
	which proves the claim.
	
	We now show \ref{wustate}. By the assumption that $\psi(x_1,x_2,w,y)$ displays $\condindstate{(X_1\otimes X_2)}{Y}{W}$, we have $\phi : I \to W$ and $f : W \to X_1 \otimes X_2$ and $g : W \to Y$ such that
	\[
		\input{weak_union_proof1.tikz}
	\]
	By conditioning $f$ on the second output, we can also write this in the form
	\[
		\input{weak_union_proof2.tikz}
	\]
	where now the dashed boxes indicate the desired conditional independence factorization.
\end{proof}

\begin{prob}
	Beyond the semigraphoid properties of \Cref{cistateprops}, does our general notion of conditional independence also satisfy the additional implication between conditional independences in probability found by Studen\'y~\cite{studeny}?
\end{prob}

An interesting compositionality property of joint distributions is the possibility of constructing \emph{conditional products}, as first defined by Dawid and Studen{\'y}~\cite{DS} and generalized in a sheaf-theory-like categorical framework in~\cite{FF}. We now show how to construct conditional products in any Markov category which has conditionals, and we prove that they satisfy the axioms of Dawid and Studen{\'y}. To this end, we work with the conceptually clearer but equivalent framework of \emph{gleaves on a lattice}~\cite[Section~4.1]{FF}.

\begin{defn}
	\label{condprod_defn}
	Suppose that $\C$ has conditionals. Then for any $\psi : I \to X \otimes W$ and $\phi : I \to W \otimes Y$ with equal marginals on $W$,
	\[
		\input{equal_marginals.tikz}
	\]
	the \emph{conditional product} of $\psi$ and $\phi$ is defined by
	\beq
		\label{condprod_def}
		\input{conditional_product2.tikz}
	\eeq
\end{defn}

The calculation
\beq
	\label{condprob_alt}
	\input{condprod_well_defined.tikz}
\eeq
shows that the conditional product does not depend on the particular choice of conditional $\psi_{|W}$, and similarly for $\phi_{|W}$ (compare~\Cref{cond_unique}). Hence $\psi \otimes_W \phi$ is well-defined.

The conditional product forms part of our story on conditional independence as follows.

\begin{prop}
	If $\C$ has conditionals, then the conditional product \cref{condprod_def} is the unique morphism $I \to X \otimes W \otimes Y$ which marginalizes to both $\psi$ and $\phi$ and displays $\condindstate{X}{Y}{W}$.
\end{prop}

\begin{proof}
	The fact that $\psi \otimes_W \phi$ displays $\condindstate{X}{Y}{W}$ is immediate from the definition. The fact that it marginalizes to $\psi$ and $\phi$ is clear by the definition of conditionals (\Cref{defn_has_conds}). Conversely, if $\eta : I \to X \otimes W \otimes Y$ has these properties, then consider a factorization~\cref{cistateeq}. By \Cref{cistate_cond}, the $f$ and $g$ making up that factorization must be the corresponding conditionals of the marginals of $\eta$, i.e.~of $\psi$ and $\phi$. Hence $\eta$ is exactly the conditional product.
\end{proof}

\newcommand{\Pfin}[1]{\mathcal{P}_{\mathrm{fin}}(#1)}

Still assuming that $\C$ has conditionals, we now explain how marginalization and conditional products define a gleaf~\cite{FF} on $\Pfin{O}$, the lattice of finite subsets of any set of objects $O \subseteq \mathrm{Obj}(\C)$. To every $\mathcal{A} \in \Pfin{O}$ we assign the set of morphisms\footnote{Strictly speaking, we need to make a choice of order for this finite tensor product, but the particular choice is irrelevant. One way to do this is to fix any total order on $O$ and order the objects in $\mathcal{A}$ accordingly.} $I \to \bigotimes_{X \in \mathcal{A}} X$. For any inclusion $\mathcal{A} \subseteq \mathcal{B}$, marginalization defines a restriction map
\[
	\C \! \left(I,\,\bigotimes_{X \in \mathcal{B}} X\right) \longrightarrow \C \! \left( I,\, \bigotimes_{X \in \mathcal{A}} X \right),
\]
and it is easy to see that this results in a presheaf on $\Pfin{O}$. The above conditional products equip this presheaf with a gluing operation: in \cite[Definition~4.1.1]{FF}, property (a) holds because a conditional product marginalizes to the two given original morphisms. Condition (c) holds because of the commutativity of conditioning with marginalization described in~\Cref{cond_marginal}. Condition (b) amounts to the claim that for $\psi : I \to X \otimes W$ and $\phi : I \to W \otimes Y \otimes Z$, the conditional product can be computed in two steps,
\beq
	\label{gleaf_ass}
	\psi \otimes_W \phi = (\psi \otimes_W \phi(wy)) \otimes_{W \otimes Y} \phi,
\eeq
where we have again used \Cref{formal_notation}. The proof of this property is more involved: using the alternative representation for the conditional product given on the right of~\Cref{condprob_alt}, we compute
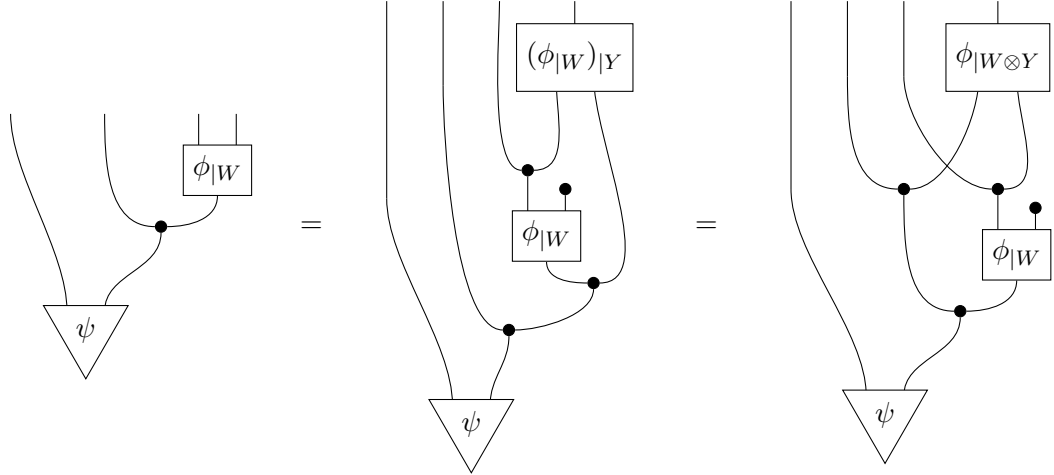
\[
	\input{gleaf_proof.tikz}
\]
where the second step consists of reordering wires and applying\footnote{This step shows that we really need the existence of conditionals in the sense of \Cref{defn_has_conds}, whereas conditional distributions alone in the sense of \Cref{defn_cond_dists} are not enough for our argument to go through.} \Cref{conditional_iterate}. Since this is exactly the desired~\cref{gleaf_ass}, we conclude that we indeed have a gleaf, or equivalently conditional products satisfying the axioms of Dawid and Studen{\'y}. We can therefore apply the theory of gleaves as developed in~\cite{FF}. One consequence is that conditional products enjoy an associativity property~\cite[Proposition~4.3.4]{FF}.

\begin{rem}
	\label{coupling_cat}
	This associativity property of conditional products implies in particular that marginalizing a conditional product in the form~\cref{condprod_def} over $W$ defines the composition in a symmetric monoidal category which we call the \emph{category of couplings}, where:
	\begin{itemize}
		\item Objects are pairs $(X,\pi_X)$ where $\pi_X : I \to X$.
		\item The hom-set between $(X,\pi_X)$ and $(X,\pi_Y)$ is the set of all $\psi : I \to X \otimes Y$ with marginals $\psi(x) = \pi_X$ and $\psi(y) = \pi_Y$.
		\item Composition is given by the conditional product~\cref{condprod_def} marginalized over the middle output.
		\item The identity morphism of $(X,\pi_X)$ is given by $\cop_X\circ\pi_X$.
	\end{itemize}
	It is also straightforward to verify directly that this data defines a category, and that it has a symmetric monoidal structure defined in the obvious way. However, we believe that the compositional structure of conditional products is more accurately captured by the concept of gleaf as described above. The main reason is that this keeps track of the joint distribution of all variables involved, not requiring one to apply marginalization during composition; it also unifies the sequential product and the monoidal product that we will consider more explicitly in \Cref{probstoch} into one kind of gluing.
\end{rem}

This ends our discussion of conditional independence via \Cref{cistatedef}. However, just as with the notion of conditional in \Cref{defn_has_conds}, we expect the concept of conditional independence to be meaningful also for morphisms whose domain is not the monoidal unit. In order to lead up to this, we briefly consider a notion of conditional independence $\condindproc{X}{Y}{A}$ for morphisms $A \to X \otimes Y$ closely related to the original definition of Cho and Jacobs~\cite[Definition~6.6]{cho_jacobs}; and also to the earlier conditions $CI_{2L}$ and $CI_{2R}$ of Coecke and Spekkens~\cite[Section~5.1]{CS}, which takes place in a different categorical setup. 

\begin{lem}
	Given a morphism $f : A \to X \otimes Y$, the following are equivalent:
	\begin{enumerate}
		\item\label{factor} $f(x,y|a) = f(x|a) f(y|a)$ in terms of \Cref{formal_notation}, meaning that
		\[
			\input{cond_ind.tikz}
		\]
		\item\label{cic} There are $g : A \to X$ and $h : A \to Y$ such that 
		\[
			\input{cond_ind2.tikz}
		\]
	\end{enumerate}
\end{lem}

\begin{proof}
	While one direction is trivial, the other is by straightforward marginalization.
\end{proof}

\begin{defn}
	\label{ciprocdef}
	Let $\C$ be a Markov category. If a morphism $f : A \to X \otimes Y$ satisfies these equivalent conditions, then we say that $f$ \emph{displays the conditional independence} $\condindproc{X}{Y}{A}$.
\end{defn}

Here, we use the double bar ``$||$'' notation in order to indicate independence conditional on an input, to be distinguished from independence conditional on an output as in \Cref{cistatedef}.

We now have at least the simplest two of the semigraphoid properties, in terms of \Cref{formal_notation}.

\begin{lem}
	\label{ciprocprops}
	The following hold:
	\begin{enumerate}
		\item Symmetry: if $f(x,y|a)$ displays $\condindproc{X}{Y}{A}$, then $f(y,x|a)$ displays $\condindproc{Y}{X}{A}$.
		\item Decomposition: if $f(x_1,x_2,y|a)$ displays $\condindproc{(X_1\otimes X_2)}{Y}{A}$, then the marginal $f(x_1,y|a)$ displays $\condindproc{X_1}{Y}{A}$.
	\end{enumerate}
\end{lem}

\begin{proof}
	Straightforward.	
\end{proof}

This notion of conditional independence also interacts well with positivity. The following property generalizes the well-known fact that if a probability measure on a product $\sigma$-algebra has a deterministic marginal, then it is a product measure.

\begin{prop}
	\label{marginal_det}
	Let $\C$ be positive (\Cref{positive_defn}). Then if a morphism $f : A \to X \otimes Y$ is such that its marginal $A \to X$ is deterministic, then $f$ displays $\condindproc{X}{Y}{A}$.	
\end{prop}

\begin{proof}
	We apply the positivity property of \Cref{positive_defn} with $f$ in place of $f$ and the marginalization morphism 	
	
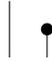
\[
		\begin{tikzpicture}
	\begin{pgfonlayer}{nodelayer}
		\node [style=none] (0) at (-0.5, 1) {};
		\node [style=bn] (1) at (0.5, 0.25) {};
		\node [style=none] (2) at (0.5, -0.5) {};
		\node [style=none] (3) at (-0.5, -0.5) {};
	\end{pgfonlayer}
	\begin{pgfonlayer}{edgelayer}
		\draw (3.center) to (0.center);
		\draw (1) to (2.center);
	\end{pgfonlayer}
\end{tikzpicture}

	\]
	in place of $g$. Then the resulting~\eqref{positive} instantiates to
	\[
		\input{pos_marg_det.tikz}
	\]
	Marginalizing over the second output now gives the desired factorization.
\end{proof}

\begin{cor}
	Let $\C$ be positive. Then for a morphism $f : A \to X \otimes Y$, the following are equivalent:
	\begin{enumerate}
		\item\label{fdet} $f$ is deterministic.
		\item\label{fmargdet} Both marginals of $f$ are deterministic. 
	\end{enumerate}
\end{cor}

\begin{proof}
	Assuming \ref{fdet}, the claim \ref{fmargdet} is clear since the marginalization map itself is deterministic as well. The implication from~\ref{fmargdet} to~\ref{fdet} is a straightforward string diagram computation,
	\[
		\input{product_det.tikz}
	\]
	where the first and third step are by \Cref{marginal_det} and the second one is by assumption.
\end{proof}

We now introduce our more involved combined notion of conditional independence, which has not been considered before.

\begin{defn}
	\label{cigendef}
	A morphism $f : A \to X \otimes W \otimes Y$ \emph{displays the conditional independence} $\condindgen{X}{Y}{W}{A}$ if there are $g : A \to W$, $h : A \otimes W \to X$ and $k : W\otimes A \to Y$ such that
	\[
		\input{cond_ind_gen.tikz}
	\]
\end{defn}

It is straightforward to see that this specializes to \Cref{ciprocdef} for $W = I$ and to \Cref{cistatedef} for $A = I$.

We now have some more complicated but general rules for reasoning about this kind of conditional independence. Again these are versions of the usual semigraphoid properties, phrased in terms of \Cref{formal_notation}.

\begin{prop}
	\label{cigenprops}
	The following hold:
	\begin{enumerate}
		\item Symmetry: if $f(x,w,y|a)$ displays $\condindgen{X}{Y}{W}{A}$, then $f(y,w,x|a)$ displays $\condindgen{Y}{X}{W}{A}$.
		\item Decomposition: if $f(x_1,x_2,w,y|a)$ displays $\condindgen{(X_1\otimes X_2)}{Y}{W}{A}$, then the marginal $f(x_1,w,y|a)$ displays $\condindgen{X_1}{Y}{W}{A}$.
		\item\label{cgen} Contraction: if $f(x,w,z,y|a)$ displays $\condindgen{X}{Y}{(W\otimes Z)}{A}$ and the marginal $f(x,w,z|a)$ displays $\condindgen{X}{Z}{W}{A}$, then $f$ also displays $\condindgen{X}{(Z\otimes Y)}{W}{A}$.
	\end{enumerate}
	Moreover if $\C$ has conditionals, then also:
	\begin{enumerate}[resume]
		\item\label{wugen} Weak union: if $f(x_1,x_2,w,y|a)$ displays $\condindgen{(X_1\otimes X_2)}{Y}{W}{A}$, then it also displays $\condindgen{X_1}{Y}{(W\otimes X_2)}{A}$.
	\end{enumerate}
\end{prop}

\begin{proof}
	Again the first two implications are trivial. The proof of~\ref{cgen} and~\ref{wugen} are simple but somewhat tedious generalizations of the corresponding proofs for \Cref{cistateprops}, which we leave to the reader.
\end{proof}

Moreover, conditional independence enjoys certain compositionality properties, formulated as follows in terms of \Cref{formal_notation}.

\begin{prop}
	\label{condind_compprop}
	The following hold:
	\begin{enumerate}
		\item\label{condind_times} If a morphism $f(x_1,w_1,y_1|a_1)$ displays $\condindgen{X_1}{Y_1}{W_1}{A_1}$ and $g(x_2,w_2,y_2|a_2)$ displays $\condindgen{X_2}{Y_2}{W_2}{A_2}$, then $(f\otimes g)(x_1,x_2,w_1,w_2,y_1,y_2|a_1,a_2)$ displays $\condindgen{(X_1\otimes X_2)}{(Y_1\otimes Y_2)}{(W_1\otimes W_2)}{(A_1\otimes A_2)}$.
		\item\label{condind_compose} If a morphism $f(x_1,v,y_1|a)$ displays $\condindgen{X_1}{Y_1}{V}{A}$ and $g(x_2,w,y_2|v)$ displays $\condindgen{X_2}{Y_2}{W}{V}$, then also the composite morphism
			\[
				\input{cond_ind_compose.tikz}
			\]
			displays $\condindgen{(X_1\otimes X_2)}{(Y_2\otimes Y_1)}{(V\otimes W)}{A}$.
		\item\label{post_process} If a morphism $f(x_1,w,y_2|a)$ displays $\condindgen{X_1}{Y_1}{W}{A}$ and $g : A \otimes X_1 \to X_2$ and $h : Y_1 \otimes A \to Y_2$ are arbitrary, then also the composite morphism
			\[
				\input{cond_ind_compose2.tikz}
			\]
			displays $\condindgen{X_2}{Y_2}{W}{A}$.
		\item\label{condind_comp} If $\C$ has conditionals and a morphism $f(x,v,w,y|a)$ displays $\condindgen{X}{Y}{(V\otimes W)}{A}$, then there is a conditional $f_{|V}(x,w,y|v,a)$ which displays $\condindgen{X}{Y}{W}{(V\otimes A)}$.
	\end{enumerate}
\end{prop}

While property~\ref{condind_compose} is somewhat reminiscent of contraction, property~\ref{condind_comp} is similar to weak union.

\begin{proof}
	\ref{condind_times} and~\ref{post_process} are straightforward. For \ref{condind_compose}, note that the assumptions imply that the composite morphism has the form
	\[
		\input{condind_compose_proof.tikz}
	\]
	where we leave the relevant morphisms unnamed. The dashed boxes indicate the claimed conditional independence decomposition.

	Concerning \ref{condind_comp}, the assumption implies that $f$ has the form
	\[
		\input{condind_comp_ass.tikz}
	\]
	We construct a new morphism which involves a conditional $g_{|V}$, 
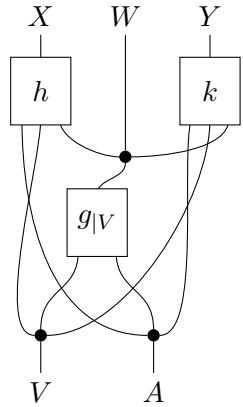
	\[
		\input{condind_comp.tikz}	
	\]
	which clearly displays $\condindgen{X}{Y}{W}{(V\otimes A)}$. Furthermore, we verify that this morphism is a conditional $f_{|V}$,
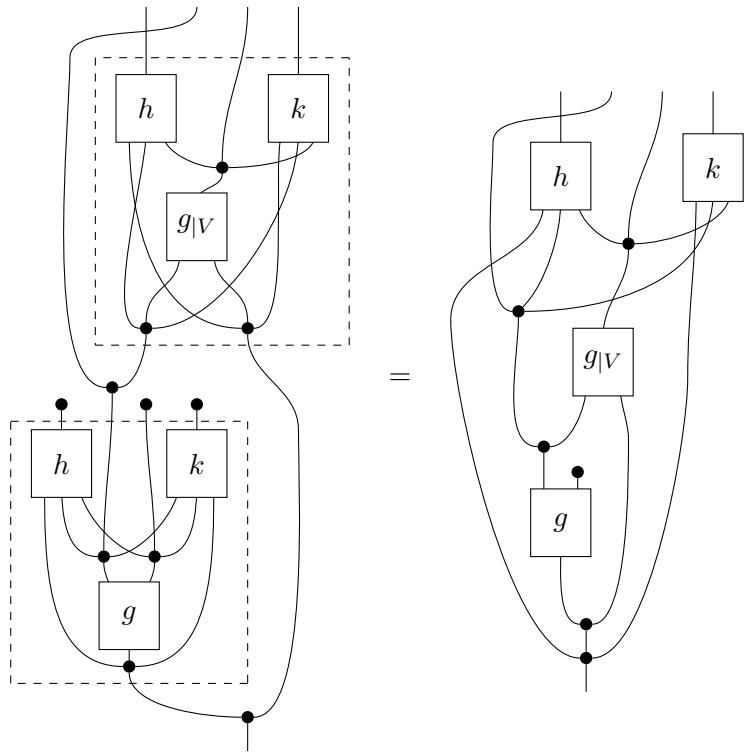
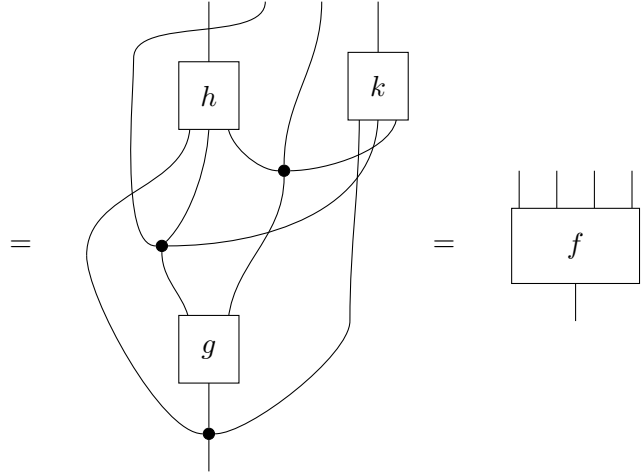
	\[
		\input{complicated_conditional1.tikz}
	\]
	\[
		\input{complicated_conditional2.tikz}
	\]
	which finishes the proof.
\end{proof}

There is one more notion of conditional independence, which will be of relevance to us in \Cref{suff}. We do not yet know how it is related to the one of \Cref{cigendef}.

\begin{defn}
	\label{def_cimarkov}
	Given a morphism $f : A \to X \otimes Y$, we say that $f$ \emph{displays the conditional independence} $\condindmarkov{A}{Y}{X}$ if it is of the form
	\[
		\input{condind_markov.tikz}
	\]
\end{defn}

Here, the diagram on the right represents a Markov process, from $A$ to $X$ to $Y$. This notion of conditional independence appears in the definition of sufficient statistic (\Cref{suff_defn}).

For $A = I$, one would expect that this conditional independence always holds. This is indeed the case as soon as $\C$ has conditional distributions, since the conditional independence equation for $\condindmarkov{I}{Y}{X}$ amounts to precisely the existence of a conditional for the given morphism $I \to X \otimes Y$.

Again we have versions of the semigraphoid properties in terms of \Cref{formal_notation}, except for symmetry which no longer can even be stated.

\begin{prop}
	\label{cimarkovprops}
	The following hold:
	\begin{enumerate}
		\item\label{dmarkov} Right decomposition: if $f(x,y_1,y_2|a)$ displays $\condindmarkov{A}{(Y_1\otimes Y_2)}{X}$, then the marginal $f(x,y_1|a)$ displays $\condindmarkov{A}{Y_1}{X}$.
		\item\label{cmarkov} Contraction: if $f(x,w,y|a)$ displays $\condindmarkov{A}{Y}{(W\otimes X)}$ and the marginal $f(x,w|a)$ displays $\condindmarkov{A}{W}{X}$, then $f$ also displays $\condindmarkov{A}{(W\otimes Y)}{X}$.
	\end{enumerate}
	Moreover if $\C$ has conditionals, then also:
	\begin{enumerate}[resume]
		\item\label{wumarkov} Weak union: if $f(x,w,y|a)$ displays $\condindmarkov{A}{(W\otimes Y)}{X}$, then it also displays $\condindmarkov{A}{Y}{(W\otimes X)}$.
	\end{enumerate}
\end{prop}

\begin{proof}
	The proof of~\ref{dmarkov} is straightforward. For \ref{cmarkov}, we compute
	\[
		\input{condind_markov_contract.tikz}
	\]
	where the assumptions are used in the first two steps. For \ref{wumarkov}, we get
	\[
		\input{condind_markov_weakunion.tikz}
	\]
\end{proof}

There may be one more notion of conditional independence along the lines of~\cite[Properties $CI_{1L}$, $CI_{1R}$]{CS} and~\cite[Proposition~6.9(3)]{cho_jacobs}, where a morphism $f : A \otimes B \to X$ displays this conditional independence of $A$ and $X$ given $B$ if it can be written in a form which simply discards $A$,
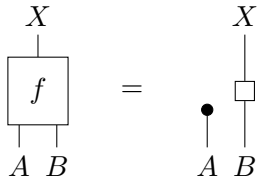
\[
	\input{cic_constant.tikz}
\]
We will not study this notion of conditional independence further in this paper.

\begin{rem}
	\label{disconnect}
	As has been pointed out to us by Rob Spekkens, all of our definitions of conditional independence have a central feature in common,  namely that upon removing all wires from the diagram which refer to objects on which we condition, the diagram disconnects into parts which then become conditionally independent. We hope that this observation may lead to a more unified and more systematic perspective on conditional independence in our setup as a theory of ``conditional connectedness'', providing a general understanding of conditional independence as a form of separation similar to separation in undirected graphical models~\cite{lauritzen_book}, simpler and more intuitive than the complicated \emph{$d$-separation} criterion for conditional independence in Bayesian networks~\cite{pearl}.
\end{rem}

Let us end this section with some comments on the relation to other existing axiomatic systems for conditional independence relations. We have already done so throughout this section for the most well-known such system: the semigraphoid axioms. However, others with a more explicitly categorical flavour have been proposed.

\subsection*{Franz's notion of independent morphisms.} Franz~\cite{franz}, and later Gerhold, Lachs and Sch\"urmann~\cite{GLS}, have developed a categorical framework for independence which unifies and generalizes several notions of independence that had previously been considered, including ordinary stochastic independence in probability theory and free independence in free probability. However, the basic setup is quite different, and involves working with a category of random variables, meaning measure-preserving maps between probability spaces, instead of a category of Markov kernels. In our setup, involving a Markov category $\C$ whose morphisms we think of as Markov kernels between measurable spaces, we can \emph{define} a probability space as a pair $(X,\psi)$ consisting of an object $X \in \C$ and a morphism $\psi : I \to X$. A \emph{random variable} or \emph{measure-preserving map} from $(X,\psi)$ to $(Y,\phi)$ is then a deterministic morphism $f : X \to Y$ such that $f\circ \psi = \phi$. In other words, we consider the comma category $I/\C_\det$ on $\C$, where $\C_\det \subseteq \C$ is the subcategory of deterministic morphisms, as our category of probability spaces\footnote{See \Cref{prob_cat} for a more refined approach.}. This category is then again semicartesian, so that Franz's notion of independence applies. Instantiating his definition then gives the following, for $(X,\psi)$ a probability space in our sense: deterministic morphisms $f : X \to Y$ and $g : X \to Z$ are independent if and only if there is deterministic $h : X \to Y \otimes Z$ satisfying
\[
	\input{franz_independent.tikz}
\]
Due to the cartesianness of the monoidal subcategory of deterministic morphisms (\Cref{rem_det_cartesian}), we can conclude that $h$ is necessarily equal to the tupling $(f,g) : X \to Y \otimes Z$, the unique deterministic morphism which projects to $f$ and $g$. Therefore $f$ and $g$ are independent in Franz's sense if and only if $(f,g) \circ \psi$ displays the independence $Y \perp Z$ in the sense of \Cref{independence}. Note that $(f,g) \circ \psi : I \to Y \otimes Z$ can be interpreted as the joint distribution of the random variables $f$ and $g$.

In summary, Franz's definition keeps track explicitly of the sample space on which random variables are defined (the $X$ in our discussion above), while our notion of conditional independence of \Cref{cistatedef} does not. And there is no need to, as independence is a property only of the \emph{joint distribution} of two random variables, regardless of the underlying sample space. Apart from this difference, Franz's notion of independence, when instantiated in categories of probability spaces which are constructed from Markov categories, is essentially equivalent to the one of \Cref{independence}. 

\subsection*{Simpson's axioms for conditional independence.} Simpson~\cite{simpson} has proposed rather general axioms for independence and conditional independence in a categorical framework. When applied to probabilistic independence, one needs to work with categories of probability spaces, just as in Franz's approach. The key difference is that Simpson does not define what (conditional) independence means in such a category, but rather proposes \emph{axioms} which any reasonable notion of conditional independence should satisfy. We have not yet undertaken a more detailed comparison with Simpson's approach.

\section{Almost surely}
\label{sec_as}

A notion that appears in almost every theorem of probability is \emph{almost surely}. In the context of Markov categories, an elegant axiomatization of almost sure equality has been given by Cho and Jacobs~\cite[Definition~5.1]{cho_jacobs}. Here, we generalize their definition and prove further properties. We also put forward the idea that many other concepts can be relativized with respect to almost surely, and illustrate this by introducing an ``almost surely'' variant of deterministic morphisms (\Cref{defn_det}). This concept then will be shown to interact well with almost sure equality. As an application, we construct the category of probability spaces internal to any Markov category $\C$. 

The following definition specializes to one given by Cho and Jacobs~\cite[Definition~5.1]{cho_jacobs} in the case $\Theta = I$. Our choice of notation $p : \Theta \to X$ anticipates statistical models (\Cref{def_stat_model}), indicating that almost sure equality is particularly relevant in a context where $\Theta$ is a parameter space for a statistical model living on a sample space $X$.

\begin{defn}
	\label{defn_as}
	Given any morphism $p : \Theta \to X$, we say that $f,g : X \to Y$ are \emph{$p$-a.s.~equal} and write $f =_{p\as} g$ if
	\[
		\input{asequal.tikz}
	\]
\end{defn}

\begin{ex}
	\label{finstoch_as}
	In $\FinStoch$, we have $f =_{p\as} g$ if and only if $f(y|x) = g(y|x)$ for all $x \in X$ for which there is $a \in A$ with $p(x|a) > 0$.
\end{ex}

\begin{ex}
	\label{stoch_as}
	In $\Stoch$, we have $f =_{p\as} g$ if and only if
	\beq
		\label{stoch_as_eq}
		\int_S f(T|x) \, p(dx|a) = \int_S g(T|x) \, p(dx|a)
	\eeq
	for all $a\in A$ and $S \in \Sigma_X$ and $T \in \Sigma_Y$. This is equivalent to the functions $f(T|-), g(T|-) : X \to [0,1]$ being almost surely equal with respect to the probability measure $p(-|a)$ on $X$ for every parameter value $a$, generalizing~\cite[Proposition~5.3]{cho_jacobs}.
	
	If $f$ and $g$ are the images in $\Stoch$ of morphisms in $\Meas$, which by abuse of notation we also denote by $f, g : X \to Y$, then this condition is equivalent to
	\[
		p(S \cap f^{-1}(T)|a) = p(S \cap g^{-1}(T)|a)
	\]
	for all $a \in A$, $S \in \Sigma_X$ and $T \in \Sigma_Y$. Using the additivity of the probability measure $p(-|a)$, this is readily seen to be equivalent to the condition that
	\[
		p\!\left(f^{-1}(T) \setminus g^{-1}(T)\right|a) = 0 = p\!\left(g^{-1}(T) \setminus f^{-1}(T)\right|a)
	\]
	for all $a \in A$ and $T \in \Sigma_Y$. Expressed in terms of symmetric difference $\triangle$, we can also write this as
	\[
		p\!\left(f^{-1}(T) \,\triangle\, g^{-1}(T)\right|a) = 0,
	\]
	which is a known condition\footnote{It had been introduced to us by Robert Furber and seems to be due to Edgar~\cite[Section~1]{edgar_sections}.} equivalent to almost sure equality in the standard sense for countably generated $\sigma$-algebras~\cite[343F]{fremlin3}, but known to be different in general: Fremlin~\cite[343J]{fremlin3} presents an intuitive example of a measurable and measure-preserving map $f$ from a measure space $(X,\Sigma_X,p)$ to itself such that although $f(x) \neq x$ \emph{for all} $x \in X$ and $\Sigma_X$ separates the points of $X$, the map $f$ is $p$-almost surely equal to $\id_X$ in our sense.
	
	Generally speaking, our working hypothesis is that~\cref{stoch_as_eq} is the ``correct'' notion of almost sure equality in $\Stoch$. This hypothesis will be supported by our results of the following sections, where we find that using~\cref{stoch_as_eq} as almost sure equality recovers standard concepts, such as bounded completeness (\Cref{stoch_complete}).
\end{ex}

Returning to the general theory, here is a basic property of a.s.~equality which follows directly from the definition.

\begin{lem}
	\label{as_compose}
	Let $p : \Theta \to X$. If $f =_{p\as} g$ for $f, g : X \to Y$ and $h : Y \to Z$, then also $hf =_{p\as} hg$.
\end{lem}

Furthermore, an occasionally useful fact is that if $g =_{p\as} f$, then the defining equation still holds with more than one copy operation taking place before the application of $f$ or $g$. This is a special case of the following:

\begin{lem}
	\label{lem_asequal_general}
	Let $p : \Theta \to X$. If $f_1, f_2 : X \to Y$ and $g_1, g_2 : X \to Z$ are such that $f_1 =_{p\as} f_2$ and $g_1 =_{p\as} g_2$, then also
	\[
		\input{asequal_general.tikz}
	\]
\end{lem}

\begin{proof}
	Straightforward application of the assumptions and the associativity and commutativity of the comultiplication.	
\end{proof}

A similar (and easy to prove) observation is that if $f_1, f_2 : X \to Y$ and $g_1, g_2 : W \to Z$ are such that $f_1 =_{p\as} f_2$ for $p : \Theta \to X$ and $g_1 =_{q\as} g_2$ for $q : \Phi \to W$, then also $(f_1 \otimes g_1) =_{p\otimes q} (f_2 \otimes g_2)$.

\begin{rem}
	\label{aseq_compose}
	Almost sure equality can obviously always be weakened by precomposition: if $f =_{p\as} g$, and if the codomain of $q$ matches the domain of $p$, then also $f =_{pq\as} g$. The causality axiom (\Cref{causal_defn}) provides conversely a way to decompose a relativizing morphism: marginalizing its conclusion over the middle output shows that if $f =_{pq\as} g$, whenever this makes sense, then also $fp =_{q\as} gp$. But regardless of whether the causality axiom holds or not, $f =_{p\as} g$ clearly implies $fp = gp$.
\end{rem}

We now consider conditionals in the sense of \Cref{defn_has_conds}, and note that they are almost surely unique~\cite[p.~22]{cho_jacobs}.

\begin{prop}
	\label{cond_unique}
	Given a morphism $\psi : I \to X \otimes Y$, the conditional $\psi_{|X} : X \to Y$ is unique $\psi(x)$-a.s.~(if it exists).
\end{prop}

Here, $\psi(x)$ is the marginal morphism $I \to X$ of the original $\psi(x,y)$ in terms of \Cref{formal_notation}. The uniqueness statement means that any two such conditionals are $\psi(x)$-a.s.~equal.

\begin{proof}
	This is an immediate consequence of the defining equation
	\[
		\input{cond_state.tikz}
	\]
\end{proof}

A related point is that almost sure equality lets us define categories of probability spaces in $\C$. A \emph{probability space} in a Markov category $\C$ is a pair $(X,\psi)$ where $X \in \C$ is an object and $\psi : I \to X$ is a morphism playing the role of a distribution on $X$. If $(X,\psi)$ and $(Y,\phi)$ are probability spaces, then a morphism $f : X \to Y$ is \emph{measure-preserving} if $\phi = f\psi$. In the following definition, we impose the causality condition as an apparently necessary hypothesis for our well-definedness argument below.

\newcommand{\Probstoch}[1]{\mathsf{ProbStoch}(#1)}
\newcommand{\Prob}[1]{\mathsf{Prob}(#1)}

\begin{defn}
	Suppose that $\C$ is causal (\Cref{causal_defn}). Then the \emph{category of probability spaces and Markov kernels} $\Probstoch{\C}$ has as objects the probability spaces $(X,\psi)$ in $\C$, and as morphisms $(X,\psi) \to (Y,\phi)$ those maps $f : X \to Y$ which satisfy $f \psi = \phi$, modulo $\psi$-a.s.\ equality, with composition inherited from $\C$.
\end{defn}

In categories like $\FinStoch$ and $\Stoch$, the condition $f \psi = \phi$ says exactly that $f$ should be measure-preserving in the usual sense.

\begin{prop}
	\label{probstoch}
	Suppose that $\C$ is causal.
	\begin{enumerate}
		\item\label{probstoch_welldef} Composition in $\Probstoch{\C}$ is well-defined, and $\Probstoch{\C}$ inherits the symmetric monoidal structure from $\C$.
		\item\label{couplings_iso} If $\C$ has conditionals, then $\Probstoch{\C}$ is isomorphic to the category of couplings described in \Cref{coupling_cat}.
	\end{enumerate}
\end{prop}

For \ref{couplings_iso}, recall that the existence of conditionals implies causality (\Cref{cond_causal}), so that only the existence of conditionals needs to be verified when applying this to a given Markov category $\C$. We do not know whether the well-definedness of composition would still hold without causality.

\begin{proof}
	\begin{enumerate}
		\item Concerning the well-definedness, let $f,f' : X \to Y$ and $g : Y \to Z$ be representative morphisms between probability spaces $(X,\psi)$, $(Y,\phi)$ and $(Z,\pi)$, and suppose that $f =_{\psi\as} f'$. Then we get $gf =_{\psi\as} gf'$ by \Cref{as_compose}. Well-definedness in $g$ follows from the causality property of \Cref{causal_defn}: the assumption $g =_{\phi\as} g'$ means that
		\[
			\input{compose_welldefined1.tikz}
		\]
		and therefore applying causality and marginalizing over the middle output results in the desired equation
		\[
			\input{compose_welldefined2.tikz}
		\]
		Hence composition is well-defined.

		The monoidal structure on $\Probstoch{\C}$ is given by
		\[
			(X,\psi)\otimes (Y,\phi) := (X\otimes Y, I \stackrel{\cong}{\to} I \otimes I \stackrel{\psi \otimes \phi}{\longrightarrow} X \otimes Y),
		\]
		and similarly on morphisms. With the structure isomorphisms also inherited from $\C$ in the obvious way, it is straightforward to check that we indeed get a symmetric monoidal category.

	\item We first establish the bijection between the hom-sets: to a given morphism $f : (X,\psi) \to (Y,\phi)$ represented by $f \in \C(X,Y)$, we assign the joint $I \to X \otimes Y$ given by~\cite[Section~7.1]{golubtsov2}
		\[
			\input{coupling.tikz}
		\]
	Clearly $f$ itself is a conditional of this distribution with respect to the first output, and this assignment establishes a bijection between $\psi$-a.s.\ equivalence classes of morphisms $f : X \to Y$ satisfying $\phi = f\psi$ and couplings, i.e.~morphisms $I \to X \otimes Y$ with marginals $\psi$ and $\phi$. Moreover, together with the formula~\cref{condprob_alt} for the conditional product, it is easy to see that this bijection respects composition. \qedhere
	\end{enumerate}
\end{proof}

For example, \ref{probstoch_welldef} allows us to construct $\Probstoch{\Stoch}$, the category of all probability spaces and measure-preserving maps modulo almost sure equality\footnote{Assuming that almost sure equality is defined appropriately, namely as in~\Cref{stoch_as}.}, since we already know $\Stoch$ to be causal (\Cref{stoch_causal}).

\begin{rem}
	\label{dagger}
	If $\C$ has conditionals, we now sketch how Bayesian inversion equips the category $\Probstoch{\C}$ with a symmetric monoidal dagger functor in the sense of dagger categories\footnote{This means that $-^\dag : \C^\op \to \C$ is identity-on-objects, satisfies $(gf)^\dag = f^\dag g^\dag$ and $f^{\dag\dag} = f$ and $\id^\dag = \id$ and $(g\otimes f)^\dag = g^\dag \otimes f^\dag$, and is such that the dagger of any symmetric monoidal structure isomorphism is its inverse. The latter holds in our case because the dagger of any deterministic isomorphism is its inverse, which applies to the structure isomorphisms by \Cref{structure_det}. All of these properties are straightforward to check.}. Since this applies in particular to $\BorelStoch$ by \Cref{borelstoch_cond}, this generalizes a recent result due to Clerc, Danos, Dahlqvist and Garnier, who proved this in explicit measure-theoretic terms for $\BorelStoch$~\cite[Theorem~2.10]{DSDG}. Concretely, if $f : (X,\psi) \to (Y,\phi)$ is a morphism represented by $f : X \to Y$ with $f\psi = \phi$, then we get a morphism $f^\dag : (Y,\phi) \to (X,\psi)$ represented by any $f^\dag : Y \to X$ satisfying the equation
	\beq
		\label{bayes_inv}
		\input{bayesian_inversion.tikz}
	\eeq
	This equation expresses the property that $f^{\dag}$ should be a conditional on $Y$ for the joint $I \to X \otimes Y$ depicted in~\eqref{bayes_inv}. In terms of the characterization as the category of couplings stated in~\Cref{probstoch}\ref{couplings_iso}, this dagger functor is a mere bookkeeping device, swapping couplings $I \to X \otimes Y$ to couplings $I \to Y \otimes X$ by composition with $\swap_{X,Y} : X \otimes Y \to Y \otimes X$. This characterization can be used to give a simple proof of the fact that the defining properties of a symmetric monoidal dagger functor are satisfied.

	In conclusion, for any Markov category $\C$ which has conditionals, Bayesian inversion is implemented by a symmetric monoidal dagger functor on $\Probstoch{\C}$.
\end{rem}

Returning to the theory of almost surely, we believe that other concepts of probability and statistics can and should be relativized similarly with respect to almost surely in a similar way as \Cref{defn_as} does for equality, by starting with the original definition and composing it with $p$ followed by $\cop_X : X \to X \otimes X$. For concreteness and future use, we show what this amounts to in the case of \Cref{defn_det}.

\begin{defn}
	\label{defn_as_det}
	Given $p : \Theta \to X$, we say that $f : X \to Y$ is \emph{$p$-a.s.~deterministic} if
	\[
		\input{asdeterministic.tikz}
	\]
\end{defn}

In other words, we now postulate the determinism equation only with respect to $p$-a.s.~equality. We typically expect such relativizations to respect a.s.~equality, which is indeed what happens in this case:

\begin{lem}
	\label{lem_asequal_asdeterministic}
	Let $p : \Theta \to X$. If $f,g : X \to Y$ are such that $g =_{p\as} f$ and $f$ is $p$-a.s.~deterministic, then so is $g$.
\end{lem}

\begin{proof}
	We use the assumptions to compute
	\[
		\input{asequal_asdeterministic.tikz}
	\]
	where the last step is by \Cref{lem_asequal_general}.
\end{proof}

\begin{ex}
	\label{stoch_asdet}
	Generalizing \Cref{stoch_det}, in $\Stoch$ a morphism $f$ is $p$-a.s.~deterministic if for all $\theta \in \Theta$ we have
	\[
		\int_{x \in R} f(S|x) \, f(T|x) \, p(dx|\theta) = \int_{x \in R} f(S \cap T) \, p(dx|\theta)
	\]
	for all $R \in \Sigma_X$ and $S,T \in \Sigma_Y$. The universal quantification over $R$ implies that the integrands $x \mapsto f(S|x) f(T|x)$ and $x \mapsto f(S \cap T|x)$ are $p(-|\theta)$-almost surely equal for all $\theta$. By reasoning analogous to \Cref{stoch_det}, we thereby obtain that $f$ is $p$-a.s.~deterministic if and only if for every $S \in \Sigma_X$ we have $f(S|x) \in \{0,1\}$ at least $p(-|\theta)$-almost surely for every $\theta$.
\end{ex}

\begin{rem}
	\label{prob_cat}
	Assuming that $\C$ is causal, we can restrict the category $\Probstoch{\C}$ considered in \Cref{probstoch} to the almost surely deterministic morphisms, resulting in a symmetric monoidal subcategory $\Prob{\C} \subseteq \Probstoch{\C}$ which we call the \emph{category of probability spaces in $\C$}. More precisely, a morphism $(X,\psi) \to (Y,\phi)$ represented by $f \in \C(X,Y)$ is in $\Probstoch{\C}$ if and only if $f$ is $\psi$-a.s.\ deterministic. \Cref{lem_asequal_asdeterministic} shows that this property does not depend on the choice of representative. The fact that $\Prob{\C}$ is closed under composition again follows from the causality property: the assumption that $g$ is $f\psi$-a.s.\ deterministic gives
	\[
		\input{prob_proof1.tikz}
	\]	
	thanks to causality, and the assumption that $f$ is $\psi$-a.s.\ deterministic allows us to write
	\[
		\input{prob_proof2.tikz}
	\]
	so that both steps together prove the claim.

	Finally, it is easy to see that $\Prob{\C}$ is indeed closed under monoidal product and contains the structure isomorphisms by \Cref{structure_det}, so that it indeed forms a symmetric monoidal subcategory of $\Probstoch{\C}$ if $\C$ is causal.
\end{rem}

We now consider some relations between almost sure determinism and the positivity axiom discussed in \Cref{positive_defn} and after. This following observation is in the spirit of \Cref{deterministic_composition}.

\begin{prop}
	\label{retract_asdet}
	Suppose that $f : X \to Y$ and $g : Y \to Z$ in $\C$ are such that $gf$ is deterministic. Suppose also that
	\begin{enumerate}
		\item\label{detass} $f$ is deterministic, or 
		\item\label{posass} $\C$ is positive (\Cref{positive_defn}).
	\end{enumerate}
	Then $g$ is $f$-a.s.~deterministic.
\end{prop}

In particular, if $gf = \id$ and $\C$ is positive, then $g$ is $f$-a.s.~deterministic. \Cref{finstoch_pm} shows that $gf = \id$ alone is not enough to imply that $g$ is $f$-a.s.~deterministic.

\begin{proof}
	\begin{enumerate}
		\item This is a straightforward string diagram computation,
			\[
				\input{asdeterministic_proof3.tikz}
			\]
		\item We use the positivity equation of \Cref{cond_positive} repeatedly together with associativity,
			\[
				\input{asdeterministic_proof1.tikz}
			\]
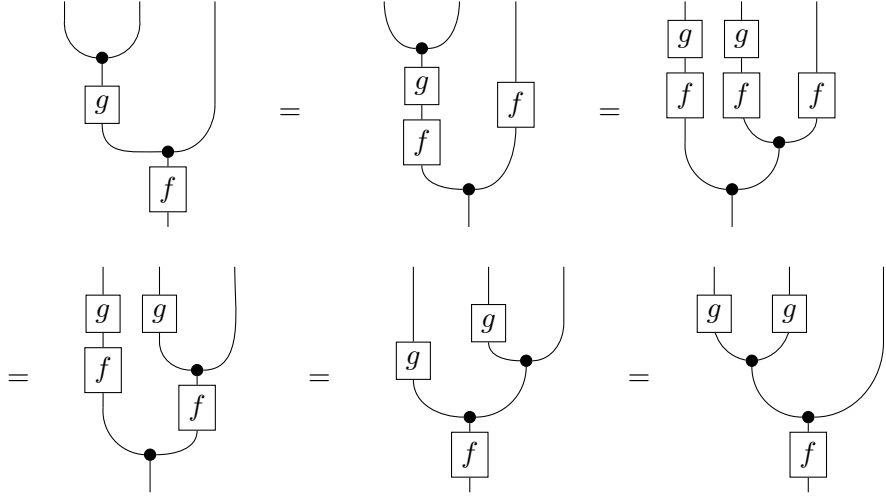
			\[
				\input{asdeterministic_proof2.tikz}
			\]
	\end{enumerate}
\end{proof}

In the proof of \Cref{thm_FN}, we will need a stronger positivity condition than \Cref{positive_defn} obtained by relativizing almost surely.

\begin{defn}
	\label{spositive_defn}
	We say that $\C$ is \emph{strictly positive} if whenever morphisms
	\[
		\begin{tikzcd}
			\Theta \ar["p"]{r} & X \ar["f"]{r} & Y \ar["g"]{r} & Z
		\end{tikzcd}
	\]
	are such that $gf$ is $p$-a.s.~deterministic, then also
	\beq
		\label{spositive}
		\input{spositive.tikz}
	\eeq
\end{defn}

\begin{rem}
	\label{gfp_det}
	Carrying an extra $p : \Theta \to X$ around in the proof of \Cref{retract_asdet} shows the following more general statement. Suppose that $gf$ is $p$-a.s.~deterministic, and that also
	\begin{enumerate}
		\item $f$ is $p$-a.s.~deterministic, or
		\item $\C$ is strictly positive.
	\end{enumerate}
	Then $g$ is $fp$-a.s.~deterministic.
\end{rem}

As a generalization of \Cref{cond_positive}, we have:

\begin{lem}
	If $\C$ has conditionals, then $\C$ is strictly positive.
\end{lem}

\begin{proof}
	Use the same $s$ as in~\eqref{positive_cond_eq}, and then do the analogous computation as in the proof of \Cref{cond_positive} while carrying the additional $p$ along.
\end{proof}

In particular, $\FinStoch$ and $\Gauss$ are strictly positive. For $\Stoch$, we need a separate argument generalizing \Cref{stoch_positive}.

\begin{ex}
	\label{stoch_spositive}
	$\Stoch$ is strictly positive, as we show now. By the characterization of $p$-a.s.~deterministic morphisms of \Cref{stoch_asdet}, the assumption is that
	\beq
		\label{stoch_spositive_ass}
		\int_{y\in Y} g(S|y) \, f(dy|x)  \: \in \{0,1\}
	\eeq
	for all $S \in \Sigma_X$ and $p(-|\theta)$-almost all $x \in X$ and all $\theta \in \Theta$. What we want to prove is that the equation
	\[
		\int_{y \in T} g(S|y) \, f(dy|x) = f(T|y) \, \int_{y \in Y} g(S|y) \, f(dy|x)
	\]
	holds for all $S \in \Sigma_X$ and $T \in \Sigma_Y$ for $p(-|\theta)$-almost all $x \in X$ and all $\theta$. This is by the same argument as in \Cref{stoch_positive}.
\end{ex}

As a final note, let us propose one more definition which can be nicely formulated in terms of almost sure equality, although we will not study it any further in this paper.

\newcommand{\supp}[1]{\mathrm{supp}(#1)}

\begin{defn}
	\label{defn_support}
	Let $\C$ be a Markov category. For any morphism $p : A \to X$, the \emph{support} $\supp{p}$ is any object which represents the functor
	\[
		\C(X,-)/=_{p\as} \; : \; \C \longrightarrow \Set
	\]
	mapping every object $Y$ to the set of equivalence classes of morphisms $X \to Y$ modulo $p\as$ equality.
\end{defn}

More concretely, this means that $\supp{p}$ comes equipped with a \emph{universal morphism} $i : \supp{p} \to X$ having the property that for every $Y \in \C$, composition with $i$ induces a bijection between morphisms $\supp{p} \to Y$ and $p$-a.s.~equivalence classes of morphisms $X \to Y$.

\begin{rem}
	\label{supp_usual}
	While this is somewhat speculative at the present stage, our hope is that this notion of support will specialize to the usual notion of support of a Radon probability measure on a locally compact space, or more generally that of a $\tau$-smooth probability measure on any topological space, when applied to morphisms $I \to X$ in a suitable Markov category with topological spaces as objects. So far we have not considered such a category yet.
\end{rem}

As is usually the case for an object specified in terms of a universal property, the support $\supp{p}$ may or may not exist for a given morphism $p$ in a given category $\C$.

\begin{qstn}
	Suppose that $\supp{p}$ exists. Then does $p : A \to X$ necessarily factor through $i : \supp{p} \to X$? Is $i$ necessarily deterministic?
\end{qstn}

As easily follows from the following example, the answer to both questions is positive in the case of $\FinStoch$.

\begin{ex}
	\label{finstoch_supp}
	In $\FinStoch$, the support of any morphism $p$ exists, and is exactly the support in the usual sense,
	\[
		\supp{p} = \{ x \in X \mid \exists a : p(x|a) > 0 \},
	\]
	where the universal morphism $i : \supp{p} \to X$ is the inclusion map. The universality is easy to prove as a consequence of \Cref{finstoch_as}.
\end{ex}

Finally, we remark that the existence of supports can be useful for simplifying arguments which involve relativization with respect to almost surely.

\begin{prop}
	If every morphism in $\C$ has a support and $\C$ is positive (\Cref{positive_defn}), then $\C$ is also strictly positive (\Cref{spositive_defn}).
\end{prop}

\begin{proof}
	Straightforward.
\end{proof}

\section{Sufficient statistics and the Fisher--Neyman factorization theorem}
\label{suff}

We finally move on to the treatment of statistics within our formalism, comprising this section and the following two. Informally, a \emph{statistic} is a function of a random sample intended to capture the essential information contained in this sample. Perhaps the most common example is the \emph{sample mean}: tossing a coin $n \gg 1$ times results in a probability distribution over $2^n$ outcomes, or in our setting a morphism $I \to \{\mathrm{heads},\mathrm{tails}\}^n$. The sample mean is the map $\{\mathrm{heads},\mathrm{tails}\}^n \to [0,1]$ which maps a sequence of outcomes to the relative frequency of (say) heads, which is the essential information contained in the sequence, assuming that the individual throws are independent. This independence assumption is a \emph{statistical model}: a hypothesis about a distribution which comes in the form of a collection of distributions on the sample space $X$, in this case $\{\mathrm{heads},\mathrm{tails}\}^n$, parametrized by a parameter space $\Theta$, in this case the space of probability distributions describing a single toss, which we can identify with $[0,1]$. In other words, a statistical model is a \emph{Markov kernel} $\Theta \to X$~\cite{mccullagh}.

\begin{defn}
	\label{def_stat_model}
	Given a sample space $X\in\C$, a \emph{statistical model} with values in $X$ consists of another object $\Theta\in\C$ and a morphism $p : \Theta\to X$.
\end{defn}

Put yet differently, a statistical model for $X$ is an object in the slice category over $X$. This notion of statistical model is the simplest and most naive version, where one considers only a single fixed sample space $X$; for a more refined notion and detailed discussion, see~\cite{mccullagh} as well as the \emph{repetitive structures} of~\cite{lauritzen_book2}.\footnote{It seems plausible to us that McCullagh's and Lauritzen's refined definitions fit perfectly well into our formalism as well by working within a suitable categories of diagrams (\Cref{sec_diagrams}), but we have not yet worked out the details of this.}

While our definition of statistical model keeps track of an explicit parametrization of all distributions via a parameter space $\Theta$, one can alternatively omit this parametrization and consider categories of unparametrized statistical models, which are mere sets of probability distributions~\cite[Definition~2.20]{jhht}. We view this difference as relatively minor, since a set of probability distributions can always be taken to be parametrized by itself.

\begin{defn}
	\label{def_statistic}
	A \emph{statistic} for a statistical model $p : \Theta \to X$ is a deterministic morphism $s : X\to V$ to some other object $V$.
\end{defn}

This is essentially the traditional definition (modulo subtleties like those of \Cref{stoch_det}, depending on which category $\C$ we are working in).

By combining a statistic $t : X\to V$ with a statistical model $p : \Theta\to X$, we obtain a morphism $\Theta\to V\otimes X$ which describes the joint distribution over sample space and statistic as a function of the model parameter,
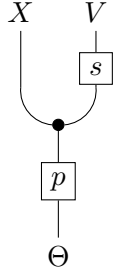
\[
	\input{statistic.tikz}
\]
By construction, this morphism displays the conditional independence $\condindmarkov{\Theta}{V}{X}$ (\Cref{def_cimarkov}).

An important property which a statistic may or may not have is its \emph{sufficiency}. Intuitively, a sufficient statistic for a statistical model is as informative than the full sample.

\begin{defn}
	\label{suff_defn}
	A statistic $s : X \to V$ for a statistical model $p : \Theta \to X$ is \emph{sufficient} if there is $\alpha : V \to X$ such that
	\beq
		\label{suffdef}
		\input{sufficient_statistic.tikz}
	\eeq
\end{defn}

Formulating sufficiency in terms of the existence of such a morphism is certainly not new and has been considered previously in more concrete ways, e.g.~in Dawid's work on statistical operations~\cite[Example~III(a)]{dawid_operations}. Also Golubtsov~\cite[Section~7]{golubtsov1} already seems to have envisioned the possibility of giving such an abstract categorical definition of sufficiency.

We also say that $\alpha$ is a \emph{sufficiency witness} for $s$ if the above condition holds; another way to phrase it is to say that the morphism on the left in~\cref{suffdef} displays the conditional independence $\condindmarkov{\Theta}{X}{V}$ in addition to $\condindmarkov{\Theta}{V}{X}$, which is a well-known way to formulate sufficiency. Also note the similarity of this condition with Bayesian inversion~\cref{bayes_inv}. Remarkably, the condition automatically implies that $\alpha$ \emph{almost surely splits} $s$ in the sense that $s \alpha =_{sp\as} \id_V$,
\beq
	\label{suffsplit}
	\input{sufficient_statistic_split.tikz}
\eeq
where the first equation is by sufficiency and the second one because $s$ is deterministic. So intuitively, sufficiency means that given a value of the statistic $s$, applying $\alpha$ plays the role of making an \emph{inference} over the corresponding outcomes in $X$ which give rise to that value of $s$, and in such a way that this distribution coincides with the conditional distribution over outcomes given any parameter value. While the morphism $\alpha$ does play a role similar to that of a conditional, our definition may also be of interest in categories which do not have conditionals, such as $\Stoch$ (\Cref{stoch_cond_dist}).

Sufficient statistics enjoy the following compositionality property, inspired by a similar compositionality result of Jost, L\^e, Luu and Tran in their more specific setting~\cite[Theorem~2.28(1)]{jhht}.

\begin{lem}
	Let $s : X \to V$ and $t : V \to W$ be statistics. If $s$ is sufficient for $p$ and $t$ is sufficient for $sp$, then also the composite $ts$ is sufficient for $p$.
\end{lem}

\begin{proof}
	Suppose that $\alpha : V \to X$ is a sufficiency witness for $s$ with respect to $p$ and $\beta : W \to V$ is a sufficiency witness for $t$ with respect to the pushforward model $sp$. Then we show that $\alpha\beta$ is a sufficiency witness for $ts$ with respect to $p$,
	\[
		\input{compositional_sufficiency.tikz}
	\]
	using the respective assumed instances of~\cref{suffdef}.
\end{proof}

Here is an abstract characterization of sufficient statistics, which seems to be a close relative of or version of the Fisher--Neyman factorization theorem.

\begin{thm}
	\label{thm_FN}
	Suppose that $\C$ is strictly positive (\Cref{spositive_defn}). Then a statistic $s : X \to V$ is sufficient for a statistical model $p : \Theta \to X$ if and only if there is $\alpha : V \to X$ such that $\alpha s p = p$ and $s \alpha =_{sp\as} \id_V$.
\end{thm}

More precisely, $\alpha : V \to X$ is a sufficiency witness if and only if it satisfies these alternative conditions. Since we have already assembled all the relevant ingredients, the proof is now essentially trivial.

\begin{proof}
	The ``only if'' direction is trivial by marginalization of~\cref{suffdef} and by~\cref{suffsplit}. For the ``if'' direction, we show that $\alpha s p = p$ and $s \alpha =_{sp\as} \id$ together imply~\cref{suffdef}. Indeed,
	\[
		\input{fisher_neyman_proof.tikz}
	\]
	where both steps are by assumption, and the second one also uses the strict positivity property of~\eqref{spositive}.
\end{proof}

We now illustrate the relation to the Fisher--Neyman factorization theorem in the case of our basic running example.

\begin{ex}
	\label{finstoch_fn}
	Consider $\FinStoch$, where statistics are just functions between finite sets. Now consider the equation $\alpha s p = p$ from \Cref{thm_FN}, which becomes
	\[
		p(x|\theta) \: = \: \sum_{v,x'} \alpha(x|v) \, s(v|x') \, p(x'|\theta).
	\]
	The condition $s \alpha =_{sp\as} \id_V$ implies that the sum over $v$ has only one contributing term, namely $v = s(x)$, where $s$ now also denotes the underlying function of finite sets by abuse of notation. Therefore
	\[
		p(x|\theta) \: = \: \alpha(x|s(x)) \sum_{x' \: : \: s(x') = s(x)} p(x'|\theta).
	\]
	This is the Fisher--Neyman factorization into one term which only depends on $x$, and one term which depends on $\theta$ and where the dependence on $x$ is only through $s(x)$. Conversely, given a Fisher--Neyman factorization
	\[
		p(x|\theta) = g_\theta(s(x)) \, h(x),
	\]
	we construct $\alpha$ satisfying the required properties. We can assume without loss of generality that for every $v \in V$ there is $x \in s^{-1}(v)$ and $\theta \in \Theta$ with $p(x|\theta) > 0$, since on all other $v$ we can define $\alpha(-|v)$ arbitrarily. Then we get that the denominator in the first case of the expression
	\[
		\alpha(x|v) := \begin{cases} \left(\sum_{x' \in s^{-1}(v)} h(x')\right)^{-1} h(x) & \textrm{if } s(x) = v, \\ 0 & \textrm{otherwise} \end{cases}
	\]
	is nonzero, making it well-defined. Now $s \alpha =_{sp\as} \id$ holds by construction, and plugging in the above expressions gives
	\begin{align*}
		(\alpha s p)(x|\theta) & = \sum_{x' \in X} \alpha(x|s(x')) \, p(x'|\theta) = \sum_{x' \in s^{-1}(x)} \alpha(x|s(x)) \, p(x'|\theta) = p(x|\theta),
	\end{align*}
	as was to be shown.
\end{ex}

We leave the investigation of further instances of \Cref{thm_FN} to future work, and in particular the problem of determining whether our result also specializes to the standard Fisher--Neyman factorization theorem in its general measure-theoretic form~\cite[Corollary~1]{HS}.

\section{Complete morphisms, ancillary statistics, and Basu's theorem}
\label{sec_basu}

\emph{Basu's theorem}~\cite[Theorem~2]{basu} is a classical result on the independence of a \emph{complete} sufficient statistic from any ancillary statistic. Let us start by discussing the more mysterious notion of completeness, introduced by Lehmann and Scheff\'e in 1950~\cite[Section~3]{LS}. We first explain how it secretly is a concept that applies to general morphisms in every Markov category.

Recall from \Cref{aseq_compose} that if $g =_{f\as} h$ for $f : X \to Y$ and $g,h : Y \to Z$, then also $gf = hf$. Now $f$ is defined to be complete if the converse holds as well.

\begin{defn}
	\label{defn_complete}
	A morphism $f : X \to Y$ is \emph{complete} if $gf = hf$ implies $g =_{f\as} h$.
\end{defn}

This definition is in the spirit of Dawid's \cite[Definition~4.1]{dawid_operations}. Note that at our level of abstraction, we have no way to distinguish between \emph{complete} and \emph{bounded complete}. We will prove in \Cref{stoch_complete} that in $\Stoch$, our definition actually reproduces the latter.

\begin{rem}
	\label{complete_rem}
	The main difference between our definition and the standard one is terminological: while standardly, completeness is a property of a \emph{statistic} $s$ relative to a \emph{statistical model} $p$, we consider it rather as a property of the composite morphism $sp$, since it only depends on the latter. In other words, completeness is best viewed a property of Markov kernels, or more generally morphisms in a Markov category.
\end{rem}

\begin{ex}
	\label{finstoch_complete}
	In $\FinStoch$, a morphism $f : X \to Y$ is complete if and only if the affine hull of $\im{f}$ is a face of the probability simplex over $Y$. For example, an $f : I \to Y$ is complete if and only if it is a Dirac measure, or equivalently if and only if it is deterministic.

	To see that this characterization holds, it is most instructive to present a geometrical argument. We think of morphisms $g,h : Y \to Z$ as affine maps from the probability simplex over $Y$ to the probability simplex over $Z$. Now recall from \Cref{finstoch_supp} that $g =_{f\as} h$ if and only if the restrictions of $g$ and $h$ to $\supp{f}$ coincide, where $\supp{f} \subseteq Y$ is the set of all $y \in Y$ for which there is $x \in X$ with $f(y|x) > 0$. We can identify the support $\supp{f}$ with the probability simplex over $\supp{f}$, as included in the probability simplex over $Y$ as a face. Therefore $g =_{f\as} h$ if and only if $g$ and $h$ coincide on the smallest face containing $\im{f}$.

	Similarly, $gf = hf$ holds if and only if $g$ and $h$ coincide on $\im{f}$, or equivalently on the affine hull of $\im{f}$. We can now show that $gf = hf$ implies $g =_{f\as} h$ if and only if the affine hull of $\im{f}$ is a face. By the above, the ``if'' direction is clear; for the ``only if'' part, we prove the contrapositive: if the affine hull of $\im{f}$ is not a face, then there are $g$ and $h$ with $gf = hf$ but not $g =_{f\as} h$. We prove the equivalent statement that for any affine subset\footnote{An affine subset of the simplex is given by the intersection of the simplex with an affine subspace, or equivalently by suitably many hyperplanes.} of the probability simplex over $Y$, there is $Z$ together with morphisms $g,h : Y \to Z$ which are equal on the affine subset, but not anywhere outside of it. Since the probability simplex over $X$ is bounded, we can ignore the requirement that $g$ and $h$ must again land in a probability simplex, and instead find two such affine maps into some Euclidean space; composing with a suitable invertible contraction and translation will map them back into a finite-dimensional simplex. Now the solution is obvious: for example, take $g$ to be the identity map and $h$ the orthogonal projection onto the affine subspace (with respect to any choice of inner product).
\end{ex}

\begin{ex}
	\label{stoch_complete}
	Now consider completeness in $\Stoch$. Here, we claim that $f : X \to Y$ is complete if and only if for every bounded measurable $\beta : Y \to \R$,
	\beq
		\label{bounded_complete}
		\int_{y \in Y} \beta(y) \, f(dy|x) = 0 \quad \forall x \qquad \Longrightarrow \qquad \beta(y) = 0 \quad f(-|x)\as \: \forall x.
	\eeq
	This is essentially the standard definition of \emph{bounded complete}.

	To prove the equivalence, suppose first that $f$ is complete in our sense. Then for given $\beta$ satisfying the vanishing integral assumption, suppose $\beta : Y \to [0,1]$ without loss of generality by rescaling. Then with the power set $\sigma$-algebra on $\{0,1\}$, we have a morphism in $\Stoch$ given by
	\[
		g : Y \to \{0,1\}, \qquad g(\{1\}|y) := \beta(y),
	\]
	where all other values of $g$ are clearly uniquely specified and the measurability condition on $g$ holds as $\beta$ is measurable. We now have the composite morphism
	\[
		gf : X \to \{0,1\} , \qquad (gf)(\{1\}|x) = \int_{y\in Y} \beta(y) \, f(dy|x) = 0,
	\]
	which coincides with $hf : X \to \{0,1\}$ for $h : Y \to \{0,1\}$ the deterministic constant zero morphism, $h(\{1\}|y) = 0$ for all $y$. Therefore the completeness assumption on $f$ implies $g =_{f\as} h$. Using \cref{stoch_as_eq} with $T = \{1\}$ only, this is equivalent to
	\[
		\int_{y\in S} \beta(y) \, f(dy|x) = \int_{y\in S} 0 \, f(dy|x) = 0
	\]
	for all $S \in \Sigma_Y$. This implies that $\beta(y)$ indeed vanishes $f(-|x)$-almost surely for every $x$, proving~\cref{bounded_complete}.

	Conversely, suppose that $f$ is bounded complete in the sense of \cref{bounded_complete}. For given $g, h : Y \to Z$ with $gf = hf$, we will prove that
	\beq
		\label{equal_integrals}
		\int_{y\in S} g(T|y) \, f(dy|x) = \int_{y\in S} h(T|y) \, f(dy|x)
	\eeq
	for all $x \in X$ and $S \in \Sigma_Y$ and $T \in \Sigma_Z$, which is what needs to be shown by~\cref{stoch_as_eq}. For fixed $S$ and $T$, consider the function
	\[
		\beta : Y \longrightarrow [0,1], \qquad y \longmapsto g(T|y) - h(T|y) ,
	\]
	which is bounded measurable by construction. Now the integral of $\beta$ with respect to every $f(-|dx)$ vanishes by the assumption that $gf = hf$, which amounts exactly to equality of the integrals for $S = Y$. Therefore $\beta$ itself vanishes $f(-|x)$-almost surely for every $x$ thanks to the assumption~\cref{bounded_complete}. This immediately implies~\cref{equal_integrals}.
\end{ex}

We have therefore established that our abstract notion of completeness instantiates concretely to the one which is most relevant for measure-theoretic probability. Returning to the general theory, a simple first observation about completeness is as follows.

\begin{lem}
	Every deterministic morphism is complete.
\end{lem}

\begin{proof}
	Assuming $gf = hf$ for deterministic $f$, we compute
	\[
		\input{complete_deterministic_proof.tikz}
	\]
	which means exactly $g =_{f\as} h$.
\end{proof}

We have not yet studied other properties of complete morphisms, such as their compositionality. A simple use of completeness is to provide a simplified criterion for sufficiency.

\begin{prop}
	Suppose that $\C$ is strictly positive (\Cref{spositive_defn}). Let $p : \Theta \to X$ be a statistical model and $s : X \to V$ a statistic such that $sp$ is complete. Then $s$ is sufficient if and only if there is $\alpha : V \to X$ such that $\alpha s p = p$.
\end{prop}

\begin{proof}
	Applying \Cref{thm_FN}, the condition $s\alpha =_{sp\as} \id$ follows from $s\alpha sp = sp$ by the completeness assumption.
\end{proof}

The other missing ingredient for Basu's theorem is the notion of \emph{ancillary} statistic, which formalizes the idea that a statistic carries no information at all about the parameters of a statistical model.

\begin{defn}
	\label{defn_ancillary}
	A statistic $s : X \to V$ for a statistical model $p : \Theta \to X$ is \emph{ancillary} if $sp$ displays $\condindmarkov{\Theta}{T}{I}$, which means that
	\[
		\input{ancillary.tikz}
	\]
	for suitable $\psi : I \to V$.
\end{defn}

Here is our version of Basu's theorem:

\begin{thm}
	\label{thm_basu}
	Suppose that $p : \Theta \to X$ is a statistical model, and that $s : X \to V$ and $a : X \to W$ are statistics. If $s$ is sufficient, the composite morphism $sp$ is complete, and $a$ is ancillary, then the morphism
	\beq
		\label{basu_morphism}
		\input{basu_conclusion.tikz}
	\eeq
	displays $\condindproc{V}{W}{\Theta}$.
\end{thm}

The morphism~\cref{basu_morphism} represents the joint distribution of the statistics $s$ and $a$, given the model parameter.

\begin{proof}
	Let $\beta : V \to X$ be a splitting of $s$ which witnesses its sufficiency. Then we have
	\[
		\input{basu_proof1.tikz}
	\]
	for suitable $\psi : I \to U$, where the first step is by the sufficiency of $s$ and the second since $a$ is ancillary. Using that $sp$ is complete, we obtain
	\[
		\input{basu_proof2.tikz}
	\]
	where the first step is by sufficiency of $s$ and the second by completeness of $sp$ and the previous equation. This factorization implies in particular the desired conditional independence $\condindproc{V}{W}{\Theta}$.
\end{proof}

\begin{rem}
	Thanks to the characterization of completeness as bounded completeness in $\Stoch$ (\Cref{stoch_complete}), \Cref{thm_basu} reproduces essentially the standard version of Basu's theorem in the case of $\Stoch$. 
\end{rem}

\section{Minimal sufficient statistics and Bahadur's theorem}
\label{sec_bahadur}

While sufficient statistics are useful in that they contain maximal information about the parameters, in practice one also wants a statistic to \emph{discard as much irrelevant information} as possible. This is formalized by the notion of \emph{minimal} sufficient statistic, which this section develops within our framework. We will prove Bahadur's theorem, a classical result giving a sufficient criterion for the minimality of a sufficient statistic in terms of completeness.

Leading up to the definition of minimality, we first investigate how different statistics may determine each other:

\begin{defn}
	\label{stat_preorder}
	Let $s : X \to V$ and $t : X \to W$ be statistics for a statistical model $p : \Theta \to X$. Then we write $t \leq s$ if there is $c : V \to W$ such that $t =_{p\as} cs$.
\end{defn}

Such a $c$ must be $sp$-a.s.~deterministic by \Cref{gfp_det}, since $s$ is deterministic and $cs$ is $p$-a.s.~deterministic by \Cref{lem_asequal_asdeterministic}.

It is easy to see that the $\leq$-relation on statistics is reflexive and transitive. In this way, the collection of statistics for a given statistical model forms a preordered set\footnote{Or, depending on how large the category $\C$ is, a proper class equipped with a preorder. We will ignore these size issues here, for which Grothendieck universes provide the standard working solution.}, with the more informative statistics further up. This preorder is reminiscent of the standard theory of subobjects in category theory~\cite[Section~1.3]{MM}, but with the direction of arrows reversed. The idea of defining an informativeness ordering of this type in categorical terms is certainly not new and appears e.g.~also in the work of Golubtsov~\cite[Section~5]{golubtsov2}, who has used it to develop aspects of decision theory in purely categorical terms~\cite[Sections~6 and~7]{golubtsov2}.

We now restrict to the set of sufficient statistics with the induced preorder. Recall that a \emph{least element} $s$ in a preordered set is an element such that $s \leq r$ for all other elements $r$.

\begin{defn}
	\label{min_suff}
	A statistic $s : X \to T$ is \emph{minimal sufficient} if it is a least element in the preordered set of sufficient statistics for the given model $p : \Theta \to X$.
\end{defn}

Since any two least elements are ordered in both directions, a minimal sufficient statistic is essentially unique: any two minimal sufficient statistics can be computed from each other in the sense of \Cref{stat_preorder}.

Here is our version of Bahadur's theorem, showing that under mild assumptions, a complete sufficient statistic is guaranteed to be minimal.

\begin{thm}
	Suppose that $\C$ is strictly positive. Let $p : \Theta \to X$ be a statistical model which has a minimal sufficient statistic. If $s : X \to V$ is a complete sufficient statistic for $p$, then $s$ is minimal sufficient.
\end{thm}

\begin{proof}
	If there is a minimal sufficient statistic, then $s$ must lie above it. By \Cref{stat_preorder}, this means that there is some $c : V \to W$ such that $cs$ is minimal. In order to show that $s$ itself is minimal, we therefore only need to show that $s \leq cs$. Since $cs$ is still sufficient, we have a sufficiency witness $\beta : W \to X$ for $cs$, satisfying $cs \beta =_{csp\as} \id$ and $p = \beta csp$ by \Cref{thm_FN}.
	
	Since $s$ is sufficient as well, we have a witness $\alpha : V \to X$ with $s \alpha =_{sp\as} \id$ and $p = \alpha s p$. Hence $\alpha s p = \beta c s p$, and completeness of $s p$ now implies $\alpha =_{sp\as} \beta c$, and therefore $s \beta c =_{sp\as} \id$ by postcomposing with $s$. We now claim that $s =_{p\as} s \beta c s$. Indeed,
	\[
		\input{bahadur_proof2.tikz}
	\]
	using sufficiency of $s$ in the first and third step. This proves that $s \leq cs$, as was to be shown.
\end{proof}

We end with two possible avenues for further research.

\begin{prob}
	Under which conditions on $\C$ does every statistical model in $\C$ have a minimal sufficient statistic?
\end{prob}

We suspect that it is possible to derive the existence of minimal sufficient statistics from the assumption that $\C$ has all colimits and that these colimits interact suitably nicely with the Markov category structure, in a precise way which still needs to be determined. For then it may be possible to obtain a minimal sufficient statistic by taking the colimit of the diagram consisting of all sufficient statistics out of $X$ for a given statistical model $p : \Theta \to X$. (A suitable co-well-poweredness assumption may be relevant for this as well.)

\begin{rem}
	Following Lauritzen~\cite{lauritzen_book2}, another interesting avenue may be to consider the \emph{inverse problem} of finding all statistical models $p : \Theta \to X$ for which a given deterministic morphism $s : X \to V$ is a sufficient statistic.
\end{rem}

\newpage

\newgeometry{left=3cm,right=3cm,bottom=4.1cm,top=4.1cm}

\bibliographystyle{plain}
\bibliography{markov_categories}

\end{document}